\documentclass[a4paper,11pt]{article}
 \usepackage[plainpages=false]{hyperref}
 \usepackage{amsfonts,amsmath,amssymb,amsthm}
 \usepackage{latexsym,lscape,rawfonts,mathrsfs}

 \usepackage[dvips]{color}
 \usepackage{multicol}

 \usepackage{natbib}

 \usepackage[all]{xy}
 \usepackage{eufrak}
 \usepackage{makeidx}
 \usepackage{graphicx,psfrag}

 \usepackage{array,tabularx}

 \usepackage{setspace}

 \usepackage{appendix}

 \newcommand{\D}[2]{\ensuremath{ \frac{\partial{#1}}{\partial{#2}}}}
 \newcommand{\E}{\ensuremath{\mathbb{E}}}
 
 \newcommand{\N}{\ensuremath{\mathbb{N}}}
 
 \newcommand{\R}{\ensuremath{\mathbb{R}}}

 \newcommand{\KRf}{K\"ahler Ricci flow\;}
 \newcommand{\KRfc}{K\"ahler Ricci flow,\;}
 \newcommand{\KRfd}{K\"ahler Ricci flow.\;}
 \newcommand{\KRF}{K\"ahler Ricci Flow\;}

 \DeclareMathOperator{\Area}{Area}
 \DeclareMathOperator{\Length}{Length}
 \DeclareMathOperator{\Vol}{Vol}
 \DeclareMathOperator{\diam}{diam}

 \newcommand{\snorm}[2]{{ \ensuremath{\left |} #1 \ensuremath{\right |}}_{#2}}
 \newcommand{\sconv}{\ensuremath{ \stackrel{C^\infty}{\longrightarrow}}}

 \makeatletter
 \def\ExtendSymbol#1#2#3#4#5{\ext@arrow 0099{\arrowfill@#1#2#3}{#4}{#5}}
 
 \makeatother

 \definecolor{hao}{rgb}{1,0.5,0}
 \definecolor{miao}{cmyk}{0.5,0,0.2,0.2}
 \definecolor{qiao}{gray}{0.96}

 \newtheorem{claim}{Claim}
 \newtheorem*{clm}{Claim}
 
 \newtheorem{corollary}{Corollary}[section]
 \newtheorem{proposition}{Proposition}[section]
 \newtheorem{lemma}{Lemma}[section]
 \newtheorem*{lm}{Lemma}
 
 \newtheorem{theorem}{Theorem}[section]
 
 \newtheorem{definition}{Definition}[section]

 \newtheorem{remark}{Remark}[section]

 \newtheorem{theoremin}{Theorem}

 \newtheorem{remarkin}{Remark}

 \title{Space of Ricci flows (I)}
 \author{Xiuxiong Chen\footnote{Partially supported by a NSF grant.}\;,  Bing Wang}
 \date{}

\begin{document}
 \maketitle
 \begin{abstract}
    In this paper, we study the moduli spaces of noncollapsed Ricci flow
    solutions with bounded energy and scalar curvature. We show a weak compactness theorem for
    such moduli spaces and apply it to study isoperimetric constant control,
   \KRf and moduli space of gradient shrinking solitons.
 \end{abstract}

 \tableofcontents

 \section{Introduction}

 In \cite{CWa}, we study the Calabi conjecture on Fano manifolds via the K\"ahler Ricci flow.
 This leads inevitably to the study of sequential limits of the K\"ahler Ricci flow solutions over time
 intervals $[t_{i}-1, t_{i}+1]$ as $t_{i}\rightarrow \infty.\;$ The renown Hamilton-Tian conjecture
 (\cite{Tian97}) states that any such sequence must converge in Cheeger-Gromov topology to some K\"ahler Ricci soliton
 with mild singularities and the codimensions of the singularities are at least four.
 In this paper, we begin a systematic study of the moduli space of Ricci flow solutions over time interval
 $[-1,1]$ with some natural geometric constraints such as the volume ratio lower bound,
 the scalar curvature bound as well as certain integral bound on the Riemannian curvature of the evolving metrics.
  We prove a weak compactness theorem  (c.f. Theorem~\ref{theoremin: centerwcpt})
  of this space of Ricci flows under these constraints.   Restricting to Fano surfaces,
   this weak compactness theorem already  verifies the Hamilton-Tian conjecture on Fano surfaces
  (c.f. Theorem~\ref{theoremin: krfcpt}) and  it leads to a Ricci-flow-based proof
  of the Calabi conjecture on Fano surfaces (\cite{CWa}).    In the proofs of this paper, we often use several
  layers of nested contradiction arguments---a technique we learn
  from G. Perelman's seminal work of \cite{Pe1} (c.f.~\cite{KL} for beautiful explanations of this technique). \\

 Let $X^m$ be a closed $m$-dimensional Riemannian manifold,
 $\{(X^m, g(t)), -1 \leq t \leq 1\}$  be a spacetime satisfying the following
  conditions.
  \begin{itemize}
   \item $\D{}{t}g(t) = -Ric_{g(t)} + c_0 g(t)$ where $c_0 $ is a
    constant satisfying $0 \leq c_0 \leq c$.
   \item $\displaystyle \sup_{X \times [-1, 1]}|R|_{g(t)} \leq \sigma$.
   \item  $\displaystyle \frac{\Vol_{g(t)}(B_{g(t)}(x, r))}{r^m} \geq \kappa$
           for all $x \in X, \; t\in [-1, 1], \; r \in (0, 1]$.
   \item  $\int_X |Rm|_{g(t)}^{\frac{m}{2}} d\mu_{g(t)} \leq E$
              for all $t \in [-1,1]$.
 \end{itemize}
 We denote the moduli space of such $\{(X^m, g(t)), -1 \leq t \leq 1\}$ as
 $\mathscr{M}(m, c, \sigma, \kappa, E)$. It is in fact very natural for us to
 consider  this moduli space.  By virtue of Perelman's fundamental estimates
 (c.f.~\cite{SeT}), for \KRf solutions on Fano manifolds, all the above constraints
 hold except the last one.
 If the underlying manifold is a Fano surface,  then
 all these conditions are satisfied a priori.\\

 \begin{theoremin}
    If $\{ (X_i, x_i, g_i(t)) , -1 \leq t \leq 1\} \in \mathscr{M}(m, c, \sigma, \kappa, E)$
 for every $i$, by passing to subsequence, we have
 \begin{align*}
    (X_i, x_i, g_i(0)) \sconv (\hat{X}, \hat{x}, \hat{g})
 \end{align*}
 for some $C^0$-orbifold $\hat{X}$ in Cheeger-Gromov sense.  If $m$
 is odd, then $\hat{X}$ is a smooth manifold.
 \label{theoremin: centerwcpt}
 \end{theoremin}

  This is very reminiscent of the structure of the moduli space of Einstein metrics with similar constraints.
  It is appropriate
  for us to comment now on some historic background in this subject. The moduli space of Einstein
  metrics were well
  understood in early 1990s  through the work of~\cite{An89},~\cite{An90},~\cite{BKN}
  and~\cite{Tian90}.
  They showed that a sequence of
  noncollapsed Einstein manifolds $(X_i^m, x_i, g_i)$ with bounded
  scalar curvature and energy (Riemannian curvature's $L^{\frac{m}{2}}$-norm)
  will subconverge to an Einstein orbifold in
  the Cheeger-Gromov sense.   These works are fundamental and they
  were followed by many other papers in the following two decades. Of course, they
  are natural extension of earlier works of Cheeger~\cite{Che} and Gromov~\cite{Gro}.
    In~\cite{TV1},~\cite{TV2} and~\cite{TV3},  Tian and Viaclovsky studied
    the moduli space of a class of critical metrics, e.g.,
    Bach-flat metrics, constant scalar curvature K\"ahler metrics,
    etc.  One of the major difficulties there is to obtain the local volume ratio upper bound
    without Ricci curvature control.  Tian and Viaclovsky   obtained this
    estimate by blowup arguments. The work of Anderson's~\cite{An05},~\cite{An06}
    also elaborated on this theme.  Using similar idea but more complicated estimates,
    in~\cite{CWe}, the first named author
    and Brian Weber showed that the moduli spaces of extremal metrics
    (in the sense of Calabi, c.f.~\cite{Ca82})
    under similar conditions are also weakly compact.  In~\cite{Ban90}, Bando studied the ``bubble tree" structure of the
   moduli space of Einstein metrics, where every ``bubble" means a
   limit Einstein orbifold.   This study was generalized
   in~\cite{AC} and~\cite{CQY}, where the bubble tree structure of
   more general moduli spaces were studied.\\

 In order to obtain Theorem~\ref{theoremin: centerwcpt}, we need
 two essential estimates: local volume ratio upper bound and $\epsilon$-regularity,
 i.e.,
 \begin{align*}
      \sup_{B_{g(0)}(p, \frac{\rho}{2})} |\nabla^k Rm| \rho^{2+k}
      \leq C_k \{ \int_{B_{g(0)}(p,
      \rho)} |Rm|_{g(0)}^{\frac{m}{2}}\}^{\frac{2}{m}}
 \end{align*}
 whenever $\int_{B_{g(0)}(p, \rho)} |Rm|_{g(0)}^{\frac{m}{2}} < \epsilon$.
 In the case of Einstein metrics, local volume ratio upper bound is an
 application of Bishop's volume comparison theorem.
 $\epsilon$-regularity is implied by elliptic Moser iteration
 for curvature operator since it satisfies a second order elliptic equation.
 If we replace Einstein metrics by some other metrics whose curvature operators
 satisfy some second order elliptic equations, by similar method,
 $\epsilon$-regularity is still manageable.
 Now in our case, the curvature operator only satisfies a parabolic
 equation.  It's hard to use parabolic Moser iteration to prove $\epsilon$-regularity
 property here. The difficulties come from two aspects.
 First, in order to control $|Rm|$'s
 $L^{\infty}$-norm by parabolic Moser iteration,
 one requires more than $L^{\frac{m}{2}}$-norm of $|Rm|$ in a fixed geodesic ball.  One needs either
 $L^{p}$-norm ($p > \frac{m}{2}$) control of $|Rm|$ in a fixed geodesic ball, or $L^{\frac{m}{2}+1}$-norm
 control of $|Rm|$ in a ball of spacetime. Neither of them is a
 natural condition in our setting.   Second, even if we can apply parabolic Moser iteration,
 it's hard to control the local Sobolev constants. Note that the local Sobolev constant is
 determined by the geometry of each unit geodesic ball. However, the unit geodesic balls at different
 time slices are hard to compare since we don't have Ricci curvature bound.

    The way we prove $\epsilon$-regularity is to separate
 it into two properties:
 \begin{itemize}
 \item \textit{energy concentration:}   If $|Rm|_{g(0)}(p)=\rho^{-2} \geq 1$, then
            $\int_{B_{g(0)}(p, \frac{\rho}{2})} |Rm|^{\frac{m}{2}} d\mu >\epsilon$
    for some uniform constant $\epsilon$.
 \item \textit{backward pseudolocality:}   If $\sup_{B(p, \rho)}|Rm|_{g(0)} \leq \rho^{-2} $, then
  $|Rm|_{g(t)}(x)< \delta^{-2}\rho^{-2}$ for some uniform constant
  $\delta$ whenever $d_{g(t)}(x, p)< \delta,  -\delta^2 \rho^2 < t \leq 0$.
 \end{itemize}
 These two estimates together with Shi's estimates on Ricci flow
 solutions imply the $\epsilon$-regularity.      So our object is to
 show that volume ratio upper bound, energy concentration and
 backward pseudolocality hold simultaneously for the metric $g(0)$.
 We'll setup these estimates by bubble analysis  and contradiction
 arguments. Our contradiction arguments follow Perelman's proof of
  canonical neighborhood theorem (Theorem 12.1 of~\cite{Pe1}).
 This means that
  blowup and contradiction arguments will be used repeatedly even in
  one proof.  Since the blowup arguments appear so many times, we find it is convenient
  to study  a unique sequence of Ricci flow solutions blown up from the moduli space
  $\mathscr{M}(m, c, \sigma, \kappa, E)$ first.
  We call such a sequence as a refined sequence. Note that volume
  ratio, energy and backward pseudolocality are all rescaling
  invariant, so it will be sufficient to prove these three
  estimates for a refined sequence. However, a refined sequence is
  strongly related to a sequence of Ricci flat manifolds,  so the proof of such
  estimates is much easier.\\

      Actually, by this arrangement,  the proof of every rescaling
  invariant property of the moduli space  $\mathscr{M}(m, c, \sigma, \kappa, E)$
  can be reduced to the proof of this
  property on a single refined sequence. We believe this method
  is efficient and it works for every geometric flow.\\

   Using Theorem~\ref{theoremin: centerwcpt}, we can build up the
   bubble tree for every sequence in
   $\mathscr{M}(m, c, \sigma, \kappa, E)$. The only difficulty here
   is to understand the ``neck" structure. This difficulty can be
   overcome by an application of a gap lemma (Lemma~\ref{lemma: gap})
   and blowup arguments with delicate choice of blowup scales.
   Note that isoperimetric constant is a rescaling invariant,
   as an application of the bubble tree, we can obtain the following estimate.

 \begin{theoremin}
  If $\{(X, g(t)), -1 \leq t \leq 1\} \in \mathscr{M}(m, c, \sigma, \kappa, E)$ and
  $\diam_{g(0)}(X)<D$, then
  \begin{align*}
    \mathbf{I}((X, g(0)))> \iota
  \end{align*}
  for some positive constant $\iota$ depending on $m, c, \sigma, \kappa, E$ and $D$.
 \label{theoremin: centersob}
 \end{theoremin}

 Actually, this estimate only depends on the bubble tree structure.
 So it can be generalized to estimates on moduli spaces of other
 critical metrics without too much modification.

 \vspace{0.2in}

 Perelman (c.f.~\cite{SeT}) proved that scalar curvature is uniformly bounded
 along \KRf $\{(M^n, g(t)), 0 \leq t < \infty\}$( $n$ is the complex dimension).
 If $\int_M |Rm|^n \omega_t^n$ is uniformly bounded, we can apply
 Theorem~\ref{theoremin: centerwcpt} to obtain weak compactness along \KRfd
 However, if $n \geq 3$ and $\int_M |Rm|^n \omega_t^n$ is uniformly bounded, it is proved in~\cite{RZZ} that Riemannian curvature
 is actually uniformly bounded on \KRfd  Therefore every sequence $(M, g(t_i))$ must subconverge to a smooth manifold $(\hat{M}, \hat{g})$.
 If $n=2$,  $\int_M |Rm|^2 \omega_t^2$ bounded is a natural condition.
 Actually, since $\int_M (|Rm|^2 -R^2) \omega_t^2$ is a topological invariant
 (c.f.~\cite{Ca82}) and scalar curvature is uniformly bounded along the flow,
 we know $\int_M |Rm|^2 \omega_t^2$ is uniformly bounded.
 Applying Theorem~\ref{theoremin: centerwcpt} and Theorem~\ref{theoremin: centersob}
 on 2-dimensional \KRfc we have

 \begin{theoremin}
    Suppose $\{(M, g(t)), 0 \leq t < \infty\}$ is a \KRf solution on Fano
    surface $M$.  Then the isoperimetric constant
    $\mathbf{I}(M, g(t))$ is uniformly bounded from below along this flow.
    Moreover, for every sequence $t_i \to \infty$, by passing to subsequence if necessary,
    we have
    \begin{align*}
       (M, g_i(t))   \sconv (\hat{M}, \hat{g})
    \end{align*}
    where $(\hat{M}, \hat{g})$ is a
    $C^{\infty}$-orbifold satisfying K\"ahler Ricci soliton equation.
 \label{theoremin: krfcpt}
 \end{theoremin}

    In order to improve the limit $C^0$-orbifold  $\hat{M}$ to be a
    $C^{\infty}$-orbifold, we use the fact that $\hat{M}$ satisfies
    K\"ahler Ricci soliton equation (c.f.~\cite{Se1})
    and this improvement is a
    standard application of Ulenbeck's removing singularity technique.\\

  \begin{remarkin} Incidently,
  Theorem~\ref{theoremin: krfcpt} verifies Hamilton-Tian conjecture
  in the case of Fano surfaces.   This might be viewed as a first step towards
  the understanding of Hamilton-Tian conjecture in general dimension.
  In an unpublished work(c.f.~\cite{Se1},~\cite{FZ}),  Tian has pointed out earlier
  the sequential convergence of the 2-dimensional \KRf to K\"ahler
  Ricci soliton orbifolds under the Gromov-Hausdorff topology.
  Under the extra condition that Ricci curvature
  is uniformly bounded along the flow, same conclusion as Theorem~\ref{theoremin: krfcpt}
  was proved  by Natasa Sesum in~\cite{Se1}.
  However, for our purpose of
  proving Calabi Conjecture on Fano surface by flow method,
  these convergence theorems are not sufficient (We need Cheeger-Gromov
  convergence without Ricci curvature bound condition).
 \end{remarkin}

  As every gradient shrinking soliton can be looked as a central time
  slice of a Ricci flow solution, we can apply
  Theorem~\ref{theoremin: centerwcpt} and
  Theorem~\ref{theoremin: centersob} to obtain a compactness
  theorem of gradient shrinking solitons.
 \begin{theoremin}
  Suppose $(X_i^m, g_i)$ is a sequence of compact gradient shrinking solitons,
 every $(Y^m,g)$ in
 $\displaystyle  \{(X_i^m, g_i)\}_{i=1}^{\infty}$ satisfies
 \begin{itemize}
  \item $Ric + \nabla^2 f - g=0$ for $f \in C^{\infty}(Y)$.
  \item $R \leq \sigma$.
  \item $\frac{\Vol(B(y, r))}{r^m} \geq \kappa$ for every $y \in Y$ and
  $r \in (0,1)$.
  \item $\int_Y |Rm|^{\frac{m}{2}}d\mu \leq E$.
  \item $\Vol(Y) \leq V$.
  \end{itemize}
  Then by passing to subsequence if necessary, we have
  \begin{align*}
    (X_i^m, g_i) \sconv (\hat{X}^m, \hat{g})
  \end{align*}
  where $(\hat{X}^m, \hat{g})$ is a compact $C^{\infty}$-orbifold satisfying
  all the conditions listed above.  If $m$ is odd, then $\hat{X}^m$
  is a smooth manifold.
 \end{theoremin}

  \begin{remarkin}
  This theorem can be regarded as a reorganization of the main theorem
  of~\cite{We}, which is the first weak compactness theorem of
  solitons without Ricci curvature bounded condition. With this Ricci bounded condition,
  weak compactness theorems have been studied in ~\cite{CS} and~\cite{Zhx}.
  \end{remarkin}
 \vspace{0.2in}

  There are several new ingredients in this paper.   First, we
  prove the energy concentration without the direct use of Moser
  iteration on the flow.  Instead, we find the  relation
  between the energy concentration of Ricci flow
  and the energy concentration of Einstein metrics, which are the critical metrics
  of the Ricci flow. As the energy concentration of Einstein metrics is well known,
  we can use this relation to prove energy concentration property of the Ricci flow.
  This method should apply on general geometric flows.
  Second, we find the ``backward pseudolocality"
  under special conditions.   It means that a very nonflat part cannot become almost Euclidean in
  a short time under the flow.  It will be interesting to see exactly what conditions can guarantee
  the happening of this ``backward pseudolocality".
  Third, we can control the ``neck" structure in bubbles without knowing precisely the order
  of curvature decaying around singularities. This gives us more freedom to build up the bubble tree.
    Fourth, we're able to control the isoperimetric constants by bubble tree analysis. Actually, after
  the bubble tree is established, the geometry at different levels of the bubble tree
  can be estimated by each other. Therefore every rescaling invariant geometry of the moduli space
  can be controlled by it's deepest bubble and ``neck" structure.
  In particular, the isoperimetric constants can be controlled.\\

  The organization of this paper is as follows.  In section 2, we
  describe some known results important to us and set up some
  notations.  In section 3, we define refined sequence and
  prove energy concentration, backward pseudolocality,
  volume ratio upper bound and weak compactness for refined sequence.
   Then in section 4, we apply these basic properties of refined sequence to show
  the main theorems in this paper.\\

 \begin{remarkin}
  We will study space of Ricci flows with weaker constraints in subsequent papers.
 \end{remarkin}

 \noindent {\bf Acknowledgment}
 Both authors would like to thank S.K.Donaldson and G.Tian for many
 insightful discussions and for their support.
 The second author would like to thank J. Cheeger,
 B.  Chow, K. Grove, J.P. Bourguignon for their interests in this work.

 \section{Preliminaries}

 \subsection{Setup of Notations}

  In this subsection, we fix our terminology to avoid confusion.

  \begin{definition}
   A $C^{\infty} (C^0)$-orbifold $(\hat{X}^m, \hat{g})$ is a topological
   space which is a smooth manifold with a smooth Riemannian metric
   away from finitely many singular points. At every singular point,
   $\hat{X}$ is locally diffeomorphic to a cone over $S^{m-1} / \Gamma$
   for some finite subgroup $\Gamma \subset SO(m)$. Furthermore, at
   such a singular point, the metric is locally the quotient of a
   smooth (continuous) $\Gamma$-invariant metric on $B^{m}$ under the orbifold
   group $\Gamma$.

     A $C^{\infty}(C^0)$-multifold $(\tilde{X}, \tilde{g})$ is a finite union
     of $C^{\infty}(C^0)$-orbifolds after identifying finite points. In other words,
     $\displaystyle \tilde{X}= \coprod_{i=1}^{N} \hat{X}_i / \sim$ where every
     $\displaystyle \hat{X}_i$ is an orbifold, the relation $\sim$ identifies
     finite points of $\displaystyle \coprod_{i=1}^{N} \hat{X}_i$.

     For simplicity, we say a space is an orbifold (multifold) if it
     is a $C^{\infty}$-orbifold ($C^{\infty}$-multifold).
 \end{definition}

 \begin{definition}
  Suppose $(X_i, x_i, g_i)$ is a sequence of pointed complete Riemannian
  manifold,  $(\tilde{X}, \tilde{x}, \tilde{g})$ is a complete
  multifold.  We denote
  \begin{align*}
     (X_i, x_i, g_i) \sconv (\tilde{X}, \tilde{x}, \tilde{g})
  \end{align*}
  if $(X_i, x_i, g_i)$ converges to $(\tilde{X}, \tilde{x}, \tilde{g})$
  in Gromov-Hausdorff topology and the convergence is smooth away
  from singularities of $\tilde{X}$.  In other words, for every
  compact set $K$ contained in the smooth part of $\tilde{X}$, there are
  diffeomorphisms $\varphi_i: K \to \varphi_i(K) \subset X_i$ such that
  $ \varphi_i^* g_i  \sconv g $ on $K$.  We also call this
  convergence as convergence in Cheeger-Gromov sense.

   Similarly, we can define $C^{1, \gamma}$-convergence.
 \end{definition}

 \begin{definition}
 For a compact Riemannian manifold $X^m$ without boundary, we define its isoperimetric
 constant as
     \begin{align*}
    \mathbf{I}(X)  \triangleq
    \inf_{\Omega} \frac{\Area(\partial \Omega)}{\min\{\Vol(\Omega), \Vol(X \backslash \Omega)\}^{\frac{m-1}{m}}}
  \end{align*}
    where $\Omega$ runs over all domains with rectifiable boundaries in $X$,
    $\Area$ means the $(m-1)$-dimensional volume.

    For a complete Riemannian manifold $X^m$ with boundary, we define its isoperimetric
 constant as
  \begin{align*}
    \mathbf{I}(X)  \triangleq
    \inf_{\Omega} \frac{\Area(\partial \Omega)}{\Vol(\Omega)^{\frac{m-1}{m}}}
  \end{align*}
    where $\Omega$ runs over all domains with rectifiable boundaries
    in the interior of $X$.

    Similarly, we can define isoperimetric constant for an orbifold.
 \end{definition}

 \begin{definition}
   A geodesic ball $B(p, \rho)$ is called $\kappa$-noncollapsed if
  $\displaystyle  \frac{\Vol(B(q, s))}{s^m} > \kappa $
  whenever $B(q, s) \subset B(p, \rho)$.

  A Riemannian manifold $X^m$ is called $\kappa$-noncollapsed on
  scale $r$ if every geodesic ball $B(p, \rho) \subset X$ is
  $\kappa$-noncollapsed whenever $\rho \leq r$.

   A Riemannian manifold $X^m$ is called $\kappa$-noncollapsed if it
   is $\kappa$-noncollapsed on every scale $r \leq \diam (X^m)$.

 \end{definition}

 \begin{definition}
   $\int_X |Rm|^{\frac{m}{2}} d\mu$ is called the energy of the
   Riemannian manifold $(X^m, g)$.
 \end{definition}

 \begin{definition}
  We denote $\omega(m)$ as the volume of the standard ball in
  $\R^m$.  $m \omega(m)$ is the ``area" of the standard sphere in
  $\R^m$.
 \end{definition}

   \subsection{Ricci Flow}

  Perelman's improved peudolocality theorem (Theorem 10.3 in \cite{Pe1})
  is very important to our arguments. We list it below.

 \begin{theorem}[\textbf{Perelman's Improved Pseudolocality Theorem}]
  There exist $\eta, \delta>0$ with the following property.
  Suppose $g_{ij}(t)$ is a smooth solution to the Ricci flow on $[0, (\eta
  r_0)^2]$, and assume that at $t=0$ we have $|Rm|(x) \leq r_0^{-2}$
  in $B(x_0, r_0)$, and $\Vol B(x_0, r_0) \geq (1-\delta) \omega(m) r_0^m$, where $\omega(m)$
  is the volume of the unit ball in $\R^m$. Then the estimate
  $|Rm|_{g(t)}(x) \leq (\eta r_0)^{-2}$ holds whenever $0 \leq t \leq (\eta
  r_0)^2$,  $d_{g(t)}(x, x_0) < \eta r_0$.
 \label{theorem: pseudo}
 \end{theorem}
  It is not hard to see that this theorem holds for a normalized Ricci flow solution
  $\D{g_{ij}}{t}=-R_{ij}+cg_{ij}$ when $0 \leq c \leq 1$.\\

 Using the same method as in Theorem 10.1 of~\cite{Pe1}, we can pick
 up ``good" points under normalized Ricci flow.

 \begin{lemma}[\textbf{Perelman's Point-selecting Method}]
 $\{( X^m,p, g(t)), t \in I \subset \R \}$  is a parabolic normalized
 Ricci flow solution:
 \begin{align*}
    \D{g_{ij}}{t} = -R_{ij} + cg_{ij}, c \geq 0.
 \end{align*}
 It satisfies
 \begin{align*}
      |Rm|_{g(t_0)}(x_0) > \frac{\xi}{t_0}+ \eta^{-2},  \quad  d_{g(t_0)}(p,x_0) <
      \eta, \quad 0 < t_0 < \eta^2 <1
 \end{align*}
 at some point $(x_0, t_0)$, then there is a point $(q,s)$ such that
 the following properties hold.
 \begin{enumerate}
 \item{distance control:}  \quad $d_{g(s)}(p,q) < 2m \xi, \quad 0< s <t_0$.
 \item{parabolic curvature control:}  \quad
 $ \displaystyle \sup_{B_{g(s)}(q, \frac{1}{10}A Q^{-\frac12}) \times
 (s-\frac12 \xi Q^{-1}, s]} |Rm| < 4|Rm|_{g(s)}(q)$.
 \end{enumerate}
 Here $A, \eta, \xi$ are constants satisfying $A>1$,
 $\xi<\frac{1}{200m}$ and $\eta=\frac{\xi}{A}$.
  \label{lemma: choosepoints}
 \end{lemma}

  \subsection{Ricci Flat Spaces}
  As collections of previous works about Einstein metrics
  in~\cite{BKN},~\cite{An89},~\cite{An90},~\cite{Tian90},~\cite{TV1},
  ~\cite{TV2} and~\cite{TV3},
  we list the following results about Einstein metrics.

  \begin{lemma}[Bando,~\cite{Ban90}]
    There exists a constant $\epsilon_a=\epsilon_a(m, \kappa)$ such that
    the following property holds.

   If $X$ is a $\kappa$-noncollapsed, Ricci-flat ALE orbifold,
         it has unique singularity and small energy, i.e.,
         $\int_X |Rm|^{\frac{m}{2}}d\mu<\epsilon_a$,
        then $X$ is a flat cone.
  \label{lemma: gap}
  \end{lemma}

 \begin{lemma}
 Suppose $B(p, \rho)$ is a smooth, Ricci-flat, $\kappa$-noncollapsed geodesic
  ball and $\partial B(p, \rho) \neq \emptyset$. Then there is a small constant
  $\epsilon_b=\epsilon_b(m, \kappa)$ such that
  \begin{align}
   \sup_{B(p, \frac{\rho}{2})} |\nabla^k Rm| \leq
    \frac{C_k}{\rho^{2+k}}  \{\int_{B(p, \rho)} |Rm|^{\frac{m}{2}} d\mu\}^{\frac{2}{m}}
   \label{eqn: econ}
 \end{align}
 whenever $\int_{B(p, \rho)} |Rm|^{\frac{m}{2}} d\mu < \epsilon_b$.
 In particular, $B(p, \rho)$ satisfies energy concentration
 property. In other words, if $|Rm|(p) \geq \frac{1}{\rho^2}$,  then we have
     \begin{align*}
        \int_{B(p, \frac{\rho}{2})} |Rm|^{\frac{m}{2}} d\mu > \epsilon_b.
     \end{align*}
 \label{lemma: econ}
 \end{lemma}

 Note that in this lemma, $B(p, \rho)$ can be a smooth
  Ricci-flat geodesic ball in any space, it needn't to be a geodesic ball
  in a smooth manifold. The hard part is to obtain a local
  Sobolev constant control when energy is small. This control can be obtained by blowup
  arguments as in~\cite{TV1},~\cite{TV2} and~\cite{TV3}.  If $B(p, \rho)$ is a smooth
  geodesic ball in a Ricci-flat manifold, then it's local Sobolev
  constant is easy to obtain and this lemma becomes trivial.

 \begin{theorem}[\cite{An89},~\cite{BKN},~\cite{Tian90}]
  $(X_i^m, x_i, g_i)$ is a sequence of pointed Ricci-flat Riemannian manifolds
  with bounded energy. For every $r>0$,
  $(X_i, g_i)$ is $\kappa$-noncollapsed on this scale for large $i$.
  Moreover, the Einstein constants is tending to zero. Then by passing to subsequence if necessary,
  we have
 \begin{align*}
    (X_i^m, x_i, g_i) \sconv (X^m, x, g)
 \end{align*}
 where $(X, g)$ is a $\kappa$-noncollapsed, Ricci-flat ALE orbifold.
 If $m$ is odd, then $(X, g)$ is Euclidean space.
 \label{theorem: ricciflat}
 \end{theorem}

   This theorem is the starting point of this whole paper.
  Every such sequence can be looked as a blown up sequence
  from Einstein metrics with bounded Einstein constants.
  As Einstein manifolds can be looked as static solutions of normalized
  Ricci flows, a natural generalization of Theorem~\ref{theorem: ricciflat}
  is the weak compactness of ``almost Ricci-flat" Ricci
  flow solutions.   But the delicate thing is how to define
  precisely a sequence of flows as an ``almost Ricci-flat" flow sequence. If
  we require the  Ricci curvature norm  tends to zero.
  Then there will be no essential new difficulty coming out.
  The proof follows almost directly from~\cite{AC}.
  Moreover, the Ricci bound condition seems to be too strong. It
  restricts the application of such a theorem.
  So an interesting theorem should be a theorem dealing with
  weaker curvature conditions.
  Note that every scalar flat Ricci flow solution
   must be a Ricci-flat Ricci flow solution. It's natural to expect
   that scalar curvature  tends zero is a good candidate
   for ``almost Ricci-flat" condition.
  Actually, it is the case.   Under similar technical conditions,
  we can prove the convergence of such Ricci flow solution
    sequence. We call such a sequence as a refined sequence. Its
  precise definition will be given in
  Definition~\ref{definition: refined}.

\section{Refined Sequences}

   In order to prove the weak compactness of $\kappa$-noncollapsed
   Ricci flows with bounded scalar curvature and bounded energy,
   we use blowup arguments.  In
   every blowup sequence, scalar curvature must tend to zero.
   We find that it is convenient to study the properties of
   such sequences first.  We're able to show that
   such blowup sequence has weak compactness.
   The idea of the proof originates from the proof of Theorem~\ref{theorem: ricciflat}.

\subsection{Refined, E-refined and EV-refined Sequences}
  A refined sequence is a sequence of Ricci flow solutions
  blown up from noncollapsed Ricci flow solutions with bounded energy
  and bounded scalar curvature.
 \begin{definition}
    Let $\{ (X_i^m, g_i(t)), -1 \leq t \leq 1 \}$ be a sequence of Ricci flows on
    closed manifolds $X_i^m$.  It is called a
    refined sequence if the following properties  are  satisfied for every $i$.
 \begin{enumerate}
 \item $\displaystyle \D{}{t} g_{i} = -Ric_{g_i} + c_i g_i$ and
      $\displaystyle \lim_{i \to \infty} c_i =0$.
 \item  Scalar curvature norm tends to zero:
 \begin{align*}
 \displaystyle \lim_{i \to \infty} \sup_{(x,t) \in X_i \times [-1, 1] }
 |R|_{g_i(t)}(x) =0.
 \end{align*}

 \item  For every $r$, there exists $N(r)$ such that
  $(X_i, g_i(t))$ is $\kappa$-noncollapsed on scale $r$ for every $t \in [-1, 1]$
  whenever $i>N(r)$.

 \item Energy uniformly bounded by $E$:
 \begin{align*}
 \int_{X_i} |Rm|_{g_i(t)}^{\frac{m}{2}} d\mu_{g_i(t)} \leq E, \qquad \forall \; t
 \in [-1,1].
 \end{align*}

 \end{enumerate}

 We also call a pointed spacetime sequence $\{ (X_i, x_i, g_i(t)),
 -1 \leq t \leq  1  \}$ a refined sequence if $\{ (X_i, g_i(t)), -1
 \leq t \leq  1  \}$ is a refined sequence.
    \label{definition: refined}
 \end{definition}

 Note that in a refined sequence, $Vol_{g_i(0)}(X_i)$ is tending to infinity.\\

 If $m$ is odd, then the structure of refined sequence is very
 simple.\\

 \begin{theorem}[\textbf{Odd-dimensional refined sequence is trivial}]
  Suppose $m$ is odd and
  $\{(X_i^m, g_i(t)), -1 \leq t \leq 1\}$ is a refined
  sequence, then we have
 \begin{align*}
    \lim_{i \to \infty} \sup_{X_i \times [-\frac18, 0]}
    |Rm|_{g_i(t)}(x) =0.
 \end{align*}
 \label{theorem: oddrefined}
 \end{theorem}

 \begin{proof}
   We first show that there is a constant $C$ depending on this
   sequence such that
   \begin{align*}
    \sup_{X_i \times [-\frac14, 0]}
    |Rm|_{g_i(t)}(x) \leq C, \quad \forall \; i.
   \end{align*}
   Otherwise, by passing to subsequence if necessary, we have
   sequence $(x_i, t_i) \in X_i \times [-\frac14, 0]$ such that
   $\displaystyle  \lim_{i \to \infty} |Rm|_{g_i(t_i)}(x_i) =
   \infty$.

   \begin{clm}
     There are points $(y_i, s_i) \in X_i \times [-\frac12, 0]$ such that
   $\displaystyle \lim_{i \to \infty} |Rm|_{g_i(s_i)}(y_i) =\infty$.
   Moreover, $(y_i, s_i)$ satisfies curvature control, i.e.,
   \begin{align*}
      |Rm|_{g_i(t)}(x) \leq 2 |Rm|_{g_i(s_i)}(y_i), \quad
   \forall \;  x \in X_i, \; t \in [s_i- |Rm|_{g_i(s_i)}^{-\frac12}(y_i),
   s_i].
   \end{align*}
   \end{clm}

   Check if $|Rm|_{g_i(t_i)}(x_i)$ satisfies curvature control.
   If so, we let $(y_i, s_i) = (x_i, t_i)$ and stop. Otherwise,
   we can find $(x_i^{(1)}, t_i^{(1)}) \in [t_i-|Rm|_{g_i(t_i)}^{-\frac12}(x_i), t_i]$
   whose curvature norm is bigger than $2|Rm|_{g_i(t_i)}(x_i)$. Then
   check if $(x_i^{(1)}, t_i^{(1)})$ satisfies curvature control. If
   so, we let $(y_i, s_i)=(x_i^{(1)}, t_i^{(1)})$. Otherwise, we
   continue our point selecting process.

   After each step, the Riemannian curvature norm of the base point doubles. So
   all the steps are processed in the compact region
   $X_i \times [t_i - 2|Rm|_{g_i(t_i)}^{-\frac12}(x_i), t_i] \subset X_i \times [-\frac12, 0]$
   which has bounded geometry.
   Clearly, this process must stop in finite times. We define $(y_i, s_i)$
   to be the base point of the last step. This finishes the proof of
   Claim.\\

   Let $\tilde{g}_i(t)= Q_ig_i(Q_i^{-1}t + s_i)$, where
   $Q_i=|Rm|_{g_i(s_i)}(y_i)$.  So $\{(X_i, y_i, \tilde{g}_i(t)), -1 \leq t \leq 1\}$
   is a refined sequence satisfying
   \begin{align*}
       |Rm|_{\tilde{g}_i(t)}(x) \leq 2, \quad \forall \; x \in X_i,
       \; -1 \leq t \leq 0.
   \end{align*}
   By the compactness of Ricci flow solutions, $\{(X_i, y_i, g_i(t)), -1 < t \leq 0\}$
   will smoothly converge to a Ricci flow solution $\{(\tilde{X}, \tilde{y}, \tilde{g}(t)), -1 < t \leq
   0\}$.  Clearly, it is a scalar-flat Ricci flow solution.  Therefore it is
   Ricci-flat by maximal principal.
   On the other hand, the energy of $\tilde{X}$ comes from the
   energy of $X_i$. So Fatou's lemma tells us that
   \begin{align*}
      \int_{\tilde{X}} |Rm|_{\tilde{g}(0)}^{\frac{m}{2}} d\mu_{\tilde{g}(0)} \leq
      \lim_{i \to \infty} \int_{X_i} |Rm|_{\tilde{g}_i(0)}^{\frac{m}{2}} d\mu_{\tilde{g}_i(0)}
      \leq
      \lim_{i \to \infty} \int_{X_i} |Rm|_{g_i(s_i)}^{\frac{m}{2}} d\mu_{g_i(s_i)}
       \leq E.
   \end{align*}
   Therefore, $(\tilde{X}, \tilde{g}(0))$ is a Ricci-flat manifold with
   bounded energy. Moreover, it is $\kappa$-noncollapsed on all
   scales. It follows from~\cite{An89} or~\cite{BKN}
   that $(\tilde{X}, \tilde{g}(0))$ is a Ricci flat ALE manifold.  However,
   such a manifold must be Euclidean space if $m$ is odd.   Therefore, we have
   \begin{align*}
    1= \lim_{i \to \infty} |Rm|_{\tilde{g}_i(0)}(y_i) = |Rm|_{\tilde{g}(0)}(\tilde{y}) = 0.
   \end{align*}
   Contradiction! Therefore, there must be a constant $C$ such that
   \begin{align*}
    \sup_{X_i \times [-\frac14, 0]}
    |Rm|_{g_i(t)}(x) \leq C, \quad \forall \; i.
   \end{align*}

   Let $(z_i, \theta_i)$ be the point in $X_i \times [-\frac18, 0]$
 with largest curvature norm.  As Riemannian curvature is uniformly bounded on
 $X_i \times [-\frac14, 0]$, we can assume that $(X_i, z_i, g_i(\theta_i))$ smoothly converges  to a complete manifold $(X_{\infty}, z_{\infty}, g_{\infty})$.
 As argued before, $X_{\infty}$ is an odd dimensional Ricci flat ALE space. Therefore,
 $X_{\infty}$ is a Euclidean space. Then it follows that
 \begin{align*}
 \lim_{i \to \infty}  \sup_{X_i \times [-\frac14, 0]} |Rm|_{g_i(t)}(x)
 =\lim_{i \to \infty}  |Rm|_{g_i(\theta_i)}(z_i)=|Rm|_{g_{\infty}}(z_{\infty})=0.
 \end{align*}

 \end{proof}

  Because of this simplicity, we're only interested in refined
  sequences of even dimension.
  When $m$ is even, the
  phenomena are much more complicated. As the in
  the proof of Theorem~\ref{theorem: oddrefined}, we can use blowup
  arguments to see what happens at the points with global maximal
  Riemannian curvature norm.     The blowup limit will be a
  $\kappa$-noncollapsed, Ricci-flat ALE manifold. Some nontrivial
  examples do exist when $m$ is even. So no contradiction can be
  obtained if Riemannian curvature is not uniformly bounded in the
  central time periods. Actually, if we construct a refined sequence
  by letting every solution be a static Ricci flow soluiton, i.e.,
  Einstein manifold, we see that it is really possible that
  Riemannian curvature is not uniformly bounded.          However,
  about these global maximal points, we can still draw some conclusion.
  They must satisfy the energy concentration property,
  the volume ratio of every geodesic ball centered at these points
  is bounded from above whenever the radius is comparable to $|Rm|^{-\frac12}$ of
  these base points.     It's natural to hope that both these two
  properties hold for all high curvature points in a refined
  sequence, not only for the points with global maximal Riemannian
  curvature. But there is an obvious difficulty to prove this directly:
  we don't have Harnack inequality for Riemannian curvature.  It is possible that
  Riemannian curvature tends to infinitely large during infinitely
  small distance. When this happens, the base points will be
  absorbed by singularities even we can take limit. So no contradiction
  can be obtained then.  Needless to say that we don't know whether
  we can take limit now.    In order to overcome this difficulty,
  we first study $EV$-refined sequences, which is a ``special"
  refined sequence where weak limit can be taken.  For $EV$-refined sequences, we find the two
  properties (volume ratio bound and energy concentration) can be
  proved even if the limit space contains singularities.
  After we obtain these two properties, we return to study how
  ``special" the $EV$-refined sequences are.  It turns out that they
  are not special at all, every refined sequence is an $EV$-refined
  sequence.  Therefore, every refined sequence satisfies energy
  concentration and volume ratio bound condition, and weak limit
  exists for every refined sequence.\\

  In order to define $EV$-refined sequence, we first need to fix
  some universal constants.

  \textbf{From now on, we fix the constants $\kappa$, $E$.  Moreover, we
 define
 \begin{align*}
    \epsilon= \min\{\epsilon_a, \epsilon_b\}
 \end{align*}
 where $\epsilon_a$ and $\epsilon_b$ are constants
 in Lemma~\ref{lemma: gap} and Lemma~\ref{lemma: econ}.
 We also fix constant $N_0 = \lfloor  \frac{E}{\epsilon}
 \rfloor$.}\\

  \begin{definition}
    A refined sequence $\{(X_i, g_i(t)), -1 \leq t \leq 1\}$
  is called an E-refined sequence if there exists a constant $H$ such that
  \begin{align*}
 \int_{B(x,|Rm|_{g_i(t)}^{-\frac12}(x))} |Rm|_{g_i(t)}^{\frac{m}{2}} d\mu_{g_i(t)}
\epsilon
 \end{align*}
 whenever $(x,t) \in X_i \times [-\frac12,0]$ and $|Rm|_{g_i(t)}(x) > H$.

 We also call a pointed normalized Ricci flow sequence
    $\{ (X_i, x_i, g_i(t)), -1 \leq t \leq  1  \}$ an E-refined sequence if
    $\{ (X_i, g_i(t)), -1 \leq t \leq  1  \}$ is an E-refined sequence.
  \end{definition}

  In short, an E-refined sequence is a refined sequence whose
 center-part-solutions  satisfy energy concentration property.

  \begin{definition}
  An  E-refined sequence $\{(X_i, g_i(t)), -1 \leq t \leq 1\}$
  is called an EV-refined sequence if  there is a constant $K$ such
  that
     \begin{align*}
       \frac{\Vol_{g_i(t)}{B_{g_i(t)}(x, r)}}{r^m} <K
    \end{align*}
    for every $i$ and $(x, t) \in X_i \times [-\frac14, 0]$, $r \in (0, 1]$.

    We also call a pointed normalized Ricci flow sequence
    $\{ (X_i, x_i, g_i(t)), -1 \leq t \leq  1  \}$ an EV-refined sequence if
    $\{ (X_i, g_i(t)), -1 \leq t \leq  1  \}$ is an EV-refined sequence.
  \end{definition}

   In short, an EV-refined sequence is an E-refined sequence whose
 center-part-solutions  have bounded volume ratios (from both sides).

   Since volume ratio, energy are scaling invariants, an easy
  observation implies the following property.

 \begin{proposition}
  If $\{(X_i, g_i(t)), -1 \leq t \leq 1\}$ is a refined sequence,
  $\lambda_i \geq 2$, $t_i \in [-\frac12, 0]$,  then
  $\{(X_i, \lambda_i g_i(\lambda_i^{-1}t + t_i)), -1 \leq t \leq
  1\}$ is a new refined sequence.

  If $\{(X_i, g_i(t)), -1 \leq t \leq 1\}$ is an E-refined sequence,
  $\lambda_i \geq 2$, $t_i \in [-\frac14, 0]$,  then
  $\{(X_i, \lambda_i g_i(\lambda_i^{-1}t + t_i)), -1 \leq t \leq
  1\}$ is a new E-refined sequence.

  If $\{(X_i, g_i(t)), -1 \leq t \leq 1\}$ is an EV-refined sequence,
  $\lambda_i \geq 2$, $t_i \in [-\frac18, 0]$,  then
  $\{(X_i, \lambda_i g_i(\lambda_i^{-1}t + t_i)), -1 \leq t \leq
  1\}$ is a new EV-refined sequence.
 \label{proposition: blowup}
 \end{proposition}
  In short, blowing up a (E-, EV-)refined sequence generates a new (E-, EV-)refined
  sequence.\\

\subsection{Deepest Bubble Structure}
  If we blow up a refined sequence at maximal curvature points, we
 can obtain some Ricci-flat manifold with bounded energy. Such
 manifolds are well understood.

  \begin{lemma}
    There exists a large constant $\rho_b= \rho_b(m, E, \kappa)$ satisfying the
    following property.

   If $\{(X^m, x, g(t)),  -\frac{1}{1000m} \leq t \leq 0 \}$ is a
   compact spacetime  satisfying
  \begin{enumerate}
   \item $\displaystyle \D{}{t}g(t)=-Ric_{g(t)} + c g(t)$ where $c$ is a
   constant,
   \item $\displaystyle |c| \leq \frac{1}{\rho_b^2}$ and
    $\displaystyle \sup_{X \times [-\frac{1}{1000m}, 0]} |R| \leq
    \frac{1}{\rho_b^2}$,
   \item $\int_{X} |Rm|_{g(t)}^{\frac{m}{2}} d\mu_{g(t)} \leq E, \quad \forall t \in [-\frac{1}{1000m}, 0]$.
   \item $\frac{\Vol_{g(t)}(B_{g(t)}(y,r))}{r^m} \geq \kappa$ for all $y \in X, t \in [-\frac{1}{1000m}, 0], r \in
   (0,1]$.
   \item $\displaystyle      |Rm|_{g(0)}(x) =1,    \quad   |Rm|_g  \leq 4 \quad
      \textrm{in} \; B_{g(0)}(x, \rho_b) \times [-\frac{1}{1000m}, 0]$,
   \end{enumerate}
  then there exists an $r \in (0, \frac12 \rho_b)$ such that property
  $\clubsuit$(defined below) is satisfied.
  \begin{itemize}
  \item $B_{g(0)}(x, r)$ is diffeomorphic to a nontrivial Ricci-flat ALE space.
  \item $B_{g(0)}(x, 2r) \backslash B_{g(0)}(x, r)$ is
  diffeomorphic to $S^{m-1} / \Gamma \times  [1,2)$ for some $SO(m)$'s nontrivial finite subgroup $\Gamma$.
  \item $\displaystyle \frac{\Vol_{g(0)}B_{g(0)}(x, r)}{r^m} < \frac23 \omega(m)$.
  \end{itemize}
  \label{lemma: topology}
  \end{lemma}

  \begin{proof}
   Suppose this result is wrong, then there is a sequence of $\rho_i \to \infty$
  and spacetime $\{ (X_i, x_i, g_i(t)), -\frac{1}{1000m} \leq t \leq 0\}$
  satisfying the five conditions. However, for every $r \in (0, \frac12
  \rho_i)$,  property $\clubsuit$ fails.

   As $\rho_i \to \infty$, using $\kappa$-noncollapsing condition and Shi's estimates, we
   can take the smooth limit
   \begin{align*}
   \{(X_i^m, x_i, g_i(t)),  -\frac{1}{2000m} \leq t \leq 0 \}
   \sconv
   \{(X, x, g(t)), -\frac{1}{2000m} \leq t \leq
   0\}.
 \end{align*}
 The limit solution satisfies
 \begin{align*}
   \D{}{t} g(t) = -Ric_{g(t)}=0;
   \qquad R_{g(t)} \equiv 0,   \; \forall \; t \in [-\frac{1}{2000m}, 0].
 \end{align*}
  The evolution equation of  scalar curvature is
 \begin{align*}
    \D{}{t} R_{g} = \frac12 \triangle R_{g} +
    |Ric|_{g}^{\frac{m}{2}}.
 \end{align*}
 This indicates that $g(t)$ is actually a Ricci-flat solution.  Moreover, Fatou's
 lemma tells us
 \begin{align*}
  \int_{X} |Rm|_{g(0)}^{\frac{m}{2}} d\mu_{g(0)} \leq E.
 \end{align*}
 Therefore $(X, g(0))$ is a nonflat Ricci-flat manifold with bounded energy.
 As argued in ~\cite{An89}, ~\cite{BKN} and
 ~\cite{Tian90}, we see $(X, g(0))$  is an  Asymptotically Locally Euclidean space (ALE)
 with one end at infinity.
 Blowing down this ALE, we obtain a flat cone over $S^{m-1} /
 \Gamma$ where $\Gamma$ is a finite subgroup of $SO(m)$.
    $\Gamma$ must be nontrivial. Otherwise we have
 \begin{align*}
     \lim_{\rho \to \infty} \frac{\Vol(B(x, \rho))}{\rho^m}
     = \frac{\omega(m)}{|\Gamma|}=\omega(m).
 \end{align*}
 Then Bishop volume comparison theorem implies that every geodesic
 ball in $X$ has Euclidean volume growth rate. This means
 $X$ is actually flat and $|Rm|(x)=0$.  This is
 impossible since
 $\displaystyle |Rm|(x) = \lim_{i \to \infty} |Rm|_{g_i(0)}(x_i)=1$.

  As $(X, g(0))$ is a nontrivial ALE, we can choose
  $r_0$ large such that $B_{g(0)}(x, r_0)$ is diffeomorphic to
  $(X, g(0))$ and  $B(x, 2r) \backslash B(x, r)$ is
  diffeomorphic to $S^{m-1} / \Gamma \times  [1,2)$ for a nontrivial $\Gamma$.
  Furthermore, since $|\Gamma| \geq 2$, we can make $r_0$ big enough such that
  \begin{align*}
     \frac{\Vol(B(x, r_0))}{r_0^m} < \frac23 \omega(m).
  \end{align*}
 As $(X_i, x_i, g_i(0))$ converges to  $(X, x,  g(0))$
   in smooth topology, for large $i$,  we know property $\clubsuit$
   hold for $r_0$. Notice that $r_0 \in (0, \frac12 \rho_i)$ for
   large $i$. This contradicts to our assumption!
  \end{proof}

  \begin{definition}
 If the refined sequence $\{(X_i, g_i(t)), -1 \leq t \leq 1\}$
 has energy bound $E$ and volume ratio lower bound $\kappa$, we call
  $\rho_b=\rho_b(m, E, \kappa)$ as the base radius of this refined sequence .
 \label{definition: base}
\end{definition}

 Lemma~\ref{lemma: topology} can be used to control the geometric
 property of local maximal curvature points in a refined sequence.

 \begin{corollary}
  Let $\{(X_i^m, x_i, g_i(t)),  -1 \leq t \leq 1 \}$ be a refined
   sequence satisfying
   \begin{align*}
      |Rm|_{g_i(0)}(x_i) =1,    \quad   |Rm|_{g_i} \leq 4 \quad
      \textrm{in} \; B_{g_i(0)}(x_i, \rho_b) \times [-\frac{1}{1000m}, 0],
   \end{align*}
 where $\rho_b$ is the base radius of this sequence. Then for
 large $i$, we have
 \begin{align*}
        \frac{\Vol_{g_i(0)} B_{g_i(0)}(x_i,
        \frac{\rho_b}{2})}{(\frac{\rho_b}{2})^m}< \frac34 \omega(m).
 \end{align*}
 \label{corollary: vr}
 \end{corollary}
 \begin{proof}
  Otherwise, we can extract a subsequence as a refined sequence
  violating the volume ratio control. Still denoting this refined
  sequence as $\{(X_i^m, x_i, g_i(t)),  -1 \leq t \leq 1 \}$, we
  have
 \begin{align*}
     \frac{\Vol_{g_i(0)} B_{g_i(0)}(x_i, \frac{\rho_b}{2})}{(\frac{\rho_b}{2})^m} \geq \frac34 \omega(m)
 \end{align*}
 By Lemma~\ref{lemma: topology}, we can choose $r_i \in (0,
 \frac{\rho_b}{2})$ such that
 \begin{align*}
    \frac{\Vol_{g_i(0)} B_{g_i(0)}(x_i, r_i)}{(r_i)^m} < \frac23 \omega(m)
 \end{align*}
 Take limit, we have
 \begin{align*}
  \{ B_{g_i(0)}(x_i, \frac{\rho_b}{2}), x_i,  g_i(0)\}
   \sconv \{B_{g}(x,  \frac{\rho_b}{2}), x, g\},
 \end{align*}
 where $B_{g}(x,  \frac{\rho_b}{2})$ is a smooth
 Ricci-flat geodesic ball.  Let $r_i \to r_{\infty}$. By the
 smoothness of $B_{g}(x,  \frac{\rho_b}{2})$, we
 have $0<r_{\infty} \leq \frac{\rho_b}{2}$.  Bishop volume
 comparison theorem implies that
 \begin{align*}
  \frac34 \omega(m) \leq
  \frac{\Vol(B(x, \frac{\rho_b}{2}))}{(\frac{\rho_b}{2})^m} \leq
     \frac{\Vol(B(x,r_{\infty}))}{(r_{\infty})}^m  \leq  \frac23 \omega(m).
 \end{align*}
 Contradiction!
 \end{proof}

\subsection{Basic Properties of EV-refined Sequences}
 In order to study refined sequence, we start from EV-refined
 sequence. Roughly speaking, it is refined sequence with volume ratio upper
 bound and energy concentration conditions.  These extra conditions are
 important to obtain the weak compactness theorems.

\begin{lemma}
 [\textbf{Weak Compactness of an EV-refined Sequence in $C^{1, \gamma}$-topology}]
  Suppose $\{ (X_i, x_i, g_i(t)), -1 \leq t \leq 1
 \}$ is an EV-refined sequence, $t_i \in [-\frac14,0]$.
 Then  $(X_i, x_i, g_i(t_i))$ converges to a Ricci-flat multifold
 $(X, x, g)$ in Gromov-Hausdorff topology.
  Furthermore, there are $L (\leq N_0)$ points $p^1, \cdots, p^{L} \in X$
 such that $\displaystyle X \backslash \{ \bigcup_{s=1}^L p^s\}$ is smooth and
 the convergence is in $C^{1,\gamma}$-topology away from $\{p^{s}\}_{s=1}^{L}$ for any $\gamma \in (0,1)$.
   For brevity, we denote this convergence as
  $ (X_i, x_i, g_i(t_i)) \stackrel{C^{1,\gamma}}{\to}
   (X, x, g).$

   The limit multifold $(X, g)$ satisfies the following estimates
   \begin{itemize}
  \item For every point $p^s (1 \leq s \leq L)$,  the number of cone-like ends
  at $p^s$ is bounded by $\frac{2^mK}{\kappa}$, i.e.,
  $rank(H^0(X, X \backslash \{p^s\})) \leq \frac{2^mK}{\kappa}$.
  \item $\int_{X} |Rm|_{g}^{\frac{m}{2}} d\mu_{g} \leq E$.
   \end{itemize}
  \label{lemma: wwcpt}
  In particular, $(X, g)$ is a $\kappa$-noncollapsed, Ricci-flat ALE multifold.
 \end{lemma}

 \begin{proof}
  Without loss of generality, we can assume $t_i \equiv 0$.

  Since total energy is bounded by $E$ and at least $\epsilon$ energy concentrated around any point
  with $|Rm|_{g_i(0)}>H$, we can find $L(\leq N_0)$ small geodesic balls $\{B_{g_i(0)}(y_i^{(s)}, 2r)\}_{s=1}^L$
  such that $|Rm|_{g_i(0)} < r^{-2}$ outside these geodesic balls.
This means that for every point $x \in \displaystyle X_i \backslash
(\bigcup_{v=1}^{L}
 B_{g_i(0)}(y_i^{(v)}, 3r)) $, the geodesic ball $B_{g_i(0)}(x, r)$ has $|Rm|_{g_i(0)}$ upper bound $r^{-2}$.
 Using Perelman's  pseudolocality theorem (Theorem~\ref{theorem: pseudo}),
 we are able to get a rough Riemannian
  curvature control $\digamma(r, \kappa)$ at point $x$ for a short time period
     $\daleth(r, \kappa)$.  For brevity,  we denote
 \begin{align*}
   X_{i,r} = X_i  \backslash (\bigcup_{v=1}^{L}
 B_{g_i(0)}(y_i^{(v)},  3r)).
 \end{align*}
 So we have
 \begin{align*}
   |Rm|_{g_i(t)}(x) \leq \digamma(r, \kappa), \quad \forall \;
   (x,t)
   \in X_{i,r} \times [0, \daleth(r, \kappa)].
 \end{align*}
 Note that  $\displaystyle \Vol_{g_i(0)}(X_i) \geq \kappa \rho^m$
 for every $\rho$ and large $i$. So the local volume ratio upper bound
 forces that $\displaystyle \lim_{i \to \infty} \diam_{g_i(0)}(X_i)=\infty$.
 Therefore $X_{i,r}$ is nonempty when
 $r$ very small. By Shi's local estimate (\cite{Shi1}, \cite{Shi2}), we are able
 to take smooth limit
 \begin{align*}
     (X_{i,r}, x_i, g_i(t)) \stackrel{C^\infty}{\to}
     (X_{\infty,r}, x_{\infty, r}, g_{\infty, r}(t)), \quad \forall t \in (0, \daleth(r, \kappa)].
 \end{align*}
 However, at time $t=0$, we only have weaker convergence
 \begin{align*}
    (X_{i,r}, x_i, g_i(0)) \stackrel{C^{1,\gamma}}{\to}
    (X_{\infty,r}, x_{\infty, r}, g_{\infty, r}(0)).
 \end{align*}
 Although this weak convergence cannot conclude too much directly about the
 limit manifold $(X_{\infty, r}, g_{\infty, r}(0))$, the Ricci flatness
 of the limit solution in time period $(0,\daleth(r, \kappa)]$ comes to rescue us. The
 Ricci flatness tells us that metric tensor does not change in time
 period $(0,\daleth(r, \kappa)]$. The apriori Riemannian curvature norm bound
 $\digamma(r, \kappa)$ assures the limit metric tensor changes continuously
 at time $t=0$. Therefore, $(X_{\infty,r}, x_{\infty, r}, g_{\infty,r}(0))$ are the same as $(X_{\infty,r},
 x_{\infty, r}, g_{\infty, r}(t))$ for every $t \in
 (0,\daleth(r, \kappa)]$. In particular, $(X_{\infty,r}, x_{\infty, r}, g_{\infty, r}(0))$ is Ricci-flat.

 Now for each $r$ we obtain a Ricci-flat manifold $(X_{\infty, r},
 g_{\infty, r}(0))$. Further more, if $r_1 > r_2$, $(X_{\infty, r_1},g_{\infty, r_1}(0) )$ can be naturally
 looked as a sub-metric-space of   $(X_{\infty, r_2},
 g_{\infty, r_2}(0))$. Diagonal sequence argument gives us a limit space
 $\displaystyle X_{\infty, 0}=\bigcup_{l=1}^{\infty} X_{\infty,2^{-l}}$.   Local volume ratio
 upper boundedness of $(X_i, g_i(0))$  assures that $X_{\infty, 0}$ can be compactified by adding
  $L(\leq N_0)$ points $p^1, \cdots, p^L$. We denote the pointed compactification space as
 $(X, x, g)$. Therefore, diagonal sequence argument implies
   \begin{align*}
    (X_i, x_i, g_i(0)) \stackrel{C^{1,\gamma}}{\to}
   (X, x, g),
   \end{align*}
 where $(X, g)$ is smooth and has flat Ricci curvature away from
 $\{p^1, \cdots, p^L\}$.

 \begin{clm}
  $(X, g)$ is a multifold.
 \end{clm}

 First we prove every singular point has only finite ends.
 Let $p$ be a singular point on $X$, $\Lambda$ be the number of ends at point $p$, i.e.,
 $\Lambda=rank(H^0(X, X \backslash \{p\}))$.
 Choose $\delta$ very
 small and denote $E_1(\delta), \cdots, E_{\Lambda}(\delta)$ be all the
 components of $B(p, \delta) \backslash \{p\}$. Choose $q_l ( 1 \leq l \leq \Lambda)$ in
 $E_l(\delta)$ such that $d(p, q_l)= \frac{\delta}{2}$. Clearly,
 $E_l(\delta) \supset B(q_l, \frac{\delta}{2})$.   Therefore, volume ratio
 control implies
 \begin{align*}
  K\delta^m \geq  \Vol(B(p, \delta)) = \sum_{l=1}^{\Lambda} \Vol(E_l(\delta)) \geq \sum_{l=1}^{\Lambda}
   \Vol(B(q_l, \frac{\delta}{2})) \geq \Lambda \kappa
   (\frac{\delta}{2})^m
 \end{align*}
 It follows that $\Lambda \leq \frac{2^mK}{\kappa}$.

 Then we show every end around a singular point is an orbifold end.
 Since $\Vol(E_l(\delta)) < K\delta^4$ and $\int_{X} |Rm|^{\frac{m}{2}} d\mu \leq E$, we have
 $\displaystyle \lim_{\delta \to 0} \int_{E_{l}(\delta)} |Rm|^{\frac{m}{2}}
 d\mu=0$.
 Suppose $q \in E_l(\delta)$ and $d(q,p)=s$.
  Lemma~\ref{lemma: econ} implies
 \begin{align*}
      |\nabla^k Rm|(q)(\frac{s}{2})^{2+k} \leq C_k \int_{B(q, s)} |Rm|^{\frac{m}{2}} d\mu \to
      0,
      \quad \textrm{as} \; s \to 0.
 \end{align*}
 This means the tangent space of $E_l(\delta) \cup \{p\}$ at the
 point $p$ is a flat cone.  As argued in~\cite{Tian90},  $E_l(\delta) \cup
 \{p\}$ is a $C^0$-orbifold. Therefore the local Sobolev
 constant of $E_l(\delta) \cup \{p\}$ is very close to the Euclidean
 Sobolev constant when $\delta$ very small.   In particular,
 the local Sobolev constant of $E_l(\delta) \cup \{p\}$ is bounded.
 Recall that $E_l(\delta)$ is Ricci-flat, improved Kato's inequality (if $m=4$)
 together with Moser iteration gives us
 better curvature control such that we can construct a global smooth
 coordinate on $E_l(\delta)$.  Then Uhlenbeck's removing singularity technique implies
 that $p$ is a $C^{\infty}$-orbifold singularity for every $l$.  It follows that $p$ is a $C^{\infty}$-multifold
 singularity of $X$. So we finish the proof of the Claim.


  As every singular point must
 be one of $p^s$, it is clear that the number of singular points
 is controlled by $N_0=\lfloor \frac{E}{\epsilon} \rfloor$.  By  the
 construction of $X$, we have $\int_{X} |Rm|_{g}^{\frac{m}{2}} d\mu_{g} \leq E$.
 Use the same arguments as in~\cite{BKN} or~\cite{An89},  we know $(X, g)$ is an ALE space.
 \end{proof}

 \begin{remark}
   According to the construction of the convergence, the points
   $\{p^s\}_{s=1}^L$ comes from compactification.  We call such
   points as added points. It is possible that some  added point $p^s$
   is a smooth point though it comes from a  different origin from generic smooth points.
   However, every singular point  must be an added point.
 \label{remark: added}
 \end{remark}

 We have already shown the weak compactness of EV-refined sequence in
 $C^{1,\gamma}$-topology.  It is not enough for our main object.
 What we need is the weak compactness in $C^\infty$ topology.
 In order to do so, we need to control curvature's derivatives by
 curvature. Generally this is impossible.  However, in an EV-refined
 sequence, some geometric constraints force that a very non flat
 part  cannot become flat immediately under the flow.  This allows us to control
 Riemannian curvature of the previous time by the Riemannian
 curvature of a later time.  This kind of backward pseudolocality theorem
 together with Shi's local estimate will give us the control we
 need.   So the crucial step is to prove the backward pseudolocality
 theorem.    Let's first do some preparation for it.

\begin{lemma}[\textbf{Distance Continuity for an EV-refined Sequence}]
 Suppose \linebreak $\{(X_i, g_i(t)), -1 \leq t \leq 1\}$ is an EV-refined
  sequence, $t_i \leq 0$ and $\displaystyle \lim_{i \to \infty} t_i=0$, $x_i, y_i$
  are points in $X_i$.
  Then we have
  \begin{align*}
   \lim_{i \to \infty} d_{g_i(0)}(x_i, y_i)
   = \lim_{i \to \infty}  d_{g_i(t_i)}(x_i, y_i).
  \end{align*}
  \label{lemma: dcontinuity}
  if one of the limits exists.
 \end{lemma}
 \begin{proof}

 Before we get into the details of the proof, let's illustrate the
 basic idea and set up some notations.

 The main idea is to separate the shortest geodesic connecting $x_i$
 and $y_i$  into two parts. One part contributes most to the length
 but changes little under the flow. The other part may change very
 fast, but itself's length is tiny in every time slice. Combining these two estimates, we
 obtain the continuity of distance.

 By Lemma~\ref{lemma: wwcpt}, $(X_i, x_i, g_i(t_i))$ converges to some  Ricci-flat
 multifold in $C^{1,\gamma}$-topology. Denote this limit
 as $(X, x, g)$. The number of singular points
 in $X$ is bounded by $N_0 =\lfloor \frac{E}{\epsilon} \rfloor$.
 For simplicity of our argument, we  assume that $X$ contains only one singularity.  \quad
 Note that $(X, g)$ is an ALE $C^{\infty}$-multifold, so it satisfies the following
 property:

   Fix $\varsigma >0$ arbitrarily small.  Then for every
   point $z \in X$, there exists $\rho < \frac{\varsigma}{10}$
 such that $\partial B(z, \rho)$ has at most $\frac{2^mK}{\kappa}$
 components. Each component is a smooth $(m-1)$-dimensional Riemannian
 manifold and its diameter is not greater than $2\pi \rho$.
 See Figure~\ref{figure: smallball} for intuition.

   \begin{figure}
 \begin{center}
 \psfrag{x}[c][c]{$x$}
 \psfrag{p}[c][c]{$p$}
 \psfrag{Ball}[c][c]{$B(z,  \rho)$}
 \includegraphics[width=0.8 \columnwidth]{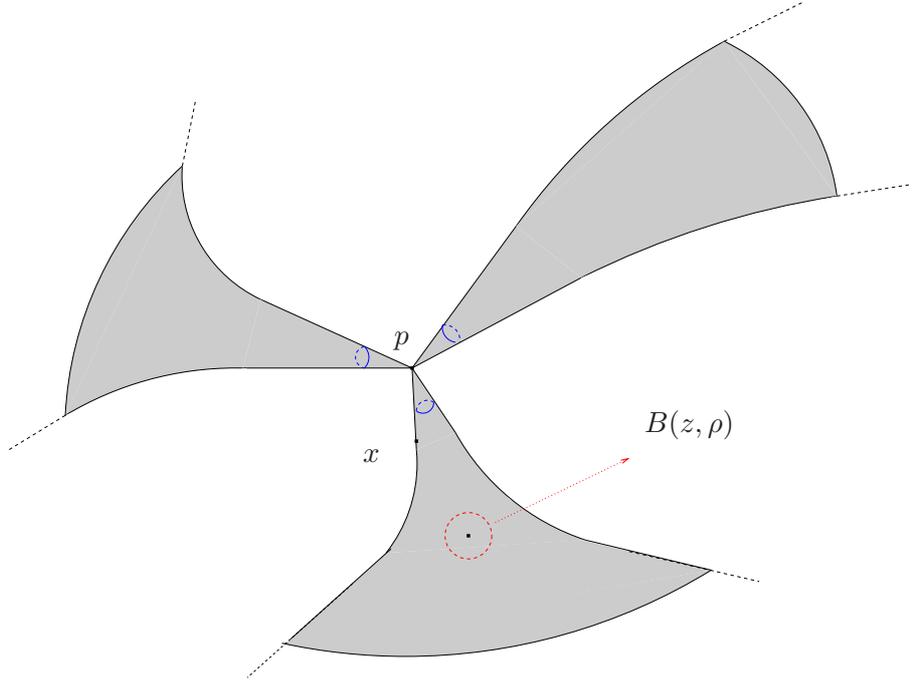}
 \caption{Small geodesic balls on $(X, g)$}
  \label{figure: smallball}
 \end{center}
 \end{figure}

 Fix $\varsigma$ and let $r$ be some number located in
 $(0, \frac{\varsigma}{10})$.

  Since $t_i \in [-\frac12, 0]$, $(X_i, g_i(t_i))$ satisfies energy concentration property:
 \begin{align*}
   \int_{B_{g_i(t_i)} (z, Q^{-\frac12})} |Rm|_{g_i(t_i)}^{\frac{m}{2}} d\mu_{g_i(t_i)} > \epsilon,
   \quad \textrm{whenever} \; Q=|Rm|_{g_i(t_i)}(z) > H.
 \end{align*}
 We have already assumed that the number of singularities is $1$,
 so energy concentration implies
 \begin{align*}
     |Rm|_{g_i(t_i)}(z) < r^{-2}, \quad \forall \;  z \in  X_i \backslash
     B_{g_i(t_i)}(p_i,2 r)
 \end{align*}
 where $p_i$ is a point with maximal $|Rm|_{g_i(t_i)}$.   Applying Perelman's
 pseudolocality Theorem, i.e., Theorem~\ref{theorem: pseudo}, we are able to control the Riemannian
 curvature on a short time period:
 \begin{align*}
      |Rm|_{g_i(t)}(z) \leq \digamma (r, \kappa),  \quad \forall \; (z,t) \in X_i
      \backslash B_{g_i(t_i)}(p_i, 3r) \times [t_i, t_i+ \daleth(r, \kappa)].
 \end{align*}
 Recall $t_i \to 0$, it follows that
 \begin{align}
      |Rm|_{g_i(t)}(z) \leq \digamma (r, \kappa),  \quad \forall \; (z,t) \in X_i
      \backslash B_{g_i(t_i)}(p_i, 3r) \times [t_i, 0]
 \label{eqn: lccd}
 \end{align}
 for large $i$.  This is a very important control in the following
 proof.

   Now we are ready to give a proof by separating geodesics. Without loss of generality, we can assume that
   $\displaystyle \lim_{i \to \infty} d_{g_i(t_i)}(x_i, y_i)$ exists.\\

\textit{Step 1.  \quad $ \displaystyle  \lim_{i \to \infty}
 d_{g_i(t_i)}(x_i, y_i) \leq \lim_{i \to \infty} \displaystyle
 d_{g_i(0)}(x_i,  y_i).$ }\\

 Let $\Gamma_i$ be a shortest geodesic connecting $x_i$ and $y_i$ in
 the Riemannian manifold $(X_i, g_i(0))$. We can separate it into
 two parts:
\begin{align*}
  \Gamma_i^{(a)} &= \Gamma_i \backslash B_{g_i(t_i)} (p_i, 3r);\\
  \Gamma_i^{(b)} &= \Gamma_i \cap B_{g_i(t_i)} (p_i, 3r).
\end{align*}

 \begin{figure}
 \begin{center}
 \psfrag{p}[c][c]{$x_i$}
 \psfrag{q}[c][c]{$y_i$}
 \psfrag{Ball}[c][c]{$B_{g_i(t_i)} (p_i, 3r)$}
 \psfrag{G}[c][c]{$\Gamma_i$}
 \psfrag{U}[c][c]{$\Upsilon_i$}
 \psfrag{t1}[c][c]{$t_i$}
 \psfrag{t2}[c][c]{$0$}
 \psfrag{t}[c][c]{$t$}
 \psfrag{A}[c][c]{$\Gamma_i$ is a geodesic under $g_i(0)$}
 \psfrag{B}[c][c]{$\Upsilon_i$ is a geodesic under $g_i(t_i)$}
 \includegraphics[width=0.8 \columnwidth]{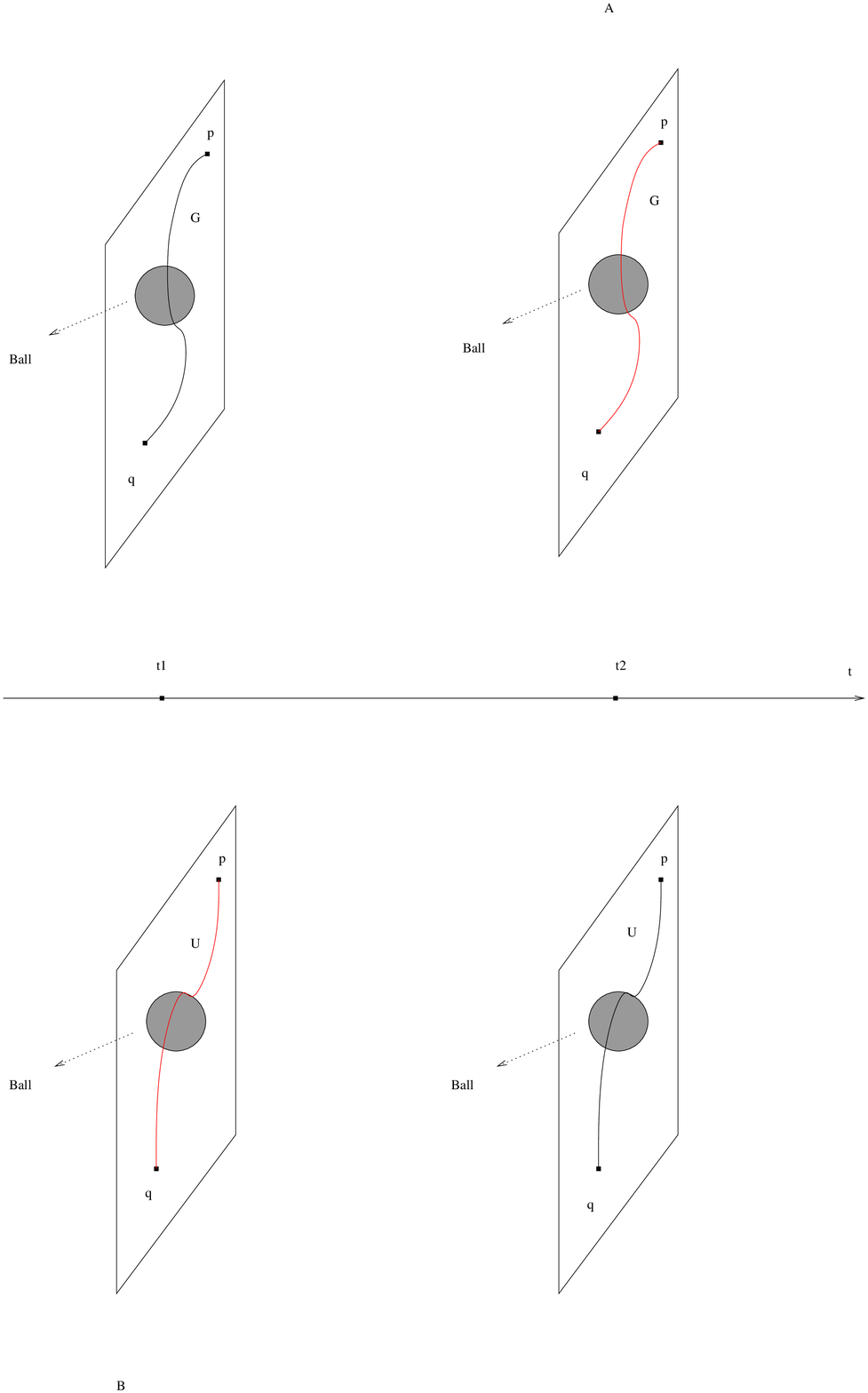}
 \caption{Choose geodesic at different time slices}
  \label{figure: choosegeodesic}
 \end{center}
 \end{figure}

 From inequality (\ref{eqn: lccd}), we see that  the Riemannian
 curvature is uniformly controlled by $\digamma(r, \kappa)$ on
  $\Gamma_i^{(a)} \times [t_i, 0]$.  Recall that $(X_i, g_i(t))$
 satisfies the equation $\displaystyle \D{g}{t}= -Ric +c_ig$. It
 follows that the changing speed of any curve's length is controlled
 by Ricci curvature's norm and $c_i$.  Note that $c_i \to 0$ and $|t_i| \to 0$ ,
 the limit reads
  \begin{align}
    \lim_{i \to \infty} \Length_{g_i(t_i)}(\Gamma_i^{(a)}) &= \lim_{i \to \infty}
    \Length_{g_i(0)}(\Gamma_i^{(a)}) \notag \\
    &\leq  \lim_{i \to \infty} \Length_{g_i(0)}(\Gamma_i)
    =\lim_{i \to \infty} d_{g_i(0)}(x_i, y_i).
 \label{eqn: gpdcontrol}
  \end{align}
 Triangle inequality implies
 \begin{align}
   d_{g_i(t_i)}(x_i, y_i) &\leq d_{g_i(t_i)}(x_i, \partial B_{g_i(t_i)} (p_i,
   3r)) \notag\\
   &\hspace{60pt}  + d_{g_i(t_i)}(\partial B_{g_i(t_i)} (p_i,
   3r), y_i) + \diam_{g_i(t_i)} B_{g_i(t_i)} (p_i, 3r) \notag\\
   &\leq \Length_{g_i(t_i)} \Gamma_i^{(a)} + 6r.
 \label{eqn: rdcontrol}
 \end{align}
 Combining inequality (\ref{eqn: gpdcontrol})
  and (\ref{eqn: rdcontrol}) yields
 \begin{align*}
   &\lim_{i \to \infty} d_{g_i(t_i)}(x_i, y_i)
    \leq \lim_{i \to \infty} d_{g_i(0)}(x_i, y_i) + 6r
    <  \lim_{i \to \infty} d_{g_i(0)}(x_i, y_i) + \varsigma.
 \end{align*}
 As $\varsigma$ can be as small as we want, this implies
  $ \displaystyle  \lim_{i \to \infty} d_{g_i(t_i)}(x_i, y_i) \leq \lim_{i \to \infty}
  d_{g_i(0)}(x_i, y_i)$.\\

\textit{Step 2.  \quad $ \displaystyle \lim_{i \to \infty}
\displaystyle d_{g_i(0)}(x_i,
 y_i) \leq \lim_{i \to \infty}  d_{g_i(t_i)}(x_i, y_i).$  }

 If $\displaystyle \lim_{i \to \infty} d_{g_i(t_i)}(x_i,
y_i)=\infty$, the statement is true automatically.  So we assume
\linebreak $\displaystyle \lim_{i \to \infty} d_{g_i(t_i)}(x_i,
y_i)=D<\infty$.

  Let $\Upsilon_i$ be a
shortest geodesic connecting $x_i$ and $y_i$ under the metric
$g_i(t_i)$. We separate it into two parts:
 \begin{align*}
 \Upsilon_i^{(a)}= \Upsilon_i \backslash B_{g_i(t_i)}(p_i,
 3r),   \quad   \Upsilon_i^{(b)}= \Upsilon_i \cap B_{g_i(t_i)}(p_i,
 3r).
 \end{align*}

 Curvature control inequality (\ref{eqn: lccd}) implies that $\Upsilon_i^{(a)}$
 stay in a part with bounded curvature, so the changing speed of its length is controlled.
 Recall $|t_i| \to 0$ and $c_i \to 0$, we have
 $\displaystyle
  \lim_{i \to \infty} \frac{\Length_{g_i(t_i)} \Upsilon_i^{(a)}}{\Length_{g_i(0)}
  \Upsilon_i^{(a)}}=1$.
 It follows that
 \begin{align}
    \lim_{i \to \infty} \Length_{g_i(0)}(\Upsilon_i^{(a)}) = \lim_{i \to \infty}
    \Length_{g_i(t_i)}(\Upsilon_i^{(a)})
    \leq  \lim_{i \to \infty} \Length_{g_i(t_i)}(\Upsilon_i)
      \leq D.
    \label{eqn: uacontrol}
 \end{align}

  Now let's  study the other part $\Upsilon_i^{(b)} = \Upsilon \cap
 B_{g_i(t_i)}(p_i, 3r)$. In this part, Riemannian curvature's control
 is missing. Moreover, $\Upsilon_i^{(b)}$ may be far away from a
 geodesic at time $t=0$. So we give up the direct control of
 $\Upsilon_i^{(b)}$. Instead, we put it in the  ball
 $B_{g_i(t_i)}(p_i, 3r)$ and control the diameter of this  ball by
 its volume and the diameter of its boundary.

  If $d_{g_i(t_i)}(x_i, p_i) \to \infty$, then
  $\Upsilon_i^{(b)}=\Upsilon_i \cap B_{g_i(t_i)}(p_i, 3r)
  =\emptyset$ and consequently $\Upsilon_i = \Upsilon_i^{(a)}$. Using inequality (\ref{eqn: uacontrol}),
   we have already
  finished the proof.   Therefore we can assume that $p_i$ stay
  within a bounded distance away from $x_i$ and $p_i \to p$. Then we have
  convergence
 \begin{align*}
    & B_{g_i(t_i)} (p_i, 3r)  \stackrel{C^{1,\gamma}}{\longrightarrow}
     B(p, 3r) \subset X, \\
    & \partial B_{g_i(t_i)} (p_i, 3r)  \stackrel{C^{1,\gamma}}{\longrightarrow}
     \partial B(p, 3r).
 \end{align*}
 Note that $p$ is a fixed point now. According to this point,
 we adjust $r \in (0, \frac{\varsigma}{10})$ such that
 \begin{itemize}
 \item   $\partial B_{g}(y, 3r)$  has at most
 $\frac{2^mK}{\kappa}$ components.
 \item    Every component of $\partial B_{g}(y, 3r)$
 is a smooth $(m-1)$-dimensional Riemannian manifold whose diameter is
 less than $4\pi r$.
 \end{itemize}
 Therefore, for every large $i$, $\partial B_{g_i(t_i)} (p_i, 3r)$
 has at most $\frac{2^mK}{\kappa}$ connected  components. For convenience, we
 denote $S_i^{(1)}, \cdots, S_i^{(k)}$ as the connected components
 of $\partial B_{g_i(t_i)} (p_i, 3r)$.  Clearly, for every $l \in \{1, \cdots, k\}$,
 $S_i^{(l)}$ is at least a $C^1$-manifold  and  $\diam_{g_i(t_i)} S_i^{(l)} <6\pi r$.
 Here $\diam_{g_i(t_i)} S_i^{(l)}$ means the intrinsic diameter of $S_i^{(l)}$
 under the metric $g_i(t_i)|_{S_i^{(l)}}$.

 Note that on $\partial B_{g_i(t_i)}(p_i, 3r)$, Riemannian curvature
 is uniformly bounded (inequality (\ref{eqn: lccd})).  Therefore the changing speed of any curve's length is
 uniformly controlled on $\partial B_{g_i(t_i)}(p_i, 3r) \times [t_i,
 0]$. As $|t_i| \to 0$, it yields that
 \begin{align*}
   \lim_{i \to \infty}
   \frac{\diam_{g_i(0)}(S_i^{(l)})}{\diam_{g_i(t_i)}(S_i^{(l)})} =1.
 \end{align*}
 for every $l \in \{1, \cdots, k\}$. Therefore, we have $\displaystyle
     \lim_{i \to \infty} \diam_{g_i(0)}(S_i^{(l)}) < 6\pi r. $

  The upper volume ratio bound reads
  \begin{align}
  \Vol_{g_i(t_i)} (B_{g_i(t_i)} (p_i, 3r)) \leq K \cdot (3r)^m.
  \label{eqn: vp}
  \end{align}
  Since the flow $\{(X_i, g_i(t)), -1 \leq t \leq 1\}$ is a solution
  of $\D{g}{t}= -Ric + c_ig$, the volume changing speed is controlled  by scalar curvature
  and $c_i$. Both of these two terms are tending to zero. Combining this
  estimate with $|t_i| \to 0$, inequality (\ref{eqn: vp}) implies
  \begin{align*}
  \Vol_{g_i(0)} (B_{g_i(t_i)}(p_i, 3r))  <K \cdot (6r)^m.
  \end{align*}
  for large $i$.

  In short, for $i$ large, we have two estimates
 \begin{align}
  \Vol_{g_i(0)} (B_{g_i(t_i)} (p_i, 3r))  < K \cdot (6r)^m,  \quad
    \diam_{g_i(0)} S_i^{(l)} < 6\pi r, \quad \forall l \in \{1, \cdots, k\}.
 \label{eqn: dv}
 \end{align}

 On a $\kappa$-noncollapsed Riemannian manifold,
 we claim that the diameter of every connected set can be controlled by its volume and
 the diameters of components of its boundary.  In order not to diversify our attention,
 we only list a precise statement here and postpone the proof to
 the appendix(Lemma~\ref{lemma: cvcd}).

 \begin{lm}
    $(X^m,g)$ is a Riemannian manifold which is
    $\kappa$-noncollapsed on scale $1$.
   $B \subset X$ is a connected open set and $\Vol(B) < \kappa$.
   $S_1, \cdots,  S_{\Lambda}$ are connected components of $\partial B$. Then
 \begin{align*}
  \diam B \leq \sum_{k=1}^{{\Lambda}} \diam S_k  +  6{\Lambda} ( \frac{Vol(B)}{\kappa})^{\frac{1}{m}}
 \end{align*}
 where $\diam S_k$ means the diameter of the manifold $(S_k, g|_{S_k})$.
 \end{lm}

 In order to apply this Lemma, we let $B=B_{g_i(t_i)}(p_i, 3r)$, $S_l = S_i^{(l)}$.
 Inequality (\ref{eqn: dv}) implies that under metric $g_i(0)$, we have
 $\Lambda \leq \frac{2^mK}{\kappa}$, $\frac{Vol(B)}{\kappa} \leq \frac{K}{\kappa}$, $\diam S_l < 6 \pi r$.

  So  the previous Lemma applies and yields
\begin{align}
    \diam_{g_i(0)} (B_{g_i(t_i)}(p_i, 3r)) < 6 \cdot \frac{2^mK}{\kappa} (\pi  +
    (\frac{K}{\kappa})^{\frac{1}{m}})r
      \triangleq C(\kappa, K)r < C(\kappa, K) \varsigma.
 \label{eqn: bdcontrol}
\end{align}

 By triangle inequality, we have
 \begin{align*}
   d_{g_i(0)}(x_i, y_i) & \leq d_{g_i(0)}(x_i, \partial B_{g_i(t_i)}(p_i,
   3r)) + d_{g_i(0)}(\partial B_{g_i(t_i)}(p_i, 3r), y_i) + \diam_{g_i(0)} B_{g_i(t_i)}(p_i, 3r)\\
   &\leq  \Length_{g_i(0)} \Upsilon_i^{(a)}   +
   \diam_{g_i(0)} B_{g_i(t_i)} (p_i, 3r).
 \end{align*}
 Combining this with inequality (\ref{eqn: uacontrol}) and  (\ref{eqn: bdcontrol}),
 we obtain
 \begin{align*}
    \lim_{i \to \infty} d_{g_i(0)}(x_i, y_i) \leq D +
    C(\kappa, K) \varsigma.
 \end{align*}
 Since $\varsigma$ can be chosen arbitrarily small, it follows that
  \begin{align*}
   \displaystyle \lim_{i \to \infty} d_{g_i(0)}(x_i, y_i) \leq D =
   \displaystyle \lim_{i \to \infty} d_{g_i(t_i)}(x_i, y_i).
  \end{align*}

 \end{proof}

 As a simple application of Lemma~\ref{lemma: dcontinuity}, we have
 the following Corollary.

 \begin{corollary}[\textbf{Coincidence of Limit Spaces}]
 If $\{(X_i, g_i(t)), -1 \leq t \leq 1\}$ is an EV-refined
 sequence, $x_i \in X_i$,  $t_i \leq 0$ and $\displaystyle \lim_{i \to \infty} t_i=0$.   Then
 \begin{align*}
       \lim_{i \to \infty} (X_i, x_i, g_i(t_i)) = \lim_{i \to
       \infty} (X_i, x_i, g_i(0)).
 \end{align*}
 \label{corollary: colimit}
 \end{corollary}

 \begin{proposition}[\textbf{Point-selecting}]
  Suppose $\{(X_i^m, x_i, g_i(t)), -1 \leq t \leq 1\}$ is an
  EV-refined sequence,  $\displaystyle \lim_{i \to \infty} |Rm|_{g_i(0)}(x_i)= \infty$.
  By taking subsequence if necessary, there exist
  points $(p_i, t_i)$ satisfying the following properties.
  \begin{itemize}
  \item $t_i \leq 0$ and $\displaystyle \lim_{i \to \infty} t_i =0$.
  \item $\displaystyle \lim_{i \to \infty} d_{g_i(t_i)}(p_i, x_i)=0$.
  \item $|Rm|_{g_i(t)}(x) \leq 4 Q_i$ whenever
    $(x, t) \in B_{g_i(t_i)}(p_i, i \rho_b Q_i^{-\frac12}) \times [t_i-\frac{1}{1000m}Q_i^{-1},
    t_i]$. Here $Q_i=|Rm|_{g_i(t_i)}(p_i)$, $\rho_b$ is the base
    radius defined in Definition~\ref{definition: base}.
  \end{itemize}
 \label{proposition: pickpoint}
 \end{proposition}
 \begin{proof}
   Apply Lemma~\ref{lemma: choosepoints} and take diagonal
   sequence.
 \end{proof}

 \begin{lemma}[\textbf{Added Point Doesn't Emerge Suddenly}]
 Suppose $\{(X_i, x_i, g_i(t)), -1 \leq t \leq 1\}$ is an EV-refined
  sequence, $t_i \leq 0$ and $\displaystyle \lim_{i \to \infty}
  t_i=0$,   $(X, x, g)$ is the limit
  space of $(X_i, x_i, g_i(0))$ and $(X_i, x_i, g_i(t_i))$. Then the
  following property holds.

    If $p \in X$ is an added point as $(X_i, x_i, g_i(0))$ converges , then $p$
  is an added point as $(X_i, x_i, g_i(t_i))$ converges.
  \label{lemma: noemerge}
 \end{lemma}
 \begin{proof}
   Application of Pseudolocality theorem.
 \end{proof}

  \begin{figure}
  \begin{center}
  \psfrag{p}{$p$}
  \psfrag{q}{$q$}
  \includegraphics[width= 0.8 \columnwidth]{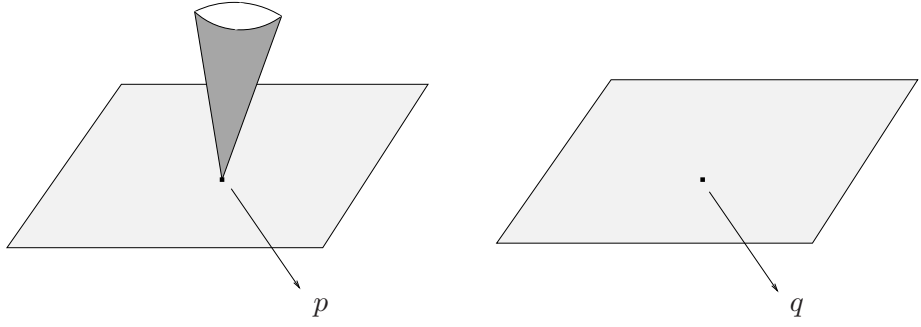}
  \caption{Impossible added points}
  \label{figure: imsingularity}
  \end{center}
  \end{figure}

\begin{lemma}[\textbf{Added Point Cannot Sit in a Smooth Part}]
  $\{(X_i, x_i, g_i(t)), -1 \leq t \leq 1\}$ is an EV-refined
  sequence, $(X, x, g)$ is the limit space
  of $(X_i, x_i, g_i(0))$.

  If $p$ is an added point of $X$, then $p$ cannot stay on smooth part, i.e.,
  for every small $\eta>0$, each component of $B(p, \eta) \backslash \{p\}$ is diffeomorphic
   to a cone over $S^{m-1}/ \Gamma$ for some nontrivial group
   $\Gamma$.

  In particular, every added point is non-smooth.
  \label{lemma: sitsmooth}
 \end{lemma}

  \begin{proof}
     Suppose not. There is a point $p \in X$ which is an added point for sequence
    $\{(X_i, x_i, g_i(0))\}$  and it sits in some smooth part.
    According to Proposition~\ref{proposition: pickpoint}, there
    exists spacetime points $(p_i, t_i)$ satisfying the following
    properties.
    \begin{itemize}
    \item $(p_i, t_i)$ is a parabolic local maximal point, i.e.,
      \begin{align*}
         |Rm|_{g_i(t)}(x) \leq 4 |Rm|_{g_i(t_i)}(p_i) \triangleq
         4Q_i,
      \end{align*}
      whenever
      $(x, t) \in B_{g_i(t_i)}(p_i,  i \rho_b Q_i^{-\frac12}) \times [t_i-\frac{1}{1000m}Q_i^{-1}, t_i]$.
    \item $p_i \to p$ as $(X_i, x_i, g_i(t_i))$ converges to $(X, x,
    g)$.
    \end{itemize}
  Shifting our original sequence in both space and time
  direction, we have
  \begin{align*}
     (X_i, p_i, g_i(t_i)) \stackrel{C^{1, \gamma}}{\rightarrow} (X, p,
     g).
  \end{align*}
 Since $p$ sits in some smooth part of $X$,  there
 exists a small number $\eta$ such that  the area of the largest component of
 $\partial B(p, r)$
 is strictly greater than $\frac78 m\omega(m) r^{m-1}$
 whenever $r \in (0, 2\eta)$.  Furthermore, by the finiteness of added point, we can choose
 $\eta$ small enough such that $p$ is the unique added point in $B(p, 2\eta)$ for the
 sequence $(X_i, x_i, g_i(t_i))$.   It follows that for every fixed number
 $r \in (0, \eta)$ and large $i$,   we have
   \begin{align}
    \frac{\Area_{g_i(t_i)}(S_i(r))}{r^{m-1}}
\frac78 \omega(m)
   \label{eqn: vlow}
   \end{align}
 whenever $r \in (0, \eta)$. Here we use $S_i(r)$ to denote one of
 the largest component of $\partial B_{g_i(t_i)}(p_i, r)$.
  Note that $\partial B_{g_i(t_i)}(q_i, \frac12 \rho_b Q_i^{-\frac12})$
 is connected and
   \begin{align*}
      \lim_{i \to \infty} \frac{\Area_{g_i(t_i)}(\partial B_{g_i(t_i)}
      (q_i, \frac12 \rho_b Q_i^{-\frac12}))}{(\frac12 \rho_b Q_i^{-\frac12})^{m-1}}  \leq \frac34 \omega(m) < \frac78 \omega(m).
   \end{align*}
   We can choose $r_i < \eta$ to be the largest radius such that the
   largest component of $\partial B_{g_i(t_i)}(p_i, r_i)$ has area
   ratio $\frac78 m\omega(m)$, i.e.,  the area of the largest
   component is  $\frac78 m\omega(m) r_i^{m-1}$.    Inequality (\ref{eqn: vlow})
   implies that $r_i \to 0$.
   Blowup the EV-refined sequence $\{(X_i, p_i, g_i(t)), -1 \leq t \leq 1\}$ by scale
   $r_i^{-2}$. In other words, let $g_i^{(1)}(t) = r_i^{-2}g_i(r_i^2 t + t_i)$
   and rename $p_i$ as $x_i^{(1)}$,   we  can extract an
   EV-refined sequence $\{(X_i, x_i^{(1)}, g_i^{(1)}(t)), -1 \leq t \leq  1\}$.
   Take limit for central time slice, we have
   \begin{align*}
     (X_i, x_i^{(1)}, g_i^{(1)}(0)) \stackrel{C^{1,\gamma}}{\rightarrow}
     (X^{(1)}, x^{(1)}, g^{(1)}).
   \end{align*}
  The limit space $(X^{(1)}, x^{(1)}, g^{(1)})$ is a Ricci-flat ALE
   multifold with total energy bounded by $E$.
   The space $X^{(1)}$ can be detached into union of orbifolds.
   The point  $x^{(1)}$ may be decomposed into several points, we
   say $x^{(1)}$ is in an orbifold if one of the decomposed points is in that
   orbifold.

  When $(X_i, p_i, g_i^{(1)}(0))$ converges, away from $x^{(1)}$,  check if there is
  any other added point staying in some smooth part.   If no, we
  stop. Otherwise, we choose such a point and denote it as
  $p^{(1)}$.  By shifting time and space again, we have
  \begin{align*}
     (X_i, p_i^{(1)}, g_i^{(1)}(t_i^{(1)})) \to  (X^{(1)}, p^{(1)},
     g^{(1)}).
  \end{align*}
  where $(p_i^{(1)}, t_i^{(1)})$ are parabolic local maximal points
  and $t_i^{(1)} \to 0^{-}$.

  As $p^{(1)}$ stays on a smooth component. Centered at
  this point, every small geodesic sphere's largest component has
  almost Euclidean area ratio.  As before, we can choose scale $r_i^{(1)}$
  to be locally the largest radius such that the largest component
  of  $\partial B_{g_i^{(1)}(t_i^{(1)})}(p_i^{(1)}, r_i^{(1)})$ has
  area ratio $\frac78 \omega(m)$.  Clearly, $r_i^{(1)} \to 0$. Blow up
  $\{(X_i, p_i^{(1)}, g_i^{(1)}(t)), -1 \leq t \leq 1\}$
  at point $(p_i^{(1)}, t_i^{(1)})$ with scale $(r_i^{(1)})^{-2}$, and rename the base point
  $p_i^{(1)}$ as $x_i^{(2)}$, we obtain a new EV-refined sequence
  $\{(X_i, x_i^{(2)}, g_i^{(2)}(t)), -1 \leq t \leq 1\}$.  Let the central time slices
  tend to $(X^{(2)}, x^{(2)}, g^{(2)})$. Then iterate the previous process.
  Figure~\ref{figure: tree1} shows such a process which stop in
  three steps.

 \setcounter{claim}{0}

 \begin{claim}
    For every $k \geq 1$, $x^{(k)}$ is an added point for
    the sequence $(X_i, x_i^{(k)}, g_i^{(k)}(0))$.
 \label{claim: added}
 \end{claim}
    We only need to show $\displaystyle \limsup_{i \to \infty}
    |Rm|_{g_i^{(k)}(0)}(x_i^{(k)})=\infty$, i.e.,
   $ \displaystyle  \limsup_{i \to \infty} r_i^{(k-1)} Q_i^{\frac12} = \infty$,
 where $Q_i=|Rm|_{g_i^{(k-1)}(t_i^{(k-1)})}(p_i^{(k-1)})$.
 Actually, as $(p_i^{(k-1)}, t_i^{(k-1)})$ are local maximal points,
 the deepest bubble blownup from these points is a non-flat ALE
 manifold which has nontrivial cone end at infinity.
 Therefore, for every fixed big number $\mathcal{R}>0$,  under the metric
  $g_i^{(k-1)}(t_i^{(k-1)})$, we have
 \begin{align*}
    \frac{\Area(\partial B(p_i^{(k-1)}, \mathcal{R} Q_i^{-\frac12}))}{(
    \mathcal{R}Q_i^{-\frac12})^{m-1}} < \frac78 \omega(m)
 \end{align*}
 for large $i$. The definition of $r_i^{(k-1)}$ implies $r_i^{(k)} > \mathcal{R}
 Q_i^{-\frac12}$. It follows that
 \begin{align*}
      \limsup_{i \to \infty} r_i^{(k)} Q_i^{\frac12} > \mathcal{R}
 \end{align*}
 for every $\mathcal{R}$. Therefore, we have
 $\displaystyle \limsup_{i \to \infty} r_i^{(k)}
 Q_i^{\frac12}=\infty$.  This finishes the proof of Claim~\ref{claim: added}.\\

 If $\theta_i$ is a time sequence tending to zero, we know
 \begin{align*}
    \lim_{i \to \infty} (X_i, x_i^{(k)}, g_i^{(k)}(\theta_i))
    = \lim_{i \to \infty} (X_i, x_i^{(k)}, g_i^{(k)}(0))
    =(X^{(k)}, x^{(k)}, g^{(k)}).
 \end{align*}
 We denote the energy of sequence
  $\displaystyle \{(X_i, x_i^{(k)}, g_i^{(k)}(\theta_i))\}_{k=1}^{\infty}$
  as $\mathcal{E}(X^{(k)}, x_i^{(k)}, {g_i^{(k)}(\theta_i)})$. It is defined
  as
  \begin{align*}
    \mathcal{E}(X^{(k)}, x_i^{(k)}, {g_i^{(k)}(\theta_i)})
    \triangleq
   \lim_{\mathcal{R} \to \infty} \limsup_{i \to \infty} \int_{B_{g_i^{(k)}(\theta_i)}(x_i^{(k)}, \mathcal{R})} |Rm|^{\frac{m}{2}}
   d\mu.
  \end{align*}

 \begin{claim}
  For every $k \geq 1$, we have energy control
   \begin{align}
   \mathcal{E}(X^{(k)}, x_i^{(k)}, {g_i^{(k)}(0)})
   \leq E -(k-1)\epsilon.
  \label{eqn: energyinduction}
  \end{align}
 \label{claim: econtrol}
 \end{claim}

 Actually, from the blowup process, we know
  \begin{align*}
   \mathcal{E}(X^{(k)}, x_i^{(k)}, {g_i^{(k)}(0)})
   \leq
   \limsup_{i \to \infty} \int_{B_{g_i^{(k-1)}(t_i^{(k-1)})}(p_i^{(k-1)},
   \delta)}
   |Rm|^{\frac{m}{2}}d\mu
  \end{align*}
  for every small $\delta$.  Under the metric $g_i^{(k-1)}(t_i^{(k-1)})$,
  $B(p_i^{(k-1)}, \delta)$ and $B(x_i^{(k-1)}, \delta)$ are disjoint.
  As $x^{(k-1)}$ is an added point for the convergence
  process
  \begin{align*}
  (X_i, x_i^{(k-1)}, g_i^{(k-1)}(0)) \to (X^{(k-1)}, x^{(k-1)},
  g^{(k-1)})
  \end{align*}
 and $t_i^{(k-1)} \to 0^{-}$, Lemma~\ref{lemma: noemerge} tells us that
  $x^{(k-1)}$ is an
  added point for the convergence process
  \begin{align*}
  (X_i, x_i^{(k-1)}, g_i^{(k-1)}(t_i^{(k-1)})) \to (X^{(k-1)}, x^{(k-1)},
  g^{(k-1)}).
  \end{align*}
  This implies that under the metric $g_i^{(k-1)}(t_i^{(k-1)})$, the energy contained
  in $B(x_i^{(k-1)}, \delta)$ is greater than $\epsilon$.
  Therefore, fix an $\mathcal{R}$ large enough, we have
  \begin{align*}
   \mathcal{E}(X^{(k)}, x_i^{(k)}, g_i^{(k)}(0)) & \leq \int_{B(x_i^{(k-1)},
   \mathcal{R})} |Rm|^{\frac{m}{2}} d\mu -\epsilon\\
   &\leq \mathcal{E}(X^{(k-1)}, x_i^{(k-1)},
   g_i^{(k-1)}(t_i^{(k-1)})) -\epsilon.
  \end{align*}
  Note that $\theta_i^{(k-2)}=(r_i^{(k-1)})^2t_i^{(k-1)} + t_i^{(k-2)} \to
  0$, similar argument shows that
  \begin{align*}
    \mathcal{E}(X^{(k-1)}, x_i^{(k-1)},  g_i^{(k-1)}(t_i^{(k-1)}))
    \leq \mathcal{E}(X^{(k-2)}, x_i^{(k-2)}, g_i^{(k-2)}(\theta_i^{(k-2)})) -\epsilon.
  \end{align*}
  Therefore, induction implies
  \begin{align*}
     \mathcal{E}(X^{(k)}, x_i^{(k)}, g_i^{(k)}(0))
     \leq
     \mathcal{E}(X^{(1)}, x_i^{(1)}, g_i^{(1)}(\theta_i)) - (k-1) \epsilon
     \leq E- (k-1) \epsilon,
  \end{align*}
  where $\displaystyle \theta_i^{(1)} = t_i^{(1)} + \sum_{a=2}^{k-1} t_i^{(a)}(r_i^{(2)} \cdots
  r_i^{(a)})^2  \longrightarrow 0$.
  This finishes the proof of Claim~\ref{claim: econtrol}.

 As a corollary of Claim~\ref{claim: econtrol}, we have the following
 fact:

\textit{The blowup process terminates in finite steps.}\\

  Suppose it stops at step $k$. Denote
  $(\tilde{X}, \tilde{x}, \tilde{g})= (X^{(k)}, x^{(k)},
  g^{(k)})$, then we know every added point away from $\tilde{x}$
  can only stay on singular part.

  \begin{claim}
   $\tilde{X}$ contains at least one Euclidean component.
  \label{claim: euclidean}
  \end{claim}

  As $\tilde{X}$ is a noncompact ALE space. Choose $\mathcal{R}$ big
  enough, we see that each connected component of  $\partial B(\tilde{x}, \mathcal{R})$ is
  diffeomorphic to some topological space form. In particular, we can assume one largest component
  is diffeomorphic to $S^{m-1}/ \Gamma$ and it's area ratio is very close to $\frac{1}{|\Gamma|} \omega(m)$.
  According to our choice of blowup scales $r_i^{(k-1)}$ and the fact $\mathcal{R} \gg 1$,  we
  see that $\frac{1}{|\Gamma|} \omega(m) \geq \frac78 \omega(m)$.
  This forces that $|\Gamma|=1$ and $\Gamma$ is trivial.   This
  means that there is at least one Euclidean end at infinity.
  Detach $\tilde{X}$ into orbifolds.  There is exactly one
  noncompact orbifold corresponding to the Euclidean end we have
  chosen. Denote this orbifold as $M$. So $M$ is a Ricci-flat
  orbifold with Euclidean volume growth rate at infinity.
  Bishop volume comparison theorem implies $M$ must be Euclidean
  space. So Claim~\ref{claim: euclidean} is proved.\\

  By the connectedness of $\tilde{X}$, there must be some added
  point on the Euclidean component $M$. According to our choice, such
  point must be $\tilde{x}$.   Decompose it as
  $\tilde{x}_1, \cdots, \tilde{x}_N \in M$.  So the intersection of the geodesic sphere $\partial B(\tilde{x}, 1)$
  and $M$ is exactly the ``unit geodesic sphere" centered at the set
  $\{\tilde{x}_1, \cdots, \tilde{x}_N\}$.  As $M$ is Euclidean
  space,  by generalized Bishop volume comparison theorem, we see
  that each component of this ``unit geodesic sphere" has an area not
  less than $m\omega(m)$.   It follows that the largest component of
  $\partial B(\tilde{x}, 1)$ must has an area strictly greater than
  $\frac78 m\omega(m)$.    On the other hand, from the generating
  process of $\tilde{X}$, we know the largest component of
  $\partial B(\tilde{x}, 1)$ has exactly area $\frac78\omega(m)$.
  Therefore we have the contradiction to set up the proof!
  \end{proof}

 \begin{figure}
 \begin{center}
 \psfrag{x}[c][c]{$(X_i, x_i, g_i(0))$}
 \psfrag{xi}[c][c]{$(X, x, g)$}
 \psfrag{x1}[c][c]{$(X_i^{(1)}, x_i^{(1)}, g_i^{(1)}(0))$}
 \psfrag{xi1}[c][c]{$(X^{(1)}, x^{(1)}, g^{(1)})$}
 \psfrag{x2}[c][c]{$(X_i^{(2)}, x_i^{(2)}, g_i^{(2)}(0))$}
 \psfrag{xi2}[c][c]{$(X^{(2)}, x^{(2)}, g^{(2)})$}
 \psfrag{x3}[c][c]{$(X_i^{(3)}, x_i^{(3)}, g_i^{(3)}(0))$}
 \psfrag{xi3}[c][c]{$(X^{(3)}, x^{(3)}, g^{(3)})$}
 \psfrag{p}[c][c]{$(X_i, p_i, g_i(t_i))$}
 \psfrag{pi}[c][c]{$(X, p, g)$}
 \psfrag{p1}[c][c]{$(X_i^{(1)}, p_i^{(1)}, g_i^{(1)}(t_i^{(1)}))$}
 \psfrag{pi1}[c][c]{$(X^{(1)}, p^{(1)}, g^{(1)})$}
 \psfrag{p2}[c][c]{$(X_i^{(2)}, p_i^{(2)}, g_i^{(2)}(t_i^{(2)}))$}
 \psfrag{pi2}[c][c]{$(X^{(2)}, p^{(2)}, g^{(2)})$}
 \psfrag{c1g}[f][f]{$C^{1, \gamma}$}
 \psfrag{s}[c][c]{\textcolor{blue}{Shifting}}
 \psfrag{i}[c][c]{\textcolor{blue}{in spacetime}}
 \psfrag{b1}[c][c]{\textcolor{red}{Blowup by $(r_i^{(1)})^{-2}$}}
 \psfrag{b2}[c][c]{\textcolor{red}{Blowup by $(r_i^{(2)})^{-2}$}}
 \psfrag{b3}[c][c]{\textcolor{red}{Blowup by $(r_i^{(3)})^{-2}$}}
 \psfrag{r1}[c][c]{\textcolor{red}{$x_i^{(1)} \triangleq p_i$}}
 \psfrag{r2}[c][c]{\textcolor{red}{$x_i^{(2)} \triangleq p_i^{(1)}$}}
 \psfrag{r3}[c][c]{\textcolor{red}{$x_i^{(3)} \triangleq p_i^{(3)}$}}
 \psfrag{final}[c][c]{Final Bubble}
 \psfrag{define}[c][c]{is defined to be}
 \psfrag{tx}[c][c]{$(\tilde{X}, \tilde{x}, \tilde{g})$}
 \includegraphics[width=0.8 \columnwidth ]{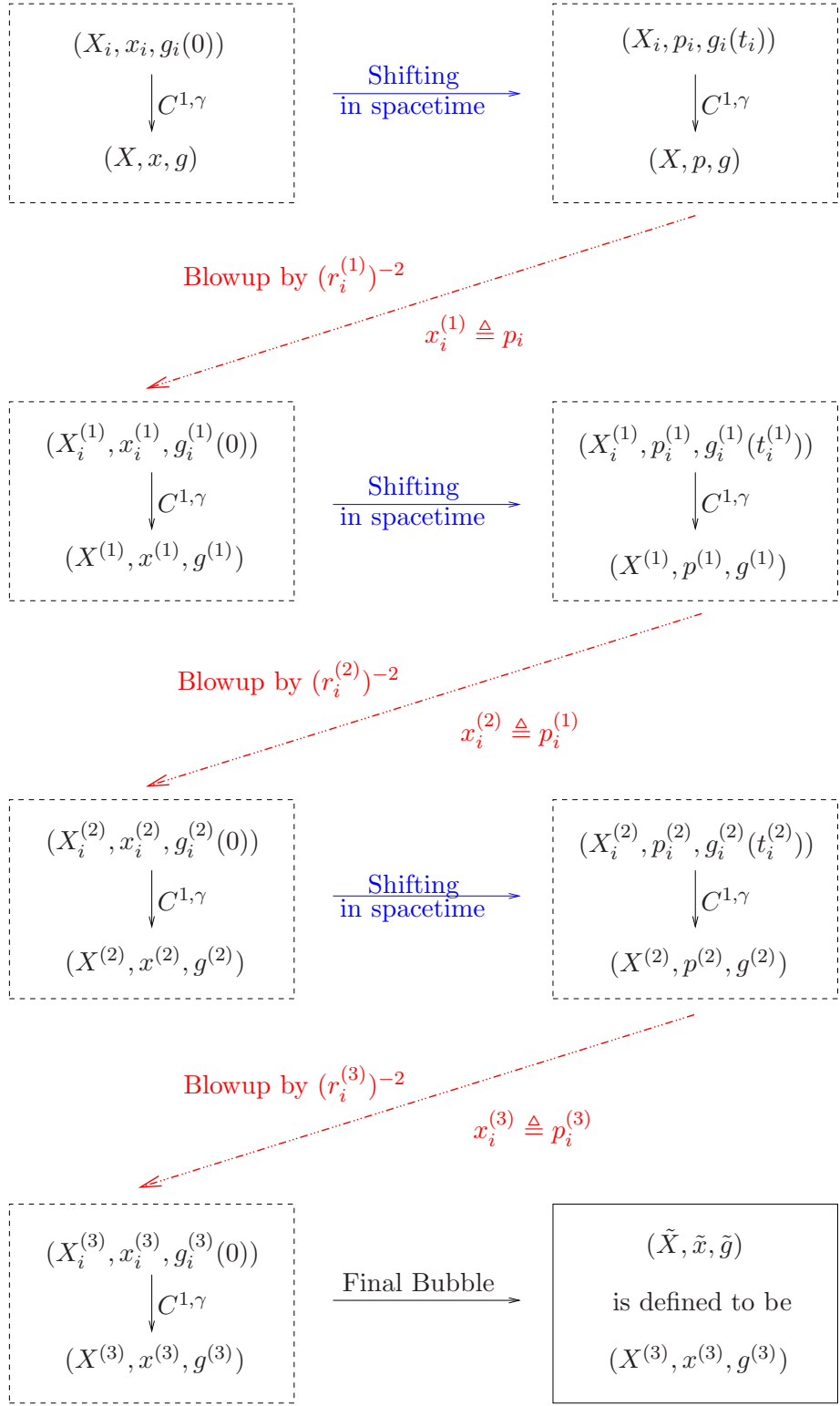}
 \end{center}
 \caption{The blowup process according to area ratio}
 \label{figure: tree1}
 \end{figure}

 As an application of Lemma~\ref{lemma: sitsmooth}, we have the
 following corollary.

 \begin{corollary}[\textbf{Singularities Depend only on the Limit Space}]
  Suppose that $\{(X_i, x_i, g_i(t)), -1 \leq t \leq 1\}$ is an EV-refined
  sequence, $\displaystyle \lim_{i \to \infty}
  t_i=0$,   $(X, x, g)$ is the common limit
  space of $(X_i, x_i, g_i(0))$ and $(X_i, x_i, g_i(t_i))$. Then the
  following property holds.

   $p \in X$ is an added point as $(X_i, x_i, g_i(t_i))$ converges if and only if $p$
  is an added point as $(X_i, x_i, g_i(0))$ converges.
  \label{corollary: cadded}
 \end{corollary}

 \begin{proof}
  Lemma~\ref{lemma: sitsmooth} implies that every added point in the limit process must be really a singular point
 in the limit space. However, every singular point in the limit space must be an added point in the limit process
 by its construction.  So we have the following equivalent conditions.
 \begin{align*}
 &\quad \; \; p \;\textrm{is an added point as} \; (X_i, x_i, g_i(0))
  \; \textrm{converges to} \; (X, x,
  g).\\
 &\Leftrightarrow p \; \textrm{is a singular point of} \; X. \\
 &\Leftrightarrow  \; p \; \textrm{is an added point as} \; (X_i, x_i, g_i(t_i))
 \; \textrm{converges to} \; (X, x,
 g).
 \end{align*}
 \end{proof}

 \textit{Therefore, whether point $p$ is an added point  in the limit
 process depends only on the failing of smoothness of $p$ on the
 limit space, it has nothing to do with the choice of the limit sequence.}\\

 By virtue of this property, we can reverse Perelman's
 pseudolocality theorem in our special setting.

 \begin{lemma}[\textbf{Backward Pseudolocality for an EV-refined Sequence}]
 If an EV-refined sequence $\{ (X_i, x_i, g_i(t)), -1 \leq t \leq 1\}$
 satisfies $\displaystyle        \sup_{B_{g_i(0)}(x_i, r)}
 |Rm|_{g_i(0)} \leq r^{-2}$ for some constant $r \in (0,1]$,    then
 there exists a small positive constant $\delta$ (depend on this
 sequence) such that
 \begin{align*}
     |Rm|_{g_i(t)}(x) \leq  \delta^{-2}r^{-2}, \quad \textrm{whenever}  \;
     d_{g_i(t)}(x_i,x) \leq
     \delta r, -\delta^2 r^2 \leq t \leq 0
 \end{align*}
 for large $i$.
  \label{lemma: prermcontrolover}
 \end{lemma}

 \begin{proof}
 Note that rescaling an EV-refined sequence by constant $\frac{1}{r^2}$ yields
  a new EV-refined sequence. Therefore, we can
 assume $r=1$ without loss of generality.

  Suppose this result is wrong, then there is a sequence of points
   $(p_i, t_i)$ satisfying
  \begin{align*}
    t_i \leq 0 \; \textrm{and} \; \lim_{i \to \infty} t_i =0;
    \quad   \lim_{i \to \infty} d_{g_i(t_i)}(p_i, x_i)=0;
    \quad   \lim_{i \to \infty} |Rm|_{g_i(t_i)}(p_i)=\infty.
  \end{align*}

  Let $(X, x, g)$ be the common limit
  space of sequences $(X_i, x_i, g_i(t_i))$ and
   $(X_i, x_i, g_i(0))$.   So $\displaystyle x=\lim_{i \to \infty} x_i=\lim_{i \to \infty} p_i$
   is an added point as  $(X_i, x_i, g_i(t_i))$ converges.  So Corollary~\ref{corollary: cadded}
   implies that $x$ is an added point as $(X_i, x_i, g_i(0))$ converges.
   This is impossible since
    $\displaystyle  \sup_{B_{g_i(0)}(x_i, 1)} |Rm|_{g_i(0)} \leq 1$
   for every $i$!
\end{proof}

\vspace{0.3in}
  Combining Lemma~\ref{lemma: prermcontrolover} with Shi's estimates,
  we see that for every geodesic ball $B_{g_i(0)}(x, r)$ satisfying
  $\displaystyle \sup_{B_{g_i(0)}(x, r)} |Rm|_{g_i(0)} \leq r^{-2}$, we have
  \begin{align*}
      |\nabla^k Rm|_{g_i(0)}(y) \leq C(k, \kappa, r)
  \end{align*}
  whenever $y \in B_{g_i(0)}(x, \frac{r}{2})$.  Using the proof of
  Lemma~\ref{lemma: wwcpt}, we can strengthen the conclusion as
  smooth convergence.  To be precise, we have\\

  \textit{
  If $\{ (X_i, x_i, g_i(t)), -1 \leq t \leq 1 \}$ is an EV-refined
  sequence, then $(X_i, x_i, g_i(0))$ converges in the pointed
  Gromov-Hausdorff topology to a Ricci-flat multifold
  $(X, x, g)$.
  Furthermore, the convergence is in $C^{\infty}$-topology away from
  finite singular points.  The number of singular points is
  not greater than $N_0$.}\\

  By further study, we can show that every singular point is irreducible.
  This means that the limit multifold must be an orbifold.

 \begin{lemma}
 [\textbf{Weak Compactness for Time Slices in an EV-refined Sequence}]
  If  $\{ (X_i, x_i, g_i(t)), -1 \leq t \leq 1
 \}$ is an EV-refined sequence, then we have
 \begin{align*}
    (X_i, x_i, g_i(0)) \stackrel{C^{\infty}}{\to} (X, x, g)
 \end{align*}
 where $(X, g)$ is a $\kappa$-noncollapsed, Ricci-flat ALE orbifold
 with at  most $N_0$ singular points.
 \label{lemma: wcptev}
 \end{lemma}

 \begin{proof}
 The smooth convergence away from singular points is a direct
  application of Shi's estimates.  We only need to show that every
  singular point is irreducible.

  Suppose this Lemma is wrong, then there is a reducible singular point $p \in X$.
  We can choose a very small number $\eta>0$ such that $\partial B(p, \eta)$
  is disconnected.  Choose $x$ and $y$ in different component of
  $\partial B(p, \eta)$, then the shortest geodesic connecting $x$
  and $y$ must pass through $p$.  By choosing $\eta$ small enough,
  we can assume $\int_{B(p, \eta)} |Rm|^{\frac{m}{2}} d\mu <
  \frac{\epsilon}{4}$.

   Assume $x_i \to x$, $y_i \to y$, $p_i \to p$. Let $\gamma_i$ be a shortest
   geodesic connecting $x_i$ and $y_i$, $q_i$ be a point on $\gamma_i$
   which is closest to the highest curvature points in $B_{g_i(0)}(p_i, \eta)$.
   Clearly, $q_i \to p$.   Under metric $g_i(0)$, define $r_i$ to be the radius
   satisfying
   $\int_{B(q_i, r_i)} |Rm|^{\frac{m}{2}} d\mu=\frac{\epsilon}{2}$.
   Since $p$ is a singularity and $q_i \to p$, we know $r_i$ is well
   defined and $r_i \to 0$.   See Figure~\ref{figure: splitlineirr}
   for intuition.

   Let $h_i(t)= r_i^{-2}g_i(r_i^2t)$, we can extract an
   EV-refined sequence $\{(X_i, q_i, h_i(t)), -1 \leq t \leq 1\}$.
   Take limit for the central time slices, we have
   \begin{align*}
       (X_i, q_i, h_i(0)) \sconv (\hat{X}, \hat{q}, \hat{h}).
   \end{align*}
   The limit space $\hat{X}$ contains a line $\hat{\gamma}$ passing through
   $\hat{q}$.

   Note that $\int_{B_{h_i(0)}(q_i, 1)} |Rm|^{\frac{m}{2}} d\mu
   =\frac{\epsilon}{2}$.  The energy concentration property for
   EV-refined sequence enforces $|Rm|_{h_i(0)}(q_i)$ is bounded.
   Therefore, $\hat{q}$ is not a singular point.   By the choice of
   $q_i$, we see that the line $\hat{\gamma}$ will not pass through
   any singular point.    Detach $\hat{X}$ into union of orbifolds.
   $\hat{\gamma}$ must stay on one of them since it doesn't hit any
   singular point.   Let $M$ be the component containing
   $\hat{\gamma}$. It is a Ricci-flat ALE orbifold containing a
   line.  Splitting theorem for Ricci-flat orbifold implies $M$ is
   the product of a manifold and a line.   The ALE condition forces
   that $M$ must be flat and consequently the Euclidean space.
   According to Lemma~\ref{lemma: sitsmooth}, $M$ doesn't
   contain any singular point.  By the connectedness of $\hat{X}$, we
   have $\hat{X}=M$.   Therefore, $\hat{X}$ is Euclidean space which
   doesn't contain any singular point. In particular,  $|Rm|_{h_i(0)}$ is uniformly
   bounded in $B_{h_i(0)}(q_i, 2)$ when $i$ large.   Control
   convergence theorem implies
   \begin{align*}
   0= \int_{\overline{B(\hat{q}, 1)}} |Rm|^{\frac{m}{2}} d\mu
    = \lim_{i \to \infty} \int_{\overline{B(q_i, 1)}}
    |Rm|_{h_i(0)}^{\frac{m}{2}}d\mu_{h_i(0)} = \lim_{i \to \infty} \int_{\overline{B(q_i,
    r_i)}} |Rm|_{g_i(0)}^{\frac{m}{2}} d\mu_{g_i(0)} = \frac{\epsilon}{2}.
   \end{align*}
   Contradiction!
 \end{proof}

   \begin{figure}
   \begin{center}
   \psfrag{x}{$x$}
   \psfrag{xi}{$x_i$}
   \psfrag{Xi}[c][c]{$B_{g_i(0)}(p_i, \eta) \subset X_i$}
   \psfrag{Xinf}[c][c]{$B(p, \eta) \subset X$}
   \psfrag{yi}{$y_i$}
   \psfrag{y}{$y$}
   \psfrag{p}{$p$}
   \psfrag{pi}{$p_i$}
   \psfrag{Bi}{$B(q_i, r_i)$}
   \psfrag{Con}[c][c]{$C^{\infty}$ convergence}
   \includegraphics[width=1.0 \columnwidth]{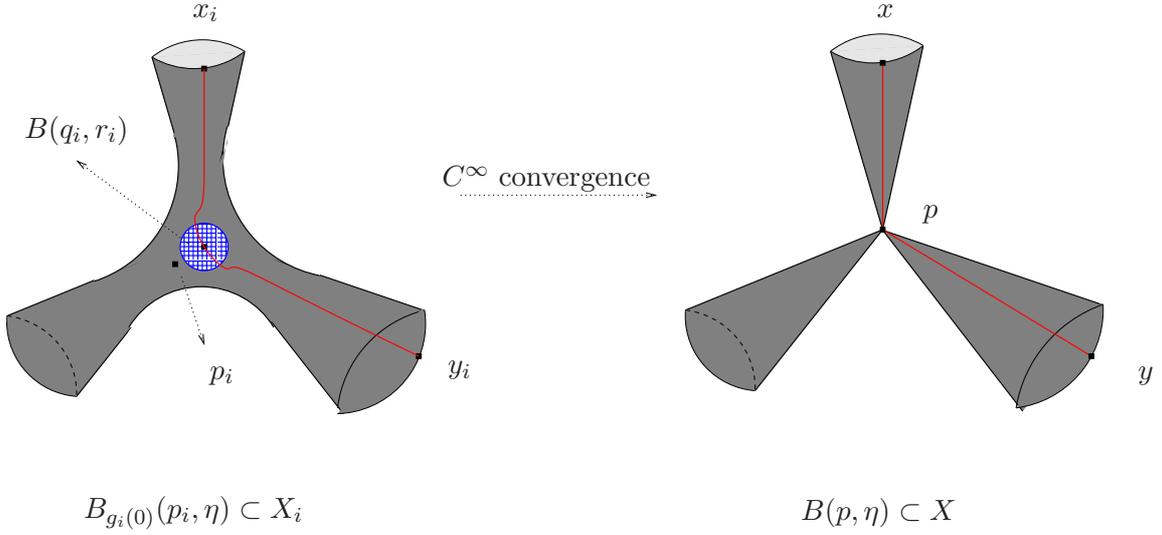}
   \end{center}
   \caption{Blowup at reducible singular point}
   \label{figure: splitlineirr}
   \end{figure}


  The weak compactness Lemma~\ref{lemma: wcptev} directly
  yields the following Corollaries.

 \begin{corollary}[\textbf{Improved Volume Ratio Upper Bound for an EV-Refined Sequence}]
 If  $\{ (X_i, x_i, g_i(t)), -1 \leq t \leq 1
 \}$ is an EV-refined sequence, $r$ is any fixed positive number,
 then we have
 \begin{align*}
    \limsup_{i \to \infty} \frac{\Vol_{g_i(0)}(B_{g_i(0)}(x_i,
    r))}{r^m} \leq \omega(m).
 \end{align*}
 \label{corollary: vupperover}
 \end{corollary}

 \begin{proof}
   Application of volume comparison theorem for Ricci-flat orbifolds (c.f.~\cite{Bor}).
 \end{proof}

 \begin{corollary}[\textbf{Energy Concentration for an EV-Refined Sequence}]
  If an EV-refined sequence
  $\{ (X_i, p_i, g_i(t)), -1 \leq t \leq 1\}$
 satisfies $Q_i=|Rm|_{g_i(0)}(p_i) \geq 1$, then
 energy concentration holds at $(p_i,0)$ for large $i$. In
 other words, we have
 \begin{align*}
    \int_{B_{g_i(0)}(p_i,Q_i^{-\frac12})} |Rm|_{g_i(0)}^{\frac{m}{2}} d\mu_{g_i(0)} > \epsilon.
    \end{align*}
  \label{corollary: evconcentration}
 \end{corollary}

 \begin{proof}
 If this result is wrong, by taking subsequence, we can obtain an EV-refined
   sequence $\{(X_i, p_i, g_i(t)), -1 \leq t \leq 1\}$ satisfying
 \begin{align*}
   Q_i=|Rm|_{g_i(0)}(p_i) \geq 1, \quad  \int_{B_{g_i(0)}(p_i,
   Q_i^{-\frac12})} |Rm|_{g_i(0)}^{\frac{m}{2}} d\mu_{g_i(0)} \leq \epsilon.
 \end{align*}
 Since blown up sequence of an EV-refined sequence is a new EV-refined sequence, we can assume
 $Q_i \equiv 1$. For the simplicity of notation, we omit the default
 metric $g_i(0)$  and denote these conditions as
 \begin{align*}
     |Rm|(p_i)=1,  \quad \int_{B(p_i, 1)} |Rm|^{\frac{m}{2}} d\mu  \leq
     \epsilon.
 \end{align*}

 \begin{clm}
   Curvature of any point in the unit ball can be controlled by its distance to the base point. In
   other words, we have
 \begin{itemize}
 \item     $|Rm|(q) < (1-d(p_i, q))^{-2}$ whenever $d(p_i, q) > 1- H^{-\frac12}$.
 \item     $|Rm|(q) \leq H$ whenever $d(p_i, q) \leq 1-H^{-\frac12}$.
 \end{itemize}
 \end{clm}
 Otherwise, $B(q, |Rm|^{-\frac12}(q)) \subset B(p_i, 1)$ and $|Rm|(q) > H$. Since energy concentrate at
 any point satisfying $|Rm|>H$, we have
 \begin{align*}
   \int_{B(p_i, 1)} |Rm|^{\frac{m}{2}} d\mu \geq \int_{B(q, |Rm|^{-\frac12}(q))}
   |Rm|^{\frac{m}{2}} d\mu > \epsilon.
 \end{align*}
 Contradiction! This proves the Claim. \\

  Lemma~\ref{lemma: wcptev} implies
 \begin{align*}
    (X_i, p_i, g_i(0)) \sconv (X, p, g).
 \end{align*}
 Our Claim assures that the geodesic ball $B(p, 1)$ is free
 of singular point.  It's clearly a $\kappa$-noncollapsed, Ricci-flat, smooth
 geodesic ball satisfying $|Rm|(p)=1$.   Therefore,
 Lemma~\ref{lemma: econ} yields
  $\displaystyle   \int_{B(p, \frac12)} |Rm|^{\frac{m}{2}} d\mu >
  \epsilon$.
 The smooth convergence then implies
 \begin{align*}
  \epsilon \geq  \lim_{i \to \infty} \int_{B(p_i, 1)}|Rm|^{\frac{m}{2}} d\mu
   \geq \lim_{i \to \infty} \int_{B(p_i, \frac12)} |Rm|^{\frac{m}{2}} d\mu
    =\int_{B(p, \frac12)} |Rm|^{\frac{m}{2}} d\mu > \epsilon.
 \end{align*}
 Contradiction!

 \end{proof}

   In the definition of an EV-refined
  sequence, we only have energy concentration for points satisfying
  $|Rm| \geq H$.   By Corollary~\ref{corollary: evconcentration}, we see
  energy concentration holds automatically for points with
   $|Rm| \geq 1$.   The gap between $1$ and $H$ can be overcome ``freely".
  Similarly, at the beginning, we only assume volume ratio has an
  upper bound $K$. However, we finally can improve this $K$ to $\omega(m)$
  ``freely".  These improvements  give us the room to show that every refined sequence
  is actually an EV-refined sequence.

\subsection{Every Refined Sequence is an EV-refined Sequence}
  In this subsection, we show that a priori local volume ratio upper
  bound  and energy concentration property  are not independent
  conditions in the definition of an EV-refined sequence. They can
  be deduced from other conditions. In other words, every refined
  sequence is an EV-refined sequence.

 \begin{proposition}
  Every E-refined sequence is an EV-refined sequence.
  \label{proposition: eev}
 \end{proposition}

 \begin{proof}
   If this result is wrong, then there exists an E-refined sequence
   $\{(X_i, g_i(t)), -1 \leq t \leq 1\}$
 which is not an EV-refined sequence. In other words, there  are triples
 $(x_i, t_i, r_i) \in X_i \times [-\frac14, 0] \times (0,1]$ such that
 \begin{align*}
   \lim_{i \to \infty} \frac{\Vol_{g_i(t_i)}(B_{g_i(t_i)}(x_i, r_i))}{r_i^{m}} = \infty.
 \end{align*}

 \begin{figure}
 \begin{center}
 \psfrag{t}[c][c]{$t$}
 \psfrag{X}[c][c]{$X_i$}
 \psfrag{b1}[c][c]{\footnotesize $B_{g_i(t_i)}(x_i, r_i)$}
 \psfrag{b2}[c][c]{\footnotesize $B_{g_i(s_i^{(1)})}(y_i^{(1)}, \rho_i^{(1)})$}
 \psfrag{b3}[c][c]{\footnotesize $B_{g_i(s_i^{(2)})}(y_i^{(2)}, \rho_i^{(2)})$}
 \psfrag{s1}[r][r]{\tiny $X_i \times [t_i-\frac18 r_i^2, t_i]$}
 \psfrag{s2}[r][r]{\tiny $X_i \times [s_i^{(1)}-\frac18(\rho_i^{(1)})^2, s_i^{(1)}]$}
 \psfrag{s3}[r][r]{\tiny $X_i \times [s_i^{(2)}-\frac18(\rho_i^{(2)})^2, s_i^{(2)}]$}
 \psfrag{t1}[l][l]{$t=-\frac14$}
 \psfrag{t2}[l][l]{\textcolor{blue}{$t=-\frac12$}}
 \psfrag{A}[c][c]{\small \textcolor{red}{$(y_i, s_i, \rho_i) \triangleq (y_i^{(2)}, s_i^{(2)}, \rho_i^{(2)})$}}
 \includegraphics[width= 0.8 \columnwidth]{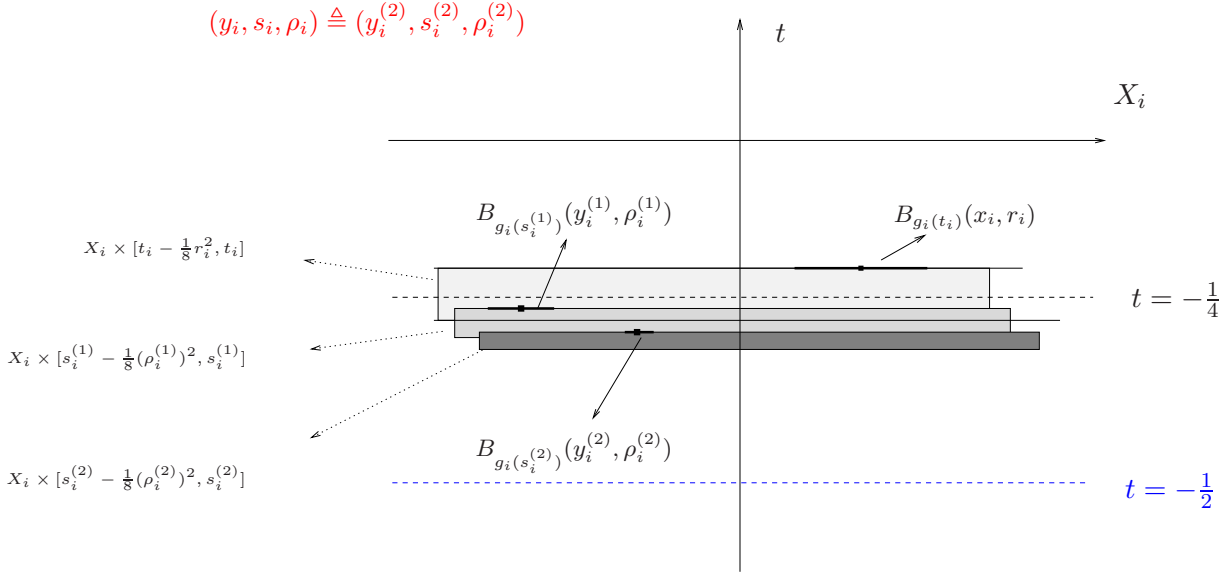}
 \caption{How to select a good scale?}
 \label{figure: selectvballs}
 \end{center}
 \end{figure}

 \begin{clm}
   There exists a triple  $(y_i, s_i, \rho_i) \in X_i \times [-\frac12, 0] \times (0, 1]$ satisfying
   the following properties.
 \begin{enumerate}
 \item $[s_i- \frac14 \rho_i^2, s_i] \subset [-\frac12, 0]$.
 \item Volume ratio upper control does not hold at $(y_i, s_i, \rho_i)$, i.e.,
      \begin{align*}
        \frac{\Vol_{g_i(s_i)}(B_{g_i(s_i)}(y_i, \rho_i))}{\rho_i^m} > 2\omega(m).
      \end{align*}
 \item At every triple
   $(z, \theta, r) \in X_i \times [s_i- \frac18 \rho_i^2] \times (0, \frac12 \rho_i]$,
       volume ratio upper control holds, i.e.,
        $ \frac{\Vol_{g_i(\theta)}(B_{g_i(\theta)}(z, r))}{r^m} \leq  2\omega(m).$
 \end{enumerate}
 \end{clm}
 We start from $(x_i, t_i, r_i)$. Check if it can be chosen as $(y_i, s_i,
 \rho_i)$. If so, we are done.  Otherwise, we can find one triple
 $(z, \theta, r) \in X_i \times [t_i-\frac18 r_i^2, t_i] \times (0, \frac12 r_i]$
 such that volume ratio upper bound fails at $(z, \theta, r)$.
  Denote this triple $(z, \theta, r)$ as $(y_i^{(1)}, s_i^{(1)},
 \rho_i^{(1)})$ and check if it can be chosen as $(y_i, s_i, r_i)$.
 If so, we finish the proof. Otherwise, we can find some triple
 $(y_i^{(2)}, s_i^{(2)}, \rho_i^{(2)}) \in X_i \times [s_i^{(1)}-\frac18  (\rho_i^{(1)})^2, s_i^{(1)}] \times (0, \frac12 \rho_i^{(1)}] $
 such that volume ratio upper bound fails at $(y_i^{(2)}, s_i^{(2)}, \rho_i^{(2)})$.
 Similarly, we can continue to define $(y_i^{(k)}, s_i^{(k)},
 \rho_i^{(k)})$ $\cdots$.

     At triple $(y_i^{(k)}, s_i^{(k)}, \rho_i^{(k)})$, we have
 \begin{align*}
     t_i - \{s_i^{(k)} - \frac14(\rho_i^{(k)})^2\}&= \frac14(\rho_i^{(k)})^2 +\{s_i^{(0)} - s_i^{(k)} \}\\
      &=\frac14(\rho_i^{(k)})^2 + \sum_{l=0}^{k-1} \{s_i^{(l)} -
      s_i^{(l+1)}\}\\
      &\leq \frac14(\rho_i^{(k)})^2 + \frac18
      \sum_{l=0}^{k-1}(\rho_i^{(l)})^2\\
      &=\frac18(\rho_i^{(k)})^2 + \frac18
      \sum_{l=0}^{k}(\rho_i^{(l)})^2.
 \end{align*}
 Note that $\rho_i^{(l)} \leq 2^{-l} r_i$, we have
 \begin{align*}
      t_i - \{s_i^{(k)} - \frac14(\rho_i^{(k)})^2\}
      &\leq  \frac18 (2^{-2k} + \frac43 (1-4^{-k-1})) r_i^2\\
      &\leq \frac{2+4^{-k}}{12} r_i^2 \leq \frac14.
 \end{align*}
 Recall that $t_i \in [-\frac14, 0]$.
 Therefore, $[s_i^{(k)}-\frac14(\rho_i^{(k)})^2, s_i^{(k)}] \subset [-\frac12, 0]$.
 So this process repeats in the compact smooth spacetime
  $X_i \times [-\frac12, 0]$.    So we have
 \begin{align*}
    \lim_{r \to 0} \sup_{(x, t) \in X_i \times [-\frac12, 0]} \frac{\Vol_{g(t)}(B_{g(t)}(x,
    r))}{r^m}= \omega(m) < 2\omega(m).
 \end{align*}
 Since $\displaystyle \lim_{k \to \infty} \rho_i^{(k)}=0$, volume ratio upper bound holds at
 $(y_i^{(k)}, s_i^{(k)}, \rho_i^{(k)})$ for large $k$, we see this process must stop in finite steps.
  So we can find a finite $k$ such that $(y_i^{(k)}, s_i^{(k)}, \rho_i^{(k)})$ can be
  chosen as $(y_i, s_i, \rho_i)$.
 Figure~\ref{figure: selectvballs} illustrates such a process which stops in 2 steps.
 We finish the proof of the Claim.\\

 Let $\tilde{g}_i(t) = (\frac{\rho_i}{2})^{-2} g_i((\frac{\rho_i}{2})^2 t + s_i)$. Clearly,
 $\{(X_i, \tilde{g}_i(t)), -1 \leq t \leq 1\}$ is an E-refined sequence
 satisfying
 \begin{align*}
    \frac{\Vol_{\tilde{g}_i(t)}(B_{\tilde{g}_i(t)}(x, r))}{r^m} \leq  2 \omega(m)
 \end{align*}
 for every $i$ and $(x, t, r) \in X_i \times [-\frac12, 0] \times (0, 1]$.
 Consequently, $\{(X_i, \tilde{g}_i(t)), -1 \leq t \leq 1\}$ is an
 EV-refined sequence.  By Corollary~\ref{corollary: vupperover},
 we have
 \begin{align*}
    \limsup_{i \to \infty} \frac{\Vol_{\tilde{g}_i(0)}(B_{\tilde{g}_i(0)}(y_i, 2))}{2^m} \leq \omega(m).
 \end{align*}
 On the other hand, by our choice of $(y_i, s_i, \rho_i)$, we have
 \begin{align*}
  \frac{\Vol_{\tilde{g}_i(0)}(B_{\tilde{g}_i(0)}(y_i, 2))}{2^m}
 = \frac{\Vol_{g_i(s_i)}(B_{g_i(s_i)}(y_i, \rho_i))}{\rho_i^m} >
 2\omega(m).
 \end{align*}
 Contradiction!
 \end{proof}

 \begin{proposition}
   Every refined sequence is an E-refined sequence.
   \label{proposition: re}
 \end{proposition}

 \begin{proof}
   If this result is wrong, then there is a refined sequence $\{(X_i, g_i(t)), -1 \leq t \leq 1\}$
  and points $(x_i, t_i) \in X_i \times [-\frac12, 0]$ such that
 \begin{align*}
     Q_i=|Rm|_{g_i(t_i)}(x_i) \to \infty,  \quad  \int_{B_{g_i(t_i)}(x_i,
     Q_i^{-\frac12})} |Rm|^{\frac{m}{2}} d\mu \leq \epsilon.
 \end{align*}

    Search points in $X_i \times [t_i-\frac12Q_i^{-1}, t_i]$ to see if
    energy concentration fails at some point with curvature norm
    greater than $2Q_i$. If there exists such a point, we choose one
    of them as a new base point and then do the same work for this
    new base point.   No matter how many steps do we have, we know this
    new base point must locate in
    $X_i \times [t_i-Q_i^{-1}, t_i] \subset X_i \times [-\frac34, t_i]$ for large $i$.
    This is a compact smooth spacetime with bounded Riemannian curvature. After each step,
    the base point's Riemannian curvature norm doubled, so our process must stop in
    finite steps. Fix the last step base point as the new base point
    $(y_i, s_i)$ and define $\bar{Q}_i= |Rm|_{g_i(s_i)}(y_i)$.
    It satisfies the following properties.
  \begin{enumerate}
   \item $\displaystyle \lim_{i \to \infty} \bar{Q}_i=\infty$.
   \item Energy concentration fails at $(y_i, s_i)$.
   \item Energy concentration holds at $(x,t)$ whenever it locates
   in $X_i \times [s_i-\frac12\bar{Q}_i^{-1}, s_i]$ and
   $|Rm|_{g_i(t)}(x) > 2\bar{Q}_i$.
   \item $[s_i-\bar{Q}_i^{-1}, s_i] \subset [-1, 0]$.
  \end{enumerate}

  Let $\tilde{g}_i(t) = \bar{Q}_i g(\bar{Q}_i^{-1}t + s_i)$. Clearly, $\{(X_i, y_i, \tilde{g}_i(t)), -1 \leq t \leq 1\}$
  is an E-refined sequence. Proposition~\ref{proposition: eev}
  implies that it is an EV-refined sequence.    Furthermore,  at the base points $(y_i, 0)$,
  energy concentration fails  and  $|Rm|_{\tilde{g}_i(0)}(y_i)=1$.
  This contradicts to Corollary~\ref{corollary: evconcentration}!
 \end{proof}

 Combining Proposition~\ref{proposition: re} and
 Proposition~\ref{proposition: eev}, we have
 \begin{lemma}
  Every refined sequence is an EV-refined sequence.
 \label{proposition: rev}
 \end{lemma}

 Since every refined sequence is an EV-refined sequence, we can
 use directly all the Lemmas of previous section. For convenience,
 we list the important ones among them.

 \begin{theorem}[\textbf{Backward Pseudolocality for a Refined Sequence}]
 If a refined sequence $\{ (X_i, x_i, g_i(t)), -1 \leq t \leq 1\}$
 satisfies $\displaystyle        \sup_{B_{g_i(0)}(x_i, r)}
 |Rm|_{g_i(0)} \leq r^{-2}$ for some constant $r \in (0,1]$,    then
 there exists a small positive constant $\delta$ (depend on this
 sequence) such that
 \begin{align*}
     |Rm|_{g_i(t)}(x) \leq  \delta^{-2}r^{-2}, \quad \textrm{whenever}  \;
     d_{g_i(t)}(x_i,x) \leq
     \delta r, \; -\delta^2 r^2 \leq t \leq 0
 \end{align*}
 for large $i$.
  \label{theorem: prermcontrol}
 \end{theorem}

 \begin{theorem}
 [\textbf{Weak Compactness for Time Slices in a Refined Sequence}]
  If  $\{ (X_i, x_i, g_i(t)), -1 \leq t \leq 1
 \}$ is a refined sequence, then we have
 \begin{align*}
    (X_i, x_i, g_i(0)) \stackrel{C^{\infty}}{\to} (X, x, g)
 \end{align*}
 where $(X, g)$ is a $\kappa$-noncollapsed, Ricci-flat ALE orbifold
 with at most $N_0$ singular points.
  \label{theorem: wcpt}
 \end{theorem}

 \begin{theorem}[\textbf{Energy Concentration in a Refined Sequence}]
 If $\{(X_i, x_i, g_i(t)), -1 \leq t \leq 1\}$ is a refined
 sequence,  $Q_i=|Rm|_{g_i(0)}(x_i) \geq 1$, then we have
 \begin{align*}
      \int_{B_{g_i(0)}(x_i, Q_i^{-\frac12})} |Rm|_{g_i(0)}^{\frac{m}{2}}
      d\mu_{g_i(0)} > \epsilon
 \end{align*}
 for large $i$.
 \label{theorem: econ}
 \end{theorem}

  \begin{theorem}[\textbf{Improved Volume Ratio Upper Bound for an Refined Sequence}]
 If  $\{ (X_i, x_i, g_i(t)), -1 \leq t \leq 1
 \}$ is a refined sequence, $r$ is any fixed positive number,
 then we have
 \begin{align*}
    \limsup_{i \to \infty} \frac{\Vol_{g_i(0)}(B_{g_i(0)}(x_i,
    r))}{r^m} \leq \omega(m).
 \end{align*}
 \label{theorem: vupper}
 \end{theorem}


\subsection{Neck Structure and Isoperimetric Constant}
 The weak compactness theorem allows us to study the neck structure
 of a refined sequence.  The local orbifold group of a singularity
 in the limit space is the same as the infinity orbifold group of a new limit
 space properly blown up from the original sequence.  The part between these
 two ends are called neck and it is almost a flat cone.  This result
 is the same as the neck theorems proved in~\cite{Ban90}
 and~\cite{AC}.  However, we have to prove it through different ways
 since we lack the estimates in~\cite{Ban90} and~\cite{AC}.

  \textit{For the simplicity of notation,
  we assume $g_i(0)$ as the default metric on manifold $X_i$ and we'll
  not mention this default metric when no confusion will be caused.}

\begin{theorem}[\textbf{Coincidence of Neck Ends}]
 Suppose $\{(X_i, x_i, g_i(t)), -1 \leq t \leq 1\}$ is a refined
  sequence, $(X, x, g)$ is the limit space of $(X_i, x_i, g_i(0))$.
 Suppose $\displaystyle p=\lim_{i \to \infty} p_i$ is  a singular point with local orbifold group $\Gamma$.  Choose $\eta>0$ very small such that the following properties hold.
  \begin{itemize}
  \item $p$ is the unique singular point in $B(p, \eta)$.
  \item $\partial B(p, r)$ is diffeomorphic to $S^{m-1}/ \Gamma$ for
     every  $r \in (0, \eta)$.
  \item $\int_{B(p, \eta)}|Rm|^{\frac{m}{2}} d\mu <\frac{\epsilon}{4}$.
  \end{itemize}
  $r_i$ is the radius satisfying
  $\int_{B(p_i, \eta) \backslash B(p_i, r_i)} |Rm|^{\frac{m}{2}}d\mu=\frac{\epsilon}{2}$.
  Then $r_i \to 0$. Let $\tilde{g}_i(t)= r_i^{-2}g_i(r_i^2 t)$, we have the convergence
  \begin{align*}
   (X_i, p_i, \tilde{g}_i(0)) \sconv (\tilde{X}, \tilde{p}, \tilde{g}),
  \end{align*}
  where $\tilde{X}$ is an ALE orbifold  whose orbifold group at infinity is
  $\tilde{\Gamma}$.

  According to this choice, $\Gamma$ and $\tilde{\Gamma}$ are isomorphic to each
  other.
 \label{theorem: trivialneck}
\end{theorem}

 \begin{proof}
    Suppose $\Gamma$ and $\tilde{\Gamma}$ are not isomorphic.

    Choose $A$ to be a big constant such that $\partial B(\tilde{p}, A)$
    is diffeomorphic to $S^{m-1} / \tilde{\Gamma}$. So we know $\partial B_{\tilde{g}_i(0)}(p_i, A)$
    is diffeomorphic to $S^{m-1} / \tilde{\Gamma}$ for large $i$.
    In other words, $\partial B_{g_i(0)}(p_i, Ar_i)$ is
    diffeomorphic to $S^{m-1} / \tilde{\Gamma}$.
    However, the convergence before blowup implies that
    \begin{align*}
     \partial B_{g_i(0)}(p_i, A^{-1}\eta) \sim S^{m-1} / \Gamma,
    \end{align*}
    where we use $\sim$ to denote  diffeomorphism.  As $Ar_i \ll A^{-1}\eta$ and $\Gamma \neq \tilde{\Gamma}$, we can
    define
    \begin{align*}
     \rho_i \triangleq \sup \{r | r \geq Ar_i, \; \partial B_{g_i(0)}(p_i, r) \; \textrm{is not diffeomorphic to } S^3 /
     \Gamma\}.
    \end{align*}
    Clearly,  $\frac{Ar_i}{\rho_i} \to 0$ and $\frac{\rho_i}{A^{-1}\eta} \to 0$. Let $\hat{g}_i(t) = \rho_i^{-2}g_i(\rho_i^2
    t)$, then we have convergence
    \begin{align*}
       (X_i, p_i, \hat{g}_i(0))  \sconv (\hat{X}, \hat{p}, \hat{g}).
    \end{align*}
    Fix every large constant $\mathcal{R}$,  we have $\mathcal{R} \rho_i <\eta$,
    $\mathcal{R}^{-1} \rho_i> r_i$ for large $i$.
    It follows that
     \begin{align*}
         \int_{B_{\hat{g}_i(0)}(p_i, \mathcal{R}) \backslash B_{\hat{g}_i(0)}(p_i,\mathcal{R}^{-1})}
           |Rm|_{\hat{g}_i(0)}^{\frac{m}{2}} d\mu_{\hat{g}_i(0)}
         & \leq
         \int_{B_{g_i(0)}(p_i,\eta) \backslash  B_{g_i(0)}(p_i,
           r_i)} |Rm|_{g_i(0)}^{\frac{m}{2}} d\mu_{g_i(0)} \leq \frac{\epsilon}{2}< \epsilon.
     \end{align*}
     Energy concentration property for refined sequence implies that
     curvature is uniformly bounded on
     $B_{\hat{g}_i(0)}(p_i, \frac12 \mathcal{R}) \backslash B_{\hat{g}_i(0)}(p_i,
     2\mathcal{R}^{-1})$.
     So $B(\hat{p}, \frac12 \mathcal{R}) \backslash B(\hat{p}, 2\mathcal{R}^{-1})$
     contains no singularity. As $\mathcal{R}$ can be chosen
     arbitrarily big, we conclude that $\hat{p}$ is the unique
     singularity in $\hat{X}$. Moreover, we have energy control
     $\displaystyle \int_{\hat{X}} |Rm|^{\frac{m}{2}} d\mu < \epsilon.$
     Therefore, gap lemma (Lemma~\ref{lemma: gap}) implies that $\hat{X}$ is a
     flat cone.  Note that $\partial B(\hat{p}, 2) \sim \partial B_{\hat{g}_i(0)}(p_i, 2\rho_i)$ for large $i$.
     By the definition of $\rho_i$, we know $\partial B(\hat{p}, 2)$ is
     diffeomorphic to $S^{m-1} / \Gamma$.  However, for large $i$,
     we also have
     \begin{align*}
     \partial B_{g_i(0)}(p_i, a \rho_i) \sim \partial B_{\hat{g}_i(0)}(p_i, a) \sim \partial B(\hat{p}, a)
     \sim S^{m-1} / \Gamma
     \end{align*}
     for every $a \in [\frac12, 1]$. This means
     \begin{align*}
       \sup \{r | r \geq Ar_i, \; \partial B_{g_i(0)}(p_i, r) \; \textrm{is not diffeomorphic to } S^3 /
     \Gamma\} \leq \frac12 \rho_i.
     \end{align*}
     This contradicts to the definition of $\rho_i$!
    \end{proof}

 \begin{lemma}[\textbf{$C^0$-Structure of Necks}]
    Same condition as in Theorem~\ref{theorem: trivialneck}.

   There is a constant $\mathcal{R}$ depending on this sequence such that
    $B(p_i, \mathcal{R}^{-1} \eta) \backslash B(p_i, \mathcal{R} r_i)$
    is $10^{-m}$-close to $C_{\mathcal{R}^{-1}\eta, \mathcal{R}r_i}(S^{m-1}/\Gamma)$ in
    $C^0$-topology,
    where $C_{\mathcal{R}^{-1}\eta, \mathcal{R}r_i}(S^{m-1}/\Gamma)$ is the corresponding
    annulus in the flat cone over $S^{m-1} / \Gamma$.
    In other words, there exists a diffeomorphism
    \begin{align*}
       \Psi: C_{\mathcal{R}^{-1}\eta, \mathcal{R}r_i}(S^{m-1}/\Gamma)
       \mapsto B(p_i, \mathcal{R}^{-1} \eta) \backslash B(p_i, \mathcal{R}r_i)
    \end{align*}
    such that
    \begin{align}
       \snorm{ \Psi^*(g_i(0)) - g_{\E}}{g_{\E}} (x) < 10^{-m} |x|^2
    \label{eqn: c0neck}
    \end{align}
    for every $x \in C_{\mathcal{R}^{-1}\eta, \mathcal{R}r_i}(S^{m-1}/\Gamma)$.

    In particular, the following properties hold when $i$ large.
 \begin{itemize}
 \item  Every neck has isoperimetric constant control
 \begin{align}
    \mathbf{I}(B(p_i, \mathcal{R}^{-1} \eta) \backslash B(p_i, \mathcal{R} r_i))
\frac12 \mathbf{I}(\R^m).
 \label{eqn: necksob}
 \end{align}
 \item  For every domain
 $\Omega \subset B(p_i, \mathcal{R}^{-1}\eta) \backslash B(p_i, \mathcal{R}r_i)$,  we have
 \begin{align}
    \frac{\Area(\partial \Omega \cap
     \{\overline{B(p_i, \mathcal{R}^{-1}\eta)} \backslash \overline{B(p_i, \mathcal{R}r_i)}\})}
     {\Area(\partial \Omega \cap \partial B(p_i,
     \mathcal{R}r_i))} > \frac12.
 \label{eqn: neckar}
 \end{align}
 \item For every domain
  $\Omega \subset B(p_i, 2\mathcal{R}^{-1} \eta) \backslash B(p_i,
  \mathcal{R}^{-1}\eta)$ satisfying $\Vol(\Omega) < \frac12  c(m)\mathcal{R}^{-m}\eta^m$, we have
 \begin{align}
    \frac{\Area(\partial \Omega \cap
     \{B(p_i, 2\mathcal{R}^{-1} \eta) \backslash \overline{B(p_i, \mathcal{R}^{-1}\eta)}\})}
     {\Area(\partial \Omega \cap \partial B(p_i, \mathcal{R}^{-1}\eta))} > \frac12
     c(m),
 \label{eqn: neckaro}
 \end{align}
 where $c(m)$ is the constant defined in Lemma~\ref{lemma: esob}.
 \end{itemize}
 \label{lemma: c0neck}
 \end{lemma}

 \begin{proof}
  We only need to prove equation (\ref{eqn: c0neck}).
  Actually, by the proof of Theorem~\ref{theorem: trivialneck}, we
  see that for every small number $\xi$, there is a constant
  $\mathcal{R}$ such that $B(p_i, 2r) \backslash B(p_i, r)$
  is $\xi$-close to the standard neck $C_{2r, r}(S^{m-1} / \Gamma)$ in the $C^{\lfloor \xi^{-1} \rfloor}$-topology.
  In other words, there is a
  diffeomorphism
  \begin{align*}
  \Psi_r: C_{2r, r}(S^{m-1} / \Gamma) \to B(p_i, 2r) \backslash B(p_i, r)
  \end{align*}
  such that
  \begin{align*}
    \snorm{\nabla^k(\Psi_r^* g_i(0) - g_{\E})}{g_{\E}} < r^{2+k}
    \xi, \quad \forall \; 0 \leq k \leq  \lfloor \xi^{-1} \rfloor
  \end{align*}
  whenever $ \mathcal{R}r_i < r < \mathcal{R}^{-1}\eta$.
 Fix $\xi << 10^{-m}$, and fix $\mathcal{R}$ according to this
 $\xi$.  By gluing all $\Psi_r$, we obtain a diffeomorphism $\Psi$
 \begin{align*}
       \Psi: C_{\mathcal{R}^{-1}\eta, \mathcal{R}r_i}(S^{m-1}/\Gamma)
       \mapsto B(p_i, \mathcal{R}^{-1} \eta) \backslash B(p_i, \mathcal{R}r_i)
 \end{align*}
 and $\Psi$ satisfies equation (\ref{eqn: c0neck}). This gluing is
 the same as the gluing Cheeger and Anderson did in their neck theorem (c.f. Theorem1.18 of~\cite{AC}).
 \end{proof}

 After we understand the neck structure, we're going to control the isoperimetric constant
  of refined sequence locally.  We first set up some lemmas to study how the isoperimetric
  constants
  change among the
 ``infinity" and ``infinitesimal" orbifold singularities.

 \begin{lemma}[\textbf{Isoperimetric Constants' Control among Bubbles of Different Depth}]
  Suppose $\{(X_i, x_i, g_i(t)), -1 \leq t \leq 1\}$ is a refined
  sequence, $(X, x, g)$ is the limit space of $(X_i, x_i, g_i(0))$.
  $\rho$ is a positive number such that $\partial B(x, \rho)$
  is free of singularities.

  If $\displaystyle \lim_{i \to \infty} \mathbf{I}(B_{g_i(0)}(x_i, \rho))=0$, then
  $B(x, \rho)$ must contain singularities.
  Suppose $\{p_{\alpha}\}_{\alpha=1}^N$ are all singularities in
  $B(x, \rho)$, $\eta$ is so small such that $\overline{B(p_{\alpha}, \eta)}$
  are disjoint to each other and $\displaystyle \cup_{\alpha=1}^N B(p_{\alpha}, \eta) \subset B(x,
  \rho)$. Then we can choose $p_{\alpha, i}$ as the points with largest
  curvature norm in $B_{g_i(0)}(p_{\alpha, i}, \eta)$ and $p_{\alpha, i} \to p_{\alpha}$.
  Accordingly, choose $r_{\alpha, i}$ and $\mathcal{R}$ such that Lemma~\ref{lemma: c0neck}
  holds on these scales around every singularity.

  Then we have
  \begin{align*}
    \lim_{i \to \infty} \mathbf{I}(B_{g_i(0)}(p_{\alpha, i},
    \mathcal{R} r_{\alpha, i}))=0,
  \end{align*}
  for some $\alpha \in \{1, \cdots, N\}$.
 \label{lemma: samebubblesob}
 \end{lemma}

\begin{proof}
 We'll finish the proof in three steps.

 \textit{Step1. \quad  $B(x, \rho)$ must contain at least one singularity.}

  Otherwise, as $\partial B(x, \rho)$ doesn't contain any
  singularity, $B(x_i, \rho)$ converges to $B(x, \rho)$
  smoothly. So we have
  \begin{align}
   0=\lim_{i \to \infty} \mathbf{I}(B(x_i, \rho)) = \mathbf{I}(B(x, \rho))
    \geq
    \mathbf{I} (X).
   \label{eqn: izero}
  \end{align}
  However, as $X$ is a $\kappa$-noncollapsed, Ricci-flat orbifold, we have
  (c.f. Theorem A of~\cite{NA})
  \begin{align*}
  \mathbf{I}(X)>C(\kappa)>0.
  \end{align*}
  This contradicts to inequality (\ref{eqn: izero})! \\

  So we can assume $B(x, \rho)$ contains singularities.
 Suppose that $\displaystyle \{p_{\alpha} = \lim_{i \to \infty} p_{\alpha, i} \}_{\alpha=1}^N$
 are all the singularities.\\

 \textit{Step2. \quad
   $\displaystyle\lim_{i \to \infty} \mathbf{I}(B(p_{\alpha, i}, \mathcal{R}^{-1}\eta))=0$
    for some $\alpha \in \{1, \cdots, N\}$.}

 Because $\displaystyle \lim_{i \to \infty} \mathbf{I}(B(x_i, \rho))=0$,
 we can choose a sequence of domains $\Omega_i \subset B(x_i, \rho)$ such that
 \begin{align*}
      F(\Omega_i) \triangleq \frac{\Area(\partial
      \Omega_i)}{\Vol(\Omega_i)^{\frac{m-1}{m}}} \longrightarrow 0.
 \end{align*}

 \setcounter{claim}{0}
 \begin{clm}
  $\displaystyle \lim_{i \to \infty} \Vol(\Omega_i)=0$.
 \end{clm}
 Otherwise, by passing to subsequence, we can assume
  $\displaystyle \lim_{i \to \infty} \Vol(\Omega_i) = \lambda>0$.
  As volume ratio has upper bound, we can choose $\delta$ very small
  such that $B(p_{\alpha, i}, \delta)$ are disjoint and
  \begin{align*}
      \sum_{\alpha=1}^N \Vol(B(p_{\alpha, i}, \delta)) < \frac{\lambda}{2}.
  \end{align*}
  So for large $i$, we know
  \begin{align*}
  \Vol( \Omega_i \backslash \cup_{\alpha=1}^N B(p_{\alpha, i}, \delta))
\frac12\Vol(\Omega_i), \quad
  \Vol( \Omega_i \backslash \cup_{\alpha=1}^N B(p_{\alpha, i}, \delta))
\frac12 \lambda.
  \end{align*}
  It follows that
  \begin{align*}
    \frac{\Area (\partial\{ \Omega_i \backslash  \cup_{\alpha=1}^N B(p_{\alpha, i}, \delta)\})}
    {\Vol(\Omega_i \backslash  \cup_{\alpha=1}^N B(p_{\alpha, i}, \delta))^{\frac{m-1}{m}}}
    &\leq
    \frac{\Area(\partial \Omega_i) + \sum_{\alpha=1}^N \Area(\partial B(p_{\alpha, i}, \delta))}
    {\Vol(\Omega_i \backslash  \cup_{\alpha=1}^N B(p_{\alpha, i}, \delta))^{\frac{m-1}{m}}}\\
    &< 2^{\frac{m-1}{m}} \frac{\Area(\partial \Omega_i)}{\Vol(\Omega_i)^{\frac{m-1}{m}}}
        + 2^{\frac{m-1}{m}} \frac{2Nm \omega(m)
        \delta^{m-1}}{\lambda^{\frac{m-1}{m}}}.
  \end{align*}
  This yields that
   \begin{align*}
   \mathbf{I}(X) \leq \mathbf{I}(X \backslash \cup_{\alpha=1}^{N} B(p_{\alpha}, \delta))
     \leq \liminf_{i \to \infty}
    \frac{\Area (\partial\{ \Omega_i \backslash  \cup_{\alpha=1}^N B(p_{\alpha, i}, \delta)\})}
    {\Vol(\Omega_i \backslash  \cup_{\alpha=1}^N B(p_{\alpha, i}, \delta))^{\frac{m-1}{m}}}
   \leq C(m) \lambda^{\frac{1-m}{m}} \delta^{m-1}
  \end{align*}
  for every small $\delta$.  It implies that $\mathbf{I}(X)=0$ which is
  impossible for a $\kappa$-noncollapsed, Ricci-flat orbifold $X$!
  This contradiction
  shows that we must have $\displaystyle \lim_{i \to \infty}
  \Vol(\Omega_i)=0$ and we finish the proof of the Claim.\\

  Then we continue to show that
  $\displaystyle \lim_{i \to \infty} \mathbf{I}(B(p_{\alpha, i}, \mathcal{R}^{-1}\eta))=0$
  for some $\alpha \in \{1, \cdots, N\}$.
  Suppose this statement is wrong, we'll have a small constant $\iota$ such that
  \begin{align}
    \mathbf{I}(X)>\iota; \quad
    \mathbf{I}(B(p_{\alpha, i}, \mathcal{R}^{-1}\eta) >
    \iota, \quad \forall \; i, \alpha.
  \label{eqn: inoutiso}
  \end{align}
 As $F(\Omega_i) \to 0$, this forces that $\Omega_i$ intersects both
 $B(p_{\alpha,i}, \mathcal{R}^{-1}\eta)$ and
  $X \backslash B(p_{\alpha,i}, \mathcal{R}^{-1}\eta)$ for some
  $\alpha$.
 For simplicity of notation, let $\delta=\mathcal{R}^{-1}\eta$ and define
 \begin{align*}
  &A_{\alpha, i}^{in}= \Area(\partial \Omega_i \cap B(p_{\alpha, i}, \delta)),
   &A_{i}^{out}= \Area(\partial \Omega_i \backslash \cup_{\alpha=1}^N B(p_{\alpha, i}, \delta));\\
  &V_{\alpha, i}^{in}= \Vol(\Omega_i \cap B(p_{\alpha, i}, \delta)),
  &{V_i^{out}}= \Vol(\Omega_i \backslash \cup_{\alpha=1}^N \backslash B(p_{\alpha, i}, \delta));\\
  &A_{\alpha, i}^{center}= \Area(\partial B(p_{\alpha, i}, \delta) \cap \Omega_i).
 \end{align*}

 According to our choice of $\eta$ and $\mathcal{R}$,
 every $B(p_{\alpha, i}, 2\delta) \backslash B(p_{\alpha, i}, \delta)$ is very close to
 an annulus in the flat cone over $S^{m-1} / \Gamma$.
 So Lemma~\ref{lemma: c0neck} applies and we have
 \begin{align}
    A_{i}^{out} &\geq \sum_{\alpha=1}^N \Area(\partial \Omega_i \cap \{B(p_{\alpha, i}, 2\delta) \backslash B(p_{\alpha, i}, \delta)\})
     \notag\\
      &> \frac12 c(m)  \sum_{\alpha=1}^N \Area(\partial \Omega_i \cap \partial B(p_{\alpha, i}, \delta))
     \notag\\
      &=\frac12 c(m) \sum_{\alpha=1}^N A_{\alpha,i}^{center}
     \label{eqn: incontrolout}
 \end{align}
 The condition $\displaystyle \lim_{i \to \infty} F(\Omega_i)=0$
 implies that
 \begin{align*}
   \lim_{i \to \infty} \frac{A_{i}^{out}+ \sum_{\alpha=1}^N A_{\alpha, i}^{in}}{({V_i^{out}}+ \sum_{\alpha=1}^N
   V_{\alpha, i}^{in})^{\frac{m-1}{m}}}=0.
 \end{align*}
 This implies that
 \begin{align}
   \lim_{i \to \infty} \frac{A_{i}^{out}+ \sum_{\alpha=1}^N A_{\alpha, i}^{in}}
   {({V_i^{out}})^{\frac{m-1}{m}}+ \sum_{\alpha=1}^N  (V_{\alpha, i}^{in})^{\frac{m-1}{m}}}=0.
 \label{eqn: sobzero}
 \end{align}

 It follows from inequality (\ref{eqn: inoutiso}) that
 \begin{align*}
   \frac{A_{\alpha, i}^{in} + A_{\alpha, i}^{center}}{(V_{\alpha, i}^{in})^{\frac{m-1}{m}}}
   \geq \iota,  \quad \forall \alpha \in \{1, 2, \cdots, N\}; \quad
  \frac{A_{i}^{out} + \sum_{\alpha=1}^N A_{\alpha, i}^{center}}{(V_i^{out})^{\frac{m-1}{m}}}
  \geq \iota.
 \end{align*}
 This yields that
 \begin{align*}
    2\sum_{\alpha=1}^N A_{\alpha, i}^{center}
    \geq \iota ((V_i^{out})^{\frac{m-1}{m}} + \sum_{\alpha=1}^N
    (V_{\alpha, i}^{in})^{\frac{m-1}{m}}) - \sum_{\alpha=1}^N A_{\alpha, i}^{in} - A_{i}^{out}.
 \end{align*}
 Consequently, we have
 \begin{align*}
   \lim_{i \to \infty} \frac{\sum_{\alpha=1}^N A_{\alpha, i}^{center}}{(V_i^{out})^{\frac{m-1}{m}} + \sum_{\alpha=1}^N
    (V_{\alpha, i}^{in})^{\frac{m-1}{m}}}
    \geq \frac{\iota}{2}
   - \frac12 \cdot \lim_{i \to \infty} \frac{A_{i}^{out}+
   \sum_{\alpha=1}^N A_{\alpha, i}^{in}}{({V_i^{out}})^{\frac{m-1}{m}}+
   \sum_{\alpha=1}^N  (V_{\alpha, i}^{in})^{\frac{m-1}{m}}}=\frac{\iota}{2}.
 \end{align*}
 Combining this with inequality (\ref{eqn: incontrolout}), we have
 \begin{align*}
  &\quad \lim_{i \to \infty} \frac{A_{i}^{out}+ \sum_{\alpha=1}^N A_{\alpha, i}^{in}}{({V_i^{out}})^{\frac{m-1}{m}}+
   \sum_{\alpha=1}^N  (V_{\alpha, i}^{in})^{\frac{m-1}{m}}}\\
  &\geq
  \lim_{i \to \infty} \frac{A_{i}^{out}}{({V_i^{out}})^{\frac{m-1}{m}}+
  \sum_{\alpha=1}^N  (V_{\alpha, i}^{in})^{\frac{m-1}{m}}}\\
  &\geq \frac12 \cdot c(m) \cdot
  \lim_{i \to \infty} \frac{\sum_{\alpha=1}^N A_{\alpha, i}^{center}}{(V_i^{out})^{\frac{m-1}{m}} + \sum_{\alpha=1}^N
    (V_{\alpha, i}^{in})^{\frac{m-1}{m}}}\\
  &\geq \frac{\iota c(m)}{4}>0.
 \end{align*}
 This contradicts to equation (\ref{eqn: sobzero})! So our
 assumption is wrong, by passing to a subsequence if necessary,
 there must exist an $\alpha \in \{1, \cdots, N\}$ such that
 \begin{align*}
    \lim_{i \to \infty} \mathbf{I}(B(p_{\alpha, i}, \mathcal{R}^{-1}\eta))=0.
 \end{align*}
 This finishes the proof of step 2. \\

 \textit{Step3.  $\displaystyle \lim_{i \to \infty} \mathbf{I}(B(p_{\alpha, i}, \mathcal{R} r_{\alpha, i}))=0$.}

 For simplicity, we denote $p_{i}, r_i$ as $p_{\alpha, i}, r_{\alpha,
 i}$. The following argument is similar to the argument in step 2,
 and it is simpler.

    Suppose this statement is wrong, then $\mathbf{I}(B(p_i, \mathcal{R} r_i))>\lambda$ for
    some fixed positive constant $\lambda$.   As
    $\displaystyle \lim_{i \to \infty} \mathbf{I}(B(p_i,
    \mathcal{R}^{-1}\eta))=0$, we can choose domains $\Omega_i \subset B(p_i,
    \mathcal{R}^{-1}\eta)$ such that
   \begin{align*}
      F(\Omega_i) \triangleq \frac{\Area(\partial
      \Omega_i)}{\Vol(\Omega_i)^{\frac{m-1}{m}}} \longrightarrow 0.
   \end{align*}
 Note that $\Omega_i$ intersects both $\overline{B(p_i, \mathcal{R}r_i)}$
 and $\overline{B(p_i, \mathcal{R}^{-1}\eta) \backslash B(p_i,\mathcal{R}r_i)}$. Otherwise, we have
    \begin{align*}
      F(\Omega_i)=\frac{\Area(\partial \Omega_i)}{\Vol(\Omega_i)^{\frac{m-1}{m}}} \geq
      \min\{\mathbf{I}(B(p_i, \mathcal{R}^{-1}\eta) \backslash B(p_i, \mathcal{R}r_i)),
        \mathbf{I}(B(p_i, \mathcal{R} r_i))\}>\min\{\frac12 \mathbf{I}(\R^m), \lambda\}.
    \end{align*}
 This contradicts to our assumption $F(\Omega_i) \to 0$!  Therefore,
 we can define
    \begin{align*}
      &{V_i^{in}}= \Vol(\Omega_i \cap B(p_i, \mathcal{R}r_i)),   \quad
       &{V_i^{out}}= \Vol(\Omega_i \backslash B(p_i, \mathcal{R}r_i)); \\
      &{A_i^{in}}=\Area(\partial \Omega_i \cap B(p_i, \mathcal{R}r_i)), \quad
       &{A_i^{out}}=\Area(\partial \Omega_i \backslash B(p_i, \mathcal{R}r_i)); \\
      &{A_i^{center}}=\Area(\partial \Omega_i \cap \partial B(p_i, \mathcal{R}r_i)).
    \end{align*}
 Since $\displaystyle \lim_{i\to \infty}F(\Omega_i)=
  \lim_{i \to \infty} \frac{{A_i^{in}} + {A_i^{out}}}{({V_i^{in}} + {V_i^{out}})^{\frac{m-1}{m}}}=0$,
 we have
 \begin{align}
   \lim_{i \to \infty}
   \frac{{A_i^{in}} + {A_i^{out}}}{({V_i^{in}})^{\frac{m-1}{m}} +
   ({V_i^{out}})^{\frac{m-1}{m}}}=0.
 \label{eqn: avzero}
 \end{align}
 Since both $B(p_i, \mathcal{R}r_i)$ and $B(p_i, \mathcal{R}^{-1}\eta) \backslash B(p_i, \mathcal{R}r_i)$
 have bounded isoperimetric constants, we can find a positive constant
 $\iota$ such that
 \begin{align*}
  \mathbf{I}(B(p_i, \mathcal{R}r_i)) > \iota, \quad
  \mathbf{I}(B(p_i, \mathcal{R}^{-1}\eta) \backslash B(p_i,
  \mathcal{R}r_i)) > \iota.
 \end{align*}
 It follows that
 \begin{align*}
    \frac{{A_i^{in}}+{A_i^{center}}}{(V_i^{in})^{\frac{m-1}{m}}} > \iota, \quad
     \frac{{A_i^{out}}+{A_i^{center}}}{(V_i^{out})^{\frac{m-1}{m}}} > \iota.
 \end{align*}
 Consequently we have
 \begin{align*}
   A_i^{center} > \frac12\iota \{(V_i^{in})^{\frac{m-1}{m}} + (V_i^{out})^{\frac{m-1}{m}}\}
             -\frac12\{A_i^{in} + A_i^{out}\}.
 \end{align*}
 This yields that
 \begin{align*}
   \lim_{i \to \infty} \frac{A_i^{center}}{(V_i^{in})^{\frac{m-1}{m}} + (V_i^{out})^{\frac{m-1}{m}}}
    \geq \frac12 \iota
   -\frac12 \lim_{i \to \infty}
    \frac{A_i^{in} + A_i^{out}}{(V_i^{in})^{\frac{m-1}{m}} + (V_i^{out})^{\frac{m-1}{m}}}
   =\frac12 \iota.
 \end{align*}

 According to our choice of $\mathcal{R}$ and $\eta$, $r_i$,
 the neck part $B(p_i, \mathcal{R}^{-1}\eta \backslash \mathcal{R}r_i)$
 is very close to the annulus in the flat cone,
 Lemma~\ref{lemma: c0neck} applies and we have
 \begin{align*}
   A_i^{out}> \frac12 A_i^{center}.
 \end{align*}
 It follows that
 \begin{align*}
       \liminf_{i \to \infty}  \frac{A_i^{in} + A_i^{out}}{(V_i^{in})^{\frac{m-1}{m}} + (V_i^{out})^{\frac{m-1}{m}}}
       \geq
       \frac12 \lim_{i \to \infty} \frac{A_i^{center}}{(V_i^{in})^{\frac{m-1}{m}} + (V_i^{out})^{\frac{m-1}{m}}}
       \geq \frac14 \iota>0.
 \end{align*}
 This contradicts to inequality (\ref{eqn: avzero})! So we finish
 the proof of step 3.
 \end{proof}

 \begin{theorem}[\textbf{Isoperimetric Constants' Bounds}]
    Suppose $\{(X_i, g_i(t)), -1 \leq t \leq 1\}$ is a refined
  sequence. Then for every radius $r$, there is a small constant $c_r$ such that
  $\mathbf{I}(B_{g_i(0)}(x, r)) > c_r$ for every $i \in \N$ and $x \in X_i$.
 \label{theorem: refsob}
 \end{theorem}

  \begin{proof}

   Suppose this statement is wrong, then there exists a radius
   $\rho'>0$ and points $x_i \in X_i$ such that
   $ \displaystyle
      \lim_{i \to \infty} \mathbf{I}(B(x_i, \rho'))=0.
   $
   Let $(X, x, g)$ be the limit space of $(X_i, x_i, g_i(0))$.
   It has finite singularities, so we can find $\rho = \rho' + \delta$ for
   a small constant $\delta$ such that $\partial B(x, \rho)$ is free
   of singular points.  Note that
   $\mathbf{I}(B(x_i, \rho)) \leq \mathbf{I}(B(x_i, \rho'))$, so we have
   \begin{align*}
      \lim_{i \to \infty} \mathbf{I}(B(x_i, \rho))=0.
   \end{align*}
   As $\partial B(x, \rho)$ is free of singularities, Lemma~\ref{lemma: samebubblesob}
   applies. So there exists a singular point $\displaystyle p$
   in $B(x, \rho)$. Furthermore, choose $p_i$, $r_i$, $\eta$ and
   $\mathcal{R}$ as in Lemma~\ref{lemma: samebubblesob}, we have
  \begin{align}
   \lim_{i \to \infty} \mathbf{I}(B(p_i, \mathcal{R}r_i))=0.
  \label{eqn: sob00}
  \end{align}

  Define $g_i^{(1)}(t)= r_i^{-2} g_i(r_i^2t), x_i^{(1)}=p_i$, we have weak compactness
  \begin{align*}
      (X_i, x_i^{(1)}, g_i^{(1)}(0)) \sconv (X^{(1)}, x^{(1)}, g^{(1)}).
  \end{align*}
  Recall that the energy of $X$ is defined to be
  \begin{align*}
   \mathcal{E}(X)= \lim_{\mathcal{S} \to \infty} \{\limsup_{i \to
     \infty} \int_{B_{g_i(0)}(x_i,
     \mathcal{S})}|Rm|^{\frac{m}{2}}d\mu\}.
  \end{align*}
  Similarly, we have
  \begin{align*}
     \mathcal{E}(X^{(1)})&= \lim_{\mathcal{S} \to \infty} \{\limsup_{i \to
     \infty} \int_{B_{g_i^{(1)}(0)}(x_i^{(1)},
     \mathcal{S})}|Rm|^{\frac{m}{2}}d\mu\}\\
      &= \lim_{\mathcal{S} \to \infty} \{\limsup_{i \to
     \infty} \int_{B_{g_i(0)}(p_i,
     \mathcal{S}r_i^2)}|Rm|^{\frac{m}{2}}d\mu\}.
  \end{align*}
  Clearly, we have
  \begin{align}
 \left\{
   \begin{array}{ll}
    &\mathcal{E}(X^{(1)}) \leq \mathcal{E}(X),\\
    &\displaystyle \lim_{i \to \infty} \mathbf{I}(B_{g_i^{(1)}(0)}(x_i^{(1)},
    \mathcal{R}))=0,
   \end{array}
 \right.
 \label{eqn: ksob0}
 \end{align}
 where the second equation follows directly from equation (\ref{eqn: sob00}).
 Note that $x^{(1)}$ must be a singularity if $X^{(1)}$ contains
 singularity.

  Note that every singularity of $X^{(1)}$ is contained in $\overline{B(x^{(1)},
  1)}$. Clearly, $\partial B(x^{(1)}, \mathcal{R})$ is free of
  singularities.  Combining this fact with the second equation of (\ref{eqn: ksob0}),
  Lemma~\ref{lemma: samebubblesob} applies.
  There must be a singular point $p^{(1)}$
  inside $B(x^{(1)}, \mathcal{R})$ satisfying
  \begin{align}
     \lim_{i \to \infty} \mathbf{I}(B_{g_i^{(1)}(0)}(p_i^{(1)}, \mathcal{R}^{(1)} r_i^{(1)})) =0.
 \label{eqn: sob01}
 \end{align}
 where $p_i^{(1)},  r_i^{(1)}, \eta^{(1)}, \mathcal{R}^{(1)}$ are  chosen as in Lemma~\ref{lemma: samebubblesob}.

 Then we define $g_i^{(2)}(t)= (r_i^{(1)})^{-2} g_i((r_i^{(1)})^2t), \; x_i^{(2)}=p_i^{(1)}$
  and have weak compactness
  \begin{align*}
      (X_i, x_i^{(2)}, g_i^{(2)}(0)) \sconv (X^{(2)}, x^{(2)}, g^{(2)}).
  \end{align*}

 If $X^{(1)}$ doesn't have any singularity on $\partial B_{g^{(1)}}(x^{(1)}, 1)$,
 then the choice of $r_i^{(1)}$ assures that
 \begin{align*}
   \mathcal{E}(X^{(2)}) \leq \mathcal{E}(X^{(1)}) -\frac{\epsilon}{2}.
 \end{align*}
 However,  if there exist a singularity on $\partial B_{g^{(1)}}(x^{(1)},
 1)$, then $x^{(1)}$ must be a singularity.  So $X^{(1)}$ must contain at least two
 separate singularities. It follows that
 \begin{align*}
    \mathcal{E}(X^{(2)}) \leq \mathcal{E}(X^{(1)}) -\epsilon.
 \end{align*}
 Therefore, on the bubble $X^{(2)}$, we have
 \begin{align*}
 \left\{
   \begin{array}{ll}
    &\mathcal{E}(X^{(2)}) \leq \mathcal{E}(X^{(1)})
    -\frac{\epsilon}{2},\\
    &\displaystyle \lim_{i \to \infty} I(B_{g_i^{(2)}(0)}(x_i^{(2)}, \mathcal{R}^{(1)}))=0.
   \end{array}
 \right.
 \end{align*}

  Inductively, for every $k$, we can define $x_i^{(k)}, p_i^{(k)}, \eta^{(k)}, \mathcal{R}^{(k)}$ and
  bubble $X^{(k)}$.  The same argument yields that
 \begin{align}
 \left\{
   \begin{array}{ll}
    &\mathcal{E}(X^{(k)}) \leq \mathcal{E}(X^{(k-1)})
    -\frac{\epsilon}{2} \leq \cdots  \leq \mathcal{E}(X) - \frac{(k-1)\epsilon}{2},\\
    &\displaystyle \lim_{i \to \infty} \mathbf{I}(B_{g_i^{(k)}(0)}(x_i^{(k)}, \mathcal{R}^{(k-1)}))=0.
   \end{array}
 \right.
 \label{eqn: ksob}
 \end{align}
 However, this two equations can not hold simultaneously for large $k$.
 Because for a large $k$,
 we must have $\mathcal{E}(X^{(k)})<\frac{\epsilon}{2}$. So $X^{(k)}$ is free of singularities.
 It follows that
  \begin{align*}
    \lim_{i \to \infty} \mathbf{I}(B_{g_i^{(k)}(0)}(x_i^{(k)}, \mathcal{R}^{(k-1)}))
     = \mathbf{I}(B_{g^{(k)}}(x^{(k)}, \mathcal{R}^{(k-1)})) \geq \mathbf{I}(X^{(k)})>0.
  \end{align*}
  This contradicts to the second equation  of (\ref{eqn: ksob})!
  \end{proof}

\section{Space of Ricci Flows}

 In this section, we always assume $\{(X^m, g(t)), -1 \leq t \leq 1\}$
to be a Ricci flow solution on a closed manifold $X^m$. Moreover, it
satisfies the following properties.
\begin{enumerate}
 \item $\D{}{t}g(t) = -Ric_{g(t)} + c_0 g(t)$ where $c_0 $ is a
 constant satisfying $0 \leq c_0 \leq c$.
 \item $\displaystyle \sup_{X \times [-1, 1]}|R|_{g(t)} \leq \sigma$.
 \item  $\displaystyle \frac{\Vol_{g(t)}(B_{g(t)}(x, r))}{r^m} \geq \kappa$
    for all $x \in X, t\in [-1, 1], r \in (0, 1]$.
 \item  $\int_X |Rm|_{g(t)}^{\frac{m}{2}} d\mu_{g(t)} \leq E$ for all $t \in [-1,
 1]$.
 \end{enumerate}
 We denote $\mathscr{M}(m, c, \sigma, \kappa, E)$ as the moduli
 space of all such Ricci flow solutions.

 \subsection{Basic Properties of the Moduli Spaces}

 \begin{lemma}[\textbf{Energy Concentration}]
  There is a constant $H=H(m, c, \sigma, \kappa, E)$ such that the
  following property holds.

  If $\{(X, g(t)), -1 \leq t \leq 1\} \in \mathscr{M}(m, c, \sigma, \kappa, E)$,
  under the metric $g(0)$, we have
   \begin{align*}
      \int_{B(x, |Rm|^{-\frac12}(x))} |Rm|^2   d\mu > \epsilon
   \end{align*}
   whenever $|Rm|(x)>H$.
 \label{lemma: econcenter}
 \end{lemma}

 \begin{proof}
   Suppose not, then there is a sequence of Ricci flow solutions $\{(X_i^m, g_i(t)), -1 \leq t \leq 1\}$
 and scales $H_i \to \infty$ violating the result. In other words,
 there are points $x_i \in X_i$ such that
 \begin{align*}
  Q_i=|Rm|_{g_i(0)}(x_i)>H_i,  \quad \int_{B_{g_i(0)}(x_i,
  Q_i^{-\frac12})} |Rm|_{g_i(0)}^{\frac{m}{2}} d\mu_{g_i(0)} \leq \epsilon.
 \end{align*}
 Let $\tilde{g}_i(t) \triangleq Q_ig_i(Q_i^{-1}t)$, we can extract a
 refined sequence $\{(X_i, x_i, \tilde{g}_i(t)), -1 \leq t \leq 1\}$
 satisfying
 \begin{align*}
   |Rm|_{\tilde{g}_i(0)}(x_i)=1, \quad  \int_{B_{\tilde{g}_i(0)}(x_i,
   1)} |Rm|_{\tilde{g}_i(0)}^{\frac{m}{2}} d\mu_{\tilde{g}_i(0)} \leq \epsilon.
 \end{align*}
 This contradicts to Theorem~\ref{theorem: econ}!

 \end{proof}

\begin{lemma}[\textbf{Backward Pseudolocality}]
  There is a small constant $\delta=\delta(m, c,\sigma, \kappa, E)$ such that
 the following property holds.

  If
  $\{(X, g(t)), -1 \leq t \leq 1\} \in \mathscr{M}(m, c, \sigma, \kappa, E)$,
  $\rho \leq 1$, then
 \begin{align*}
  |Rm|_{g(t)}(y) < \delta^{-2} \rho^{-2}, \quad \forall \; y
  \in B_{g(t)}(x, \delta \rho),  t \in (-\delta^2 \rho^2, 0],
 \end{align*}
 whenever $|Rm|_{g(0)} \leq \rho^{-2}$ in $B_{g(0)}(x, \rho)$.
 \label{lemma: bpcenter}
 \end{lemma}

\begin{proof}
  Suppose not, then there is a sequence of $\delta_i \to 0$ and a
  sequence of such Ricci flow solutions $\{(X_i, g_i(t)), -1 \leq t \leq 1\}$
  violating this result. So there exists some small geodesic balls
  $B_{g_i(0)}(x_i, \rho_i)$ and points $(y_i,t_i)$ satisfying the
  following properties.
  \begin{align*}
     &\rho_i < 1,\\
     &|Rm|_{g_i(0)}< \rho_i^{-2} \; \textit{in} \;
     B_{g_i(0)}(x_i, \rho_i),\\
     &|Rm|_{g_i(t_i)} > \delta_i^{-2}\rho_i^{-2}, \quad
      y_i \in B_{g_i(t_i)}(x_i, \delta_i\rho_i), \quad t_i \in (-\delta_i^2\rho_i^2,
      0].
  \end{align*}
  Let $\tilde{g}(t) \triangleq \rho_i^{-2} \delta_i^{-1}g_i(\rho_i^2 \delta_i
  t)$.  The spacetime
  $\{(X_i, \tilde{g}_i(t)), -\rho_i^{-2} \delta_i^{-1} \leq t \leq \rho_i^{-2} \delta_i^{-1}\}$
  satisfies
  \begin{align*}
     & |Rm|_{\tilde{g}_i(0)}(x) < \delta_i \; \textit{in} \;
      B_{\tilde{g}_i(0)}(x_i, \delta_i^{-\frac12}),\\
     &  |Rm|_{\tilde{g}_i(t_i)}(y_i) > \delta_i^{-1}, \; y_i \in
     B_{\tilde{g}_i(t_i)}(x_i, \delta_i^{\frac12}), \; t_i \in (-\delta_i, 0].
  \end{align*}
  In particular, $\{(X_i, x_i, \tilde{g}_i(t)), -1 \leq t \leq 1\}$ is a
  refined sequence satisfying
   \begin{align*}
     & |Rm|_{\tilde{g}_i(0)}(x) < 1 \; \textit{in} \;
      B_{\tilde{g}_i(0)}(x_i, 1),\\
     &  |Rm|_{\tilde{g}_i(t_i)}(y_i) > \delta_i^{-1}, \; y_i \in
     B_{\tilde{g}_i(t_i)}(x_i, \delta_i^{\frac12}), \; t_i \in (-\delta_i, 0].
  \end{align*}
  for large $i$.
  This contradicts to Theorem~\ref{theorem: prermcontrol}!
\end{proof}

 \begin{lemma}[\textbf{Volume Ratio Bounds}]
  There is a constant $\eta= \eta(m, c, \sigma, \kappa, E)$ such
  that the following property holds.

  If $\{(X, g(t)), -1 \leq t \leq 1\} \in \mathscr{M}(m, c, \sigma, \kappa, E)$,
  then
  \begin{align*}
    \frac{\Vol(B_{g(0)}(x, r))}{r^m} < 2 \omega(m), \quad \forall \;
    x \in X,
  \end{align*}
  whenever $0 < r \leq \eta$.
 \label{lemma: vuppercenter}
 \end{lemma}

  \begin{proof}
  Suppose this result is wrong, then there is a sequence of $r_i \to 0$
  and a sequence of Ricci flow solutions $\{(X_i, g_i(t)), -1 \leq t \leq 1\}$
  violating this statement, i.e., there exists $x_i \in X_i$ such that
  \begin{align*}
     \frac{\Vol(B_{g_i(0)}(x_i, r_i))}{r_i^m} \geq 2 \omega(m).
  \end{align*}
  Let $\tilde{g}_i(t)= r_i^{-2}g_i(r_i^2 t)$, then $\{(X_i, \tilde{g}_i(t)), -1 \leq t \leq 1\}$
  is a refined sequence satisfying
  \begin{align*}
     \lim_{i \to \infty} \Vol_{\tilde{g}_i(0)}(B_{\tilde{g}_i(0)}(x_i, 1)) \geq 2\omega(m).
  \end{align*}
 This contradicts to Theorem~\ref{theorem: vupper}!
 \end{proof}

  Combining Lemma~\ref{lemma: econcenter},~\ref{lemma: bpcenter} and
  Lemma~\ref{lemma: vuppercenter}, we can follow the argument of
  Lemma~\ref{lemma: wcptev} to prove the following theorem.

 \begin{theorem}[\textbf{Weak Compactness}]
    If $\{ (X_i, x_i, g_i(t)) , -1 \leq t \leq 1\} \in \mathscr{M}(m, c, \sigma, \kappa, E)$
 for every $i$, by passing to subsequence, we have
 \begin{align*}
    (X_i, x_i, g_i(0)) \sconv (\hat{X}, \hat{x}, \hat{g})
 \end{align*}
 for some $C^0$-orbifold $\hat{X}$ in Cheeger-Gromov sense.
 If $m$ is odd, then $\hat{X}$ is a smooth manifold.
 \label{theorem: centerwcpt}
 \end{theorem}

 \begin{proof}
  We only need to prove the case when $m$ is odd.  It suffices to show
  an a priori Riemannian curvature bound on $X_i \times [-\frac18, 0]$.
  If there is no such a bound, by passing to subsequence if
  necessary, we can take a sequence of points $(x_i, t_i) \in X_i \times [-\frac18, 0]$
  such that $Q_i = |Rm|_{g_i(t_i)}(x_i) \to \infty$.  Let
  $\tilde{g}_i(t)= Q_i g_i(Q_i^{-1}t + t_i)$, we obtain an odd
  dimensional refined sequence $\{(X_i, x_i, \tilde{g}_i(t)), -1 \leq t \leq 1\}$
  satisfying  $|Rm|_{\tilde{g}_i(0)}(x_i)=1$. This contradicts to
  Theorem~\ref{theorem: oddrefined}!
 \end{proof}

  Every $C^0$-orbifold is smooth away from singularities, and it is
  very close to a Euclidean cone in $C^0$-topology around
  singularities. It's not hard to see that for every $C^0$-orbifold
  $\hat{X}$, there is a $C^{\infty}$-orbifold $\tilde{X}$ which is
  very close to $\hat{X}$ in $C^0$-topology. It follows that their
  isoperimetric constants are comparable. So every $C^0$-orbifold
  has a bounded isoperimetric constant. Then the same argument as in
  Theorem~\ref{theorem: refsob} yields the following result.

 \begin{theorem}[\textbf{Isoperimetric Constants}]
  There is a constant $\iota=\iota(m, c, \sigma, \kappa, E, D)$ such
  that the following property holds.

   If $\{(X, g(t)), -1 \leq t \leq 1\} \in \mathscr{M}(m, c, \sigma, \kappa, E)$
 and $\diam_{g(0)}(X)<D$, then
 \begin{align*}
   \mathbf{I}(X, g(0)) > \iota.
 \end{align*}
 \label{theorem: centersob}
 \end{theorem}

 \subsection{Weak Compactness of \KRF on Fano Surfaces}
  Along \KRf on Fano surfaces, diameter, scalar curvature and energy
  are all uniformly bounded by Perelman's estimates (c.f.~\cite{SeT}).
  Combining this with Perelman's no local collapsing theorem,
  we can find fixed numbers $E, \sigma, \kappa, D$ such that
  \begin{align*}
  \{(M, g(t-s)),  -1 \leq t-s \leq 1\} \in \mathscr{M}(4, 1, \sigma, \kappa,
  E),
  \quad
  \diam_{g(t)}(M) < D.
  \end{align*}
  Theorem~\ref{theorem: centersob} implies
  \begin{theorem}[\textbf{Isoperimetric Constants' Control of 2-dimensional KRF}]
   On every \KRf solution $\{(M^2, g(t)), 0 \leq t < \infty \}$,
   there is a constant $\iota>0$ such that
   \begin{align*}
      \mathbf{I}((M, g(t)))> \iota, \quad \forall \; t \geq 0.
   \end{align*}
   In particular, we have
   \begin{align*}
     C_S((M, g(t))) < C, \quad \forall \; t\geq 0
   \end{align*}
   for some constant $C>0$.
  \end{theorem}

  \begin{remark}
    In~\cite{Zhq} and~\cite{Ye}, Zhang and Ye proved that the
    Sobolev constants are uniformly bounded along K\"ahler Ricci flow.
    Here we improved their theorem by showing that
    isoperimetric constants
    are uniformly bounded in the special case $n=2$.
  \end{remark}

  As a corollary of Theorem~\ref{theorem: centerwcpt}, we have

  \begin{theorem}[\textbf{Weak Compactness of 2-dimensional KRF}]
    Suppose $\{(M^2, g(t)), 0 \leq t < \infty\}$ is a \KRf on Fano
    surface $M$. For every sequence $t_i \to \infty$, by passing to subsequence if necessary,
    we have
    \begin{align*}
       (M, g_i(t))   \sconv (\hat{M}, \hat{g})
    \end{align*}
    where $(\hat{M}, \hat{g})$ is a
    K\"ahler Ricci soliton $C^{\infty}$-orbifold.
  \end{theorem}

  \begin{proof}
   The weak compactness is automatic now.  So we only need to show
   that $(\hat{M}, \hat{g})$ satisfies K\"ahler Ricci soliton
   equation on smooth part and $\hat{M}$ is a $C^{\infty}$-orbifold.

  By the monotonicity of Perelman's $\mu$-functional, $(\hat{M}, \hat{g})$
  must satisfy K\"ahler Ricci soliton equation on the smooth part of $\hat{M}$.
  This has been proved by Natasa Sesum in~\cite{Se1} under the condition Ricci
  curvature uniformly bounded. However, the bound of Ricci curvature is only
  used to obtain the uniform Sobolev constant there. As the Sobolev
  constant bound is obtained now, that proof applies directly.

  Now we show that $\hat{M}$ is a $C^{\infty}$-orbifold.  Note that
  the $\mu$-functional minimizer $f$ on $\hat{M}$ satisfies
  \begin{align*}
      Ric_{\hat{g}} + \nabla_{\hat{g}}^2 f - \hat{g}=0,   \quad
      (2\pi)^{-2} \int_{\hat{M}} e^{-f} dv=1.
  \end{align*}
  It follows that $-f$ is the normalized Ricci potential of $(\hat{M},
  \hat{g})$.  It is the limit of the normalized Ricci potentials  $u_i$ of
  $(M, g(t_i))$. However, $u_i$'s $C^1$-norm is uniformly bounded by
  Perelman's estimate. Therefore, $f$ has a uniform $C^1$-norm. In
  particular, $|\nabla f|$ is uniformly bounded. Then we can use Uhlenbeck's
  removing singularity technique to show that
  every singularity is a $C^{\infty}$-orbifold point. This part was
  proved in~\cite{CS}.
  \end{proof}

 \subsection{Weak Compactness of Gradient Shrinking Solitons}

  Theorem~\ref{theorem: centerwcpt} can be used to study the moduli
  space of gradient shrinking Ricci solitons.  Define moduli spaces
  \begin{align*}
    &\mathcal{OS}(m, \sigma, \kappa, E, V)=\{(X^m, g)| (X^m, g) \; \textrm{is a compact orbifold satisfing condition} \;
    (*)\},\\
    &\mathcal{MS}(m, \sigma, \kappa, E, V)=\{(X^m, g)| (X^m, g) \; \textrm{is a closed manifold satisfing condition} \;
    (*)\}.
  \end{align*}
  The condition (*) is listed as follows.
  \begin{itemize}
  \item $Ric + \nabla^2 f - g=0$ for some $f \in C^{\infty}(X)$.
  \item $R \leq \sigma$.
  \item $\frac{\Vol(B(x, r))}{r^m} \geq \kappa$ for every $x \in X$ and
  $r \in (0,1)$.
  \item $\int_X |Rm|^{\frac{m}{2}}d\mu \leq E$.
  \item $\Vol(X) \leq V$.
  \end{itemize}

  As an application of Theorem~\ref{theorem: centerwcpt} and
  Theorem~\ref{theorem: centersob}, we obtain a weak compactness theorem of gradient
  shinking solitons.

  \begin{theorem}[\textbf{Weak Compactness of Gradient Shrinking Solitons}]
  Under the Cheeger-Gromov topology, we have
  \begin{enumerate}
  \item $\overline{\mathcal{MS}(m, \sigma, \kappa, E, V)} \subset\mathcal{OS}(m, \sigma, \kappa, E, V)$.
  \item $\mathcal{MS}(m, \sigma, \kappa, E, V)$ is a compact space
  when $m$ is odd.
  \end{enumerate}
  \label{theorem: solitonwcpt}
  \end{theorem}

  \begin{proof}
   The second part is trivial. So we only prove the first part.
   Since every $g$ can be looked as $g(0)$ of a Ricci flow solution
   \begin{align*}
    \{(X, g(t)), -1 \leq t \leq 1\},
    \quad \D{g}{t}= -Ric_{g(t)} + g(t),
   \end{align*}
   and scalar curvature of a gradient shrinking soliton is always
   positive, so Theorem~\ref{theorem: centerwcpt} applies and we know that $\hat{X}$
   is at worst a $C^0$-orbifold if $ \hat{X} \in \overline{\mathcal{MS}(m, \sigma, \kappa, E, V)}$.
   Now we only need to show that
   each $\hat{X}$ satisfies a gradient shrinking soliton equation and it is
   actually a $C^{\infty}$-orbifold.  Similar to Theorem~\ref{theoremin: krfcpt},
   it suffices to develop an a
   priori $C^1$-norm bound of all soliton potential functions $f$ on
   solitons in $\mathcal{MS}(m, \sigma, \kappa, E, V)$. We'll find
   this estimate in fours steps.\\

 \textit{Step1. \; Sobolev constants are uniformly bounded on $\mathcal{MS}(m, \sigma, \kappa, E, V)$}

     Since $g(t)$ is $\kappa$-noncollapsed on scale $1$,
   $\Vol_{g(t)}(X) \leq V$, it's easy to see that diameters
   must be uniformly bounded from above.  Then
   Theorem~\ref{theorem: centersob}
   applies and isoperimetric constants are uniformly bounded.
   In particular, there is a uniform Sobolev constant $C_S$ for
   every soliton in $\mathcal{MS}(m, \sigma, \kappa, E, V)$.\\

 \textit{Step2. \;
 Perelman's $\mu$-functionals are uniformly bounded on
  $\mathcal{MS}(m, \sigma, \kappa, E, V)$.}

    Consider Perelman's $\mu$-functional on $(X, g)$:
    \begin{align*}
      \mu \triangleq \mu(g, \frac12)=
      \inf_{\int_X e^{-h}dv =(2\pi)^{\frac{m}{2}}}
       \int_X \{\frac12 (R+|\nabla h|^2) + h-m\} e^{-h} (2\pi)^{-\frac{m}{2}} dv.
    \end{align*}

    Let $h$ be a constant such that
    $\int_X e^{-h}dv=(2\pi)^{\frac{m}{2}}$, i.e., $h \equiv -\frac{m}{2} \log (2\pi) + \log
    \Vol(X)$, we have
    \begin{align*}
       \mu &\leq \frac12 \sup_X R -m (1 + \frac12 \log(2\pi))+ \log \Vol(X)\\
       &\leq \frac12 \sigma -m (1 + \frac12 \log(2\pi))+ \log V.
    \end{align*}
    So $\mu$ is bounded from above.  According to an observation of Ecker (c.f. Lemma 8 of~\cite{CS}),
   we have
   \begin{align*}
      \mu &\geq -C(m)(1+ \log C_S(X,g) + \log \frac12)
      + \frac12 \inf_X R\\
   &\geq-C(m)(\log C_S + 1 -\log 2).
   \end{align*}
   Therefore, $\mu$ is uniformly bounded from below.  It follows
   that $\mu$ is a bounded function on
   $\mathcal{MS}(m, \sigma, \kappa, E, V)$.\\

 \textit{Step3. \; $\snorm{f}{C^0(X)}$ are uniformly bounded on
      $\mathcal{MS}(m, \sigma, \kappa, E, V)$.}

    As $(X, g)$ is a gradient shrinking soliton, we know the
    minimizer of $\mu$ is nothing but the soliton potential function
    $f$. So $f$ satisfies the Euler-Lagrange equation
    \begin{align}
       \frac12 (R+ 2 \triangle f - |\nabla f|^2) + f- m = \mu.
    \label{eqn: el}
    \end{align}
    This together with $R+ \triangle f -m=0$ implies
    \begin{align}
       \frac12 (R + |\nabla f|^2) - f = -\mu.
    \label{eqn: muexpress}
    \end{align}
    This tells us $|\nabla f|$ can be controlled by $f$'s value.
    So for any two points $p, q \in X$, we have control (c.f. Proposition 2.4 of~\cite{We}):
    \begin{align*}
         f(q) \leq 2 f(p) - \mu  + dist(p, q)^2 \leq 2 f(p) -\mu +
         \diam(X)^2.
    \end{align*}
    Therefore, we have
    \begin{align}
       \sup_X f \leq 2 \inf_X f -\mu + \diam(X)^2.
    \label{eqn: infsup}
    \end{align}

    Since $\int_X e^{-f} dv= (2 \pi)^{\frac{m}{2}}$
    and $\log \Vol(X)$ is uniformly bounded, we see that $\displaystyle inf_X f$
    is uniformly bounded from above for all
     $X \in \mathcal{MS}(m, \sigma, \kappa, E, V)$. By virtue of
     equation (\ref{eqn: infsup}) and the fact that $\diam(X)$ is uniformly bounded,
     we see that $\sup_X f$ is uniformly bounded from above.
     Thanks to equation (\ref{eqn: muexpress}), $\inf_X f$ is
     uniformly bounded from below. Therefore we see that $\snorm{f}{C^0(X)}$
     is uniformly bounded for all
     $X \in \mathcal{MS}(m, \sigma, \kappa, E, V)$.\\

 \textit{Step4.\; $\snorm{f}{C^1(X)}$ are uniformly bounded on
      $\mathcal{MS}(m, \sigma, \kappa, E, V)$.}

    By equation (\ref{eqn: muexpress}) and the fact $\mu$, $R$ and
    $f$ are all uniformly bounded, we see that $\snorm{\nabla f}{}$
    are uniformly bounded.  It follows that $\snorm{f}{C^1(X)}$ are uniformly bounded on
      $\mathcal{MS}(m, \sigma, \kappa, E, V)$.\\

   Therefore, as $(X_i, g_i)$ converges to $(\hat{X}, \hat{g})$,
   $f_i$ converges to $\hat{f}$. Moreover, at the smooth part of
   $\hat{X}$, we have
   \begin{align*}
       Ric_{\hat{g}} + \nabla^2 \hat{f} - \hat{g}=0.
   \end{align*}
   Moreover, $\snorm{\hat{f}}{C^1(\hat{X})}$ is bounded.
   Then using Ulenbeck's trick around the singularities of
   $\hat{X}$, we see that $(\hat{X}, \hat{g})$ is a
   $C^{\infty}$-orbifold satisfying gradient shrinking soliton equation.
  \end{proof}

 \begin{remark}
  If $m=4$, the soliton equation together with Gauss-Bonnett formula
  will imply (c.f. Proposition 2.3  of~\cite{We})
  \begin{align*}
    \int_X |Rm|^2 dv = 8\pi^2 \chi(X) + 2 \Vol(X) + \frac38 \int_X
    (R-\overline{R})^2 dv.
  \end{align*}
  where $\overline{R}$ is the average of scalar curvature $R$.
  Therefore, the condition of energy bound can be replaced by a
  condition of Euler characteristic number bound in
  Theorem~\ref{theorem: solitonwcpt}.
 \end{remark}

 \appendixpage
 \addappheadtotoc
 \appendix

 \section{Estimates of Diameters by Volumes}

 \begin{lemma}
    $(X^m,g)$ is a Riemannian manifold satisfying
 \begin{align*}
     \frac{\Vol(B(p, r))}{r^m} \geq \kappa, \quad \; \forall \; x
     \in X^m, \; r \in (0, 1].
 \end{align*}
   $B \subset X$ is a connected open set and $\Vol(B) < \kappa$.
   $S_1, \cdots,  S_{\Lambda}$ are connected components of $\partial B$. Then
 \begin{align*}
  \diam B \leq \sum_{k=1}^{{\Lambda}} \diam S_k  +  3{\Lambda} ( \frac{Vol(B)}{\kappa})^{\frac{1}{m}}
 \end{align*}
 where $\diam S_k$ means the diameter of the manifold $(S_k, g|_{S_k})$.
  \label{lemma: cvcd}
 \end{lemma}

 \begin{proof}

  Choose any two points $x_1, x_2 \in B$, there are points $y_1, y_2 \in \partial
  B$ such that
\begin{align*}
  d(x_1, y_1) = d(x_1, \partial B), \quad   d(x_2, y_2) = d(x_2, \partial
  B).
\end{align*}
For simplicity, we denote $d_1=d(x_1,y_1), d_2=d(x_2,y_2)$. The
triangle inequality reads
\begin{align}
d(x_1, x_2) \leq d(x_1, y_1) + d(y_1, y_2) + d(y_2, x_2)
 = d_1 + d_2 + d(y_1, y_2). \label{eqn: drough}
\end{align}

Note that $d_1=d(x_1, \partial B)$, so $B(x_1, d_1) \subset B$. It
follows that
\begin{align*}
  \Vol(B(x_1, d_1)) \leq \Vol(B) < \kappa.
\end{align*}
 The $\kappa$-noncollapsing condition forces $d_1<1$.
 Moreover, it assures that
\begin{align*}
     \frac{\Vol (B(x_1, d_1))}{d_1^m}  \geq \kappa.
\end{align*}
Therefore,
\begin{align*}
      d_1 \leq (\frac{\Vol(B(x_1,d_1))}{\kappa})^{\frac{1}{m}} \leq
      (\frac{\Vol(B)}{\kappa})^{\frac{1}{m}}.
\end{align*}
Similarly,
\begin{align*}
      d_2 \leq (\frac{\Vol(B(x_2,d_2))}{\kappa})^{\frac{1}{m}} \leq
      (\frac{\Vol(B)}{\kappa})^{\frac{1}{m}}.
\end{align*}
Put them into equation (\ref{eqn: drough}), we have
 \begin{align}
  d(x_1, x_2) \leq 2(\frac{\Vol(B)}{\kappa})^{\frac{1}{m}} + d(y_1, y_2).
  \label{eqn: drough2}
 \end{align}

 So we only need to estimate $d(y_1, y_2)$.   There are two cases.
 $y_1, y_2$ are in a same connected component, $y_1, y_2$ are in
 different components.

 \textit{Case1.\;  $y_1, y_2$ are in a same component.}

  For simplicity, we assume $y_1, y_2 \in S_1$.  Clearly, $d(y_1, y_2) \leq \diam S_1$.
  It follows from inequality (\ref{eqn: drough2}) that
  \begin{align*}
    \diam B = d(x_1, x_2) \leq 2(\frac{\Vol(B)}{\kappa})^{\frac{1}{m}} +
    \diam S_1
    < \sum_{k=1}^{{\Lambda}} \diam S_k
    +  6{\Lambda} (\frac{Vol(B)}{\kappa})^{\frac{1}{m}}.
  \end{align*}
  So we finish the proof.

 \textit{Case2. $y_1, y_2$ are in different components.}

   For simplicity, we assume $y_1 \in S_1, \; y_2 \in S_2$.

   Suppose $\gamma$ is the shortest geodesic connecting $S_1$ and
   $S_2$ among all geodesics in $\bar{B}$. Let $L$ be the length of
   $\gamma$. Therefore $\gamma(0) \in S_1, \gamma(L) \in S_2$.
  If $L < 2( \frac{Vol(B)}{\kappa})^{\frac{1}{m}}$, we can finish
  the proof by inequality (\ref{eqn: drough2}) directly. So we
  assume that $L > 2( \frac{Vol(B)}{\kappa})^{\frac{1}{m}}$.
  Define
 \begin{align*}
    D &\triangleq 2( \frac{Vol(B)}{\kappa})^{\frac{1}{m}}, \\
    I_k &\triangleq  \inf \{ t| d(S_k, \gamma(t)) \leq  D\}, \\
    E_k &\triangleq  \sup \{ t| d(S_k, \gamma(t)) \leq  D\}.
 \end{align*}
 Clearly, $I_1=0, E_1=D; \; I_2=L-D, E_2=L$. Generally, triangle
 inequality implies $E_k - I_k \leq 2D + \diam(S_k)$.

  Since $\displaystyle [0,L] \backslash \bigcup_{k=1}^{{\Lambda}} [I_k, E_k]
  = [D, L-D] \backslash \bigcup_{k=3}^{{\Lambda}} [I_k, E_k] $, it has at most ${\Lambda}-1$
  connected components.  Moreover, its length is no less than
 \begin{align*}
    L-2D-(2D+\diam S_3)-\cdots -(2D + \diam S_{{\Lambda}}) \\
    =  L-2({\Lambda}-1)D-(\diam S_3  + \cdots \diam S_{\Lambda}).
 \end{align*}
 Suppose $T=(a, b)$ is one connected component of
 $\displaystyle [0,L] \backslash \bigcup_{k=1}^{{\Lambda}} [I_k, E_k]$.  Consider the geodesic ball
 with center $\gamma(\frac{a+b}{2})$ and radius $\frac{b-a}{2}$.
 Clearly, $B(\gamma(\frac{a+b}{2}), \frac{b-a}{2})$ locates totally
 inside $B$.  By the volume $\kappa$-noncollapsing condition, we
 have
 \begin{align*}
     |T|=b-a < 2 (\frac{Vol(B)}{\kappa})^{\frac{1}{m}}=D.
 \end{align*}
Since every component has length controlled by $D$, and the number
of components is controlled by $\Lambda-1$, we obtain
 \begin{align*}
    L-2({\Lambda}-1)D-(\diam S_3  + \cdots \diam S_{\Lambda}) <
    \left|[0,L] \backslash \bigcup_{k=1}^{{\Lambda}} [I_k, E_k] \right| \leq   ({\Lambda}-1)D.
 \end{align*}
In other words,
 \begin{align*}
   L < 3({\Lambda}-1)D + (\diam S_3  + \cdots \diam S_{\Lambda}).
 \end{align*}
 Consequently,  triangle inequality implies
 \begin{align*}
    d(y_1, y_2) &< \diam(S_1) + \diam(S_2) + L\\
           &< 3({\Lambda}-1)D + (\diam S_1 + \cdots \diam S_{\Lambda}).
 \end{align*}
 Plugging it into inequality (\ref{eqn: drough2}) yields
 \begin{align*}
  \diam B & = d(x_1, x_2) \leq 2(\frac{\Vol(B)}{\kappa})^{\frac{1}{m}}
       +3({\Lambda}-1)D + (\diam S_1 + \cdots \diam S_{\Lambda})\\
       &=(3\Lambda -2)D
         + (\diam S_1 + \cdots \diam S_{\Lambda})\\
       &< 6\Lambda (\frac{\Vol(B)}{\kappa})^{\frac{1}{m}} +
       \sum_{k=1}^{\Lambda} \diam S_k.
 \end{align*}
 So we finish the proof.
 \end{proof}

\section{Estimates on Euclidean Annulus}

 \begin{lemma}
  Suppose $A(2,1)=B(o, 2) \backslash \overline{B(o, 1)}$ is the standard
  open annulus in $\R^m$.  Let
  \begin{align*}
    c(m)=\min\{\frac{1}{1+ \frac23\frac{(m\omega(m))^{\frac{1}{m-1}}}{I_S}},\; l(m)\}
  \end{align*}
  where $\omega(m)$ is the volume of $B(o, 1)$, $I_S$ is the
  isoperimetric constant of standard sphere $\partial B(o, 1)$,
  $l(m)$ is the positive solution of equation
  $\frac{2x(1+x)}{3(1-x)} = \frac12 m\omega(m)$.  Then the following
  property holds.

   For every domain $\Omega \subset A(2,1)$ satisfying $\Vol(\Omega)< c(m)$, we
  have
  \begin{align*}
     \frac{\Area(\partial \Omega \cap A(2,1))}{\Area(\partial \Omega \cap \partial
     B(o,1))} \geq c(m).
  \end{align*}
 \label{lemma: esob}
 \end{lemma}

\begin{proof}
  For simplicity of notation, define $S_r \triangleq \partial B(o, r)$.  There is a natural projection map  $\pi$ defined as follows.
 \begin{align*}
   \pi:  A(2,1)  \mapsto  S_1,  \quad    \vec{x} \mapsto \frac{\vec{x}}{\snorm{x}{}}.
 \end{align*}
 Clearly, this projection is area decreasing. To be precise, for
 every hyper surface $H^{m-1} \subset A(2, 1)$,  we have $\Area(\pi(H)) \leq
 \Area(H)$. It follows that
 \begin{align*}
   &\quad \left|\frac{\Area(\bar{\Omega} \cap S_s)}{s} - \frac{\Area(\bar{\Omega} \cap
   S_t)}{t} \right|\\
   &=\left| \Area(\pi(\bar{\Omega} \cap S_s)) -\Area(\pi(\bar{\Omega} \cap
   S_t))\right|\\
    &\leq \Area(\pi(\bar{\Omega} \cap S_s) \backslash \pi(\bar{\Omega} \cap S_t)) +
       \Area(\pi(\bar{\Omega} \cap S_t) \backslash \pi(\bar{\Omega} \cap S_r))\\
    &\leq \Area(\partial \Omega \cap A(s, t))
 \end{align*}
 for every $1 \leq t < s < 2$.  When $t=1$, $\bar{\Omega} \cap S_1 = \partial \Omega \cap S_1$.
 It follows that
 \begin{align}
 -\Area(\partial \Omega \cap A(s, 1)) +\Area(\partial{\Omega} \cap  S_1)
 &\leq  \frac{\Area(\bar{\Omega} \cap S_s)}{s} \notag \\
 &\leq
 \Area(\partial \Omega \cap A(s, 1)) + \Area(\partial{\Omega} \cap S_1).
 \label{eqn: ainout}
 \end{align}

  Now we'll prove this Lemma by contradiction. Suppose this Lemma is
  wrong, then there is  a constant $c<c(m)$ and a domain $\Omega$
  such that
  \begin{align*}
   \Area(\partial \Omega \cap A(2, 1)) \leq c \Area(\partial{\Omega} \cap S_1), \quad
   \Vol(\Omega) \leq c.
  \end{align*}
 Put this into inequality (\ref{eqn: ainout}), we obtain
 \begin{align}
    s(1-c)\Area(\partial{\Omega} \cap S_1)
    \leq
    \Area(\bar{\Omega} \cap S_s) \leq s(1+c)\Area(\partial{\Omega} \cap S_1).
 \label{eqn: areatwocontrol}
 \end{align}
 Integrating for $s$ on interval $(1,2)$ gives us
 \begin{align}
   \frac32(1-c)\Area(\partial{\Omega} \cap S_1) \leq \Vol(\Omega) \leq
   \frac32(1+c)\Area(\partial{\Omega} \cap S_1).
 \label{eqn: va}
 \end{align}
 Together with $\Vol(\Omega) \leq c$, the left inequality yields that
 \begin{align*}
   \Area(\partial{\Omega} \cap S_1) \leq \frac{2c}{3(1-c)}.
 \end{align*}
 Putting this back to the right inequality of
 (\ref{eqn: areatwocontrol}) implies that
 \begin{align*}
   \Area(\bar{\Omega} \cap S_s) \leq \frac{2c(1+c)}{3(1-c)}s \leq
   \frac12 \Area(S_s) = \frac12 m \omega(m) s \leq m \omega(m).
 \end{align*}
 Here we used the condition $c < c(m) \leq l(m)$, the positive solution of
 $\frac{2x(1+x)}{3(1-x)} = \frac12 m\omega(m)$.
 It follows that
 \begin{align*}
              \Length(\partial \Omega \cap S_s) &\geq
              I_S \cdot \Area(\bar{\Omega} \cap S_s)
              \cdot \frac{1}{\Area(\bar{\Omega} \cap
              S_s)^{\frac{1}{m-1}}}\\
    &\geq \frac{I_S}{(m\omega(m))^{\frac{1}{m-1}}} \Area(\Omega \cap S_s)
 \end{align*}
 Here we use ``$\Length$" to denote the $(m-2)$-dimensional Hausdorff
 measure. If $\partial \Omega \cap S_s$ happens to be a hyper surface,
 then  $ \Length(\partial \Omega \cap S_s)=\infty$.
 Integrate on both sides for $s$ in $(1, 2)$, we obtain
 \begin{align*}
   \Area(\partial \Omega \cap A(2,1))
   &\geq \int_{s=1}^2 \Length(\partial \Omega \cap S_s)\\
   &\geq \frac{I_S}{(m\omega(m))^{\frac{1}{m-1}}} \Vol(\Omega).
 \end{align*}
 Combining this with our assumption
 $\frac{\Area(\partial \Omega \cap A(2, 1))}{\Area(\partial{\Omega} \cap S_1)} \leq
 c$  and left inequality of (\ref{eqn: va}), we obtain
 \begin{align*}
   c  \geq  \frac32 (1-c) \frac{I_S}{(m\omega(m))^{\frac{1}{m-1}}}.
 \end{align*}
 It follows that
 \begin{align*}
     c \geq
     \frac{1}{1+
     \frac23\frac{(m\omega(m))^{\frac{1}{m-1}}}{I_S}}>c(m).
 \end{align*}
 This contradicts to our assumption of $c$!
 \end{proof}

 By lifting to covering, we obtain the following property directly.
 \begin{lemma}
  The conclusion of Lemma~\ref{lemma: esob} still holds if we
  replace  $A(2,1)$ by $C_{2,1}(S^{m-1} / \Gamma)$ which is a
  corresponding open annulus in the flat cone over $S^{m-1} / \Gamma$.
  Here $\Gamma$ is a finite group of $SO(m-1)$.
 \label{lemma: ecsob}
 \end{lemma}

  \vspace{1in}

 Xiuxiong Chen,  Department of Mathematics, University of
 Wisconsin-Madison, Madison, WI 53706, USA; xiu@math.wisc.edu\\

 Bing  Wang, Department of Mathematics, University of Wisconsin-Madison,
 Madison, WI, 53706, USA; bwang@math.wisc.edu

 \qquad \qquad \quad \; Department of Mathematics, Princeton University,
  Princeton, NJ 08544, USA; bingw@math.princeton.edu

  \end{document}